\documentclass[final]{siamltex}
\usepackage{amsmath}
\usepackage[margin=1.65in]{geometry}
\usepackage{afterpage}
\usepackage[dvips]{graphicx}
\usepackage[normal]{subfigure}
\usepackage{amssymb}
\usepackage{verbatim}
\usepackage[usenames,dvipsnames]{color}
\usepackage{subfig,float}
\usepackage{mathrsfs}
\usepackage{colonequals}
\newcommand{\bsub}{\begin{subequations}}
\newcommand{\esub}{\end{subequations}$\!$}

\newcommand{\ds}[0]{\displaystyle}

\newcommand{\fr}{\frac}
\newcommand{\eps}{\varepsilon}

\newcommand{\emdash}{\hspace{1pt}---\hspace{1pt}}

\newcommand{\bex}{\begin{example}\rm}
\newcommand{\eex}{\end{example}}

\def\be{\begin{equation}}
\def\ee{\end{equation}}
\def\ba{\begin{align}}
\def\ea{\end{align}}

\newtheorem{thm}{Theorem}[section]
\newtheorem{conj}[thm]{Conjecture}

\newtheorem{lem}[thm]{Lemma}

\newcounter{coronumber}[thm]
\numberwithin{coronumber}{thm} 
\newtheorem{coro}[coronumber]{Corollary}
\newtheorem{example}{Example}
\newtheorem{pr}[thm]{Principle Result}

\newcommand{\R}{\mathbb{R}}

\newcommand{\til}[1]{\tilde{#1}}
\newcommand{\calig}[1]{\mathscr{#1}}

\newcommand{\ud}{\ \mathrm{d}}
\newcommand{\pd}[2]{\frac{\partial #1}{\partial #2}}

\newcommand{\D}[2]{\frac{\mathrm{d} #1}{\mathrm{d} #2}}

\newcommand{\bigoh}{\mathcal{O}}
\newcommand{\littleoh}{ \mbox{{\scriptsize $\mathcal{O}$}}}
\newcommand{\ce}{\colonequals}

\title{Analysis of the singular solution branch of a prescribed mean curvature equation with singular nonlinearity modeling a MEMS capacitor}

\author{Nicholas D. Brubaker\thanks{Department of Mathematical Sciences, University of Delaware, Newark, Delaware, 19716, USA. ({\tt brubaker@math.udel.edu}). The work of this author was supported by the National Science Foundation through a Graduate Research Fellowship.} \and Alan E. Lindsay\thanks{Department of Mathematics, University of Arizona, Tucson, Arizona, 85721, USA. ({\tt alindsay@math.arizona.edu}).}}

\begin{document}

\label{firstpage}
\maketitle

\begin{abstract}
 The existence and multiplicity of solutions to a quasilinear, elliptic partial differential equation (PDE) with singular non-linearity is analyzed. The PDE is a recently derived variant of a canonical model used in the modeling of Micro-Electro Mechanical Systems (MEMS). It is observed that the bifurcation curve of solutions terminates at single \emph{dead-end} point, beyond which no classical solutions exist. A necessary condition for the existence of solutions is developed which reveals that this dead-end point corresponds to a blow-up in the solution derivative at a point internal to the domain. By employing a novel asymptotic analysis in terms of a pair of small parameters, an accurate prediction of this dead-end point is obtained. An arc-length parameterization of the solution curve can be employed to continue solutions beyond the dead-end point, however, all extra solutions are found to be multivalued. This analysis therefore suggests the dead-end is a bifurcation point associated with the onset of multivalued solutions for the system.
 \end{abstract}
 
 \begin{keywords} 
prescribed mean curvature, disappearing solutions, singular perturbation, MEMS, singular nonlinearity
\end{keywords}

\begin{AMS}
35J93, 35P30, 34B15, 34C23, 74K15
\end{AMS}

\pagestyle{myheadings}
\thispagestyle{plain}
\markboth{N.\:D. BRUBAKER AND A.\:E. LINDSAY}{ANALYSIS OF PRESCRIBED MEAN CURVATURE EQUATION}

\section{Introduction}

A micro-electromechanical systems (MEMS) capacitor consists of two surfaces held opposite of one another. The lower surface is a rigid inelastic ground plate while the upper surface is a thin elastic membrane held fixed along its boundary and free to deflect in the presence of a potential difference $V$ (c.f. Fig.~\ref{fig:2dsetup}). When $V$ is small enough, a stable equilibrium deflection is attained by the deflecting membrane; however, if $V$ exceeds a critical value $V^{\ast}$, called the \emph{pull-in} voltage, an equilibrium  deflection is no longer attainable and the upper surface will \emph{touchdown} on the lower. This loss of a stable equilibrium is called the \emph{pull-in instability} and the mathematical modeling of its onset has been the focus of many recent studies (c.f. \cite{esposito2010memsbook,pelesko2003modeling} and the references therein for a thorough account).

\begin{figure}[h]
\centering
\includegraphics[width=0.75\textwidth]{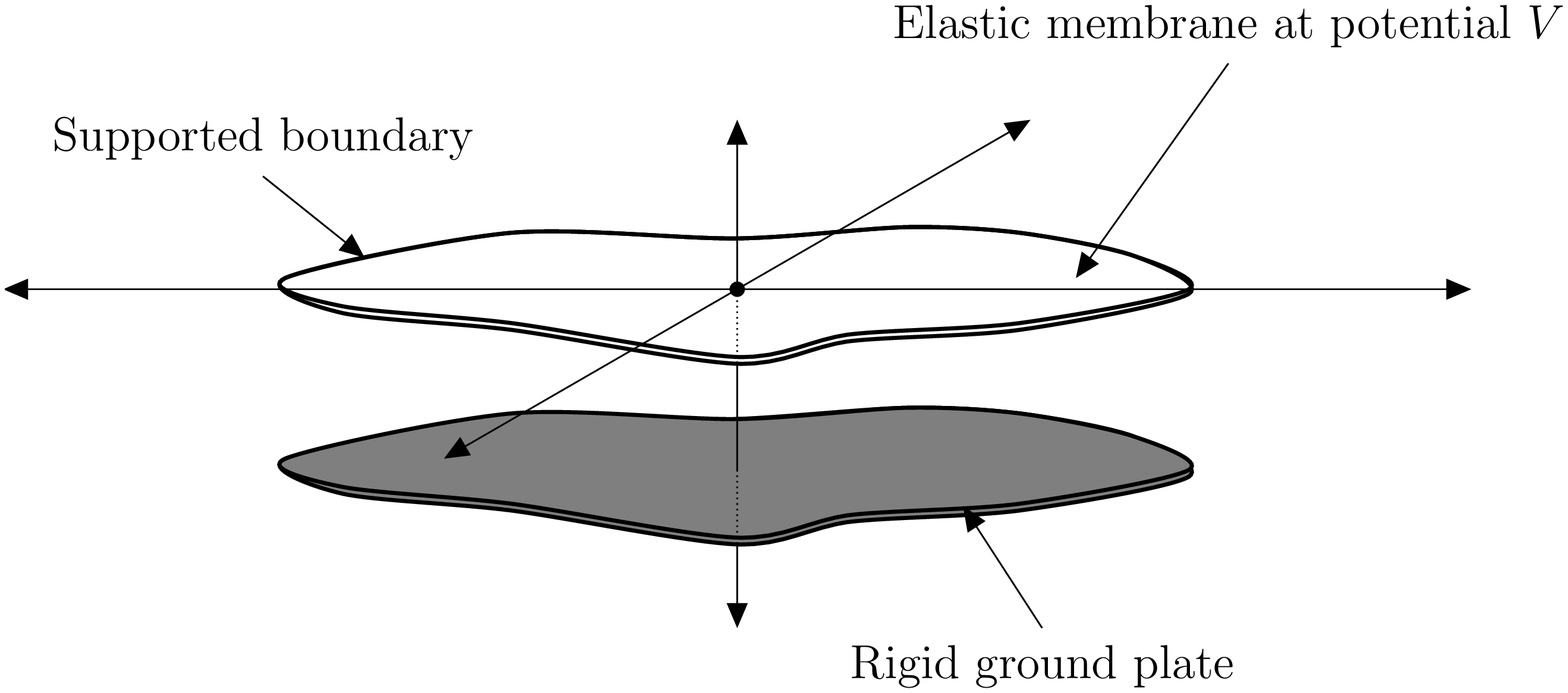}
\parbox{4.5in}{\caption{Schematic diagram of a MEMS capacitor.}\label{fig:2dsetup}}
\end{figure}

In a recent study of the physical approximations made in models of a MEMS capacitor (c.f. \cite{brubaker2011nonlinear}), the following quasilinear, elliptic partial differential equation (PDE) for the dimensionless equilibrium deflection was derived:
\be\label{eq:govpde}
  \mbox{div } \frac{\nabla u }{\sqrt{1 + \eps^2 |\nabla u|^2}} = \frac{\lambda}{(1+u)^2}, \quad x\in \Omega; \qquad  u=0,\quad x\in \partial\Omega.
\ee
Here, $0<\eps\ll 1$ is the aspect ratio of the device, $\lambda\propto V^2$ is a nonnegative dimensionless parameter quantifying the relative strengths of the elastic and electrostatic forces in the system and $\Omega$ is a bounded region in $\mathbb{R}^n$. The physically relevant dimensions{\emdash}$n=1,2${\emdash}are the focus of the present work.

In typical applications, the aspect ratio $\eps$ is a small quantity and so many MEMS researchers simplify \eqref{eq:govpde} by linearizing the scaled mean curvature operator on the left-hand side, thus yielding the equation 
\be\label{eq:stdpde}
  \Delta u = \frac{\lambda}{(1+u)^2}, \quad x\in \Omega; \qquad  u=0,\quad x\in \partial\Omega.
\ee
This reduced equation has been extensively studied and many of its properties are well known (c.f. \cite{esposito2010memsbook,wardguopan,pelesko2003modeling} and the references therein). One of the canonical properties is the existence of a critical value, $\lambda^*$, such that  for each $\lambda <  \lambda^*$, \eqref{eq:stdpde} admits a unique stable solution. At the end of this branch of stable solutions there is a saddle node bifurcation, and accordingly no solutions of \eqref{eq:stdpde} exist for  $\lambda > \lambda^*$. In the case $n=1$, \eqref{eq:stdpde} has exactly two solutions for each $\lambda < \lambda^*$ and a unique regular solution at $\lambda = \lambda^*$.  In the case $n=2$, the unstable solution branch undergoes infinitely many additional saddle node bifurcations which leads to higher multiplicity in the solution set (see the bifurcation diagrams given in Figure \ref{fig:stdbds}, where $\Omega$ is taken to be $[-1,1]$ and the two-dimensional unit disk for \subref{fig:stdbds1d} and \subref{fig:stdbds2d}, respectively).

\begin{figure}[h]
\centering
\subfigure[$n=1,\  \Omega=\{x \in \R : |x|\leq1\}$]{\label{fig:stdbds1d}\includegraphics[width=0.425\textwidth]{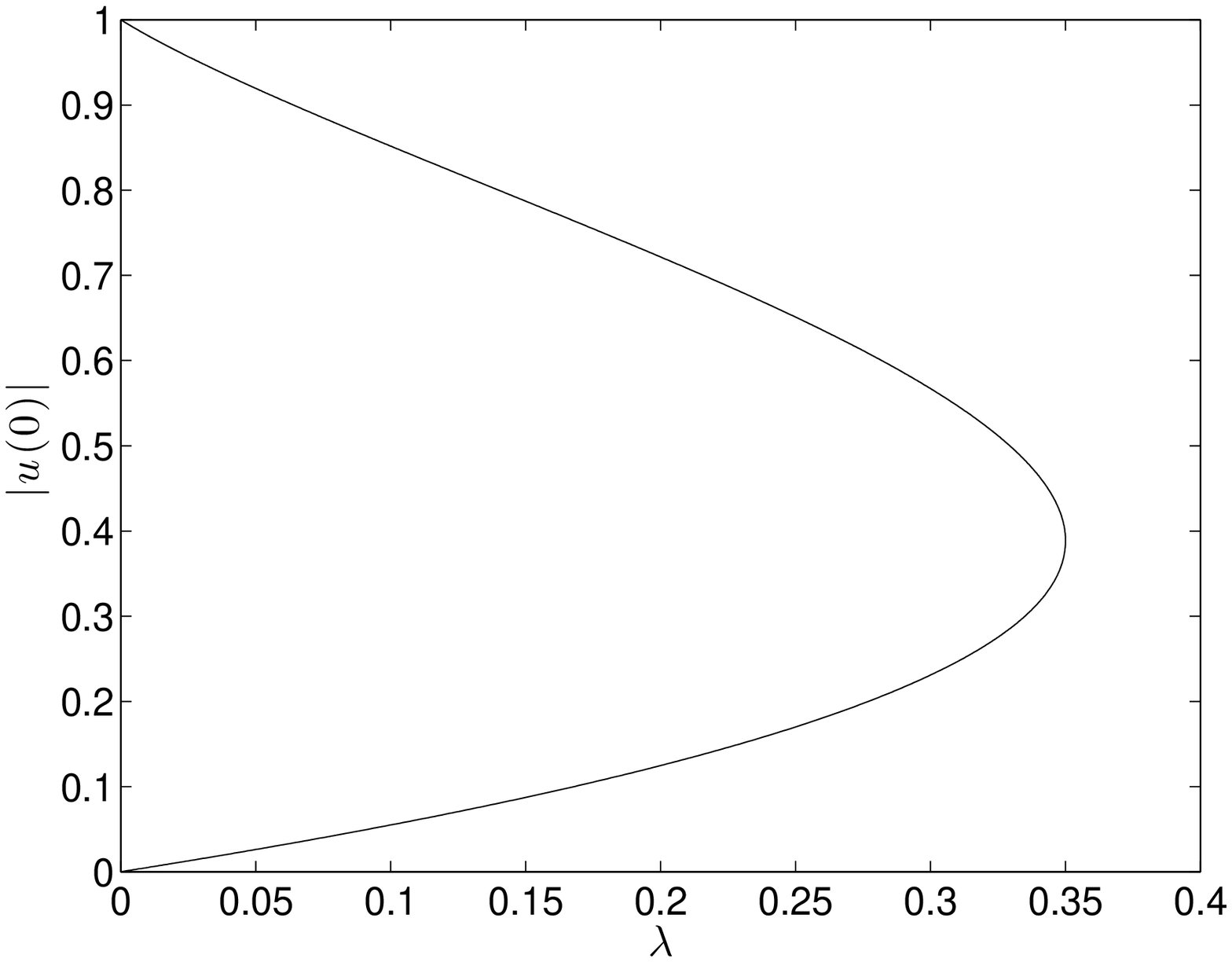}}\qquad
\subfigure[$n=2,\  \Omega= \{x \in \R^2 : |x|\leq1\}$]{\label{fig:stdbds2d}\includegraphics[width=0.425\textwidth]{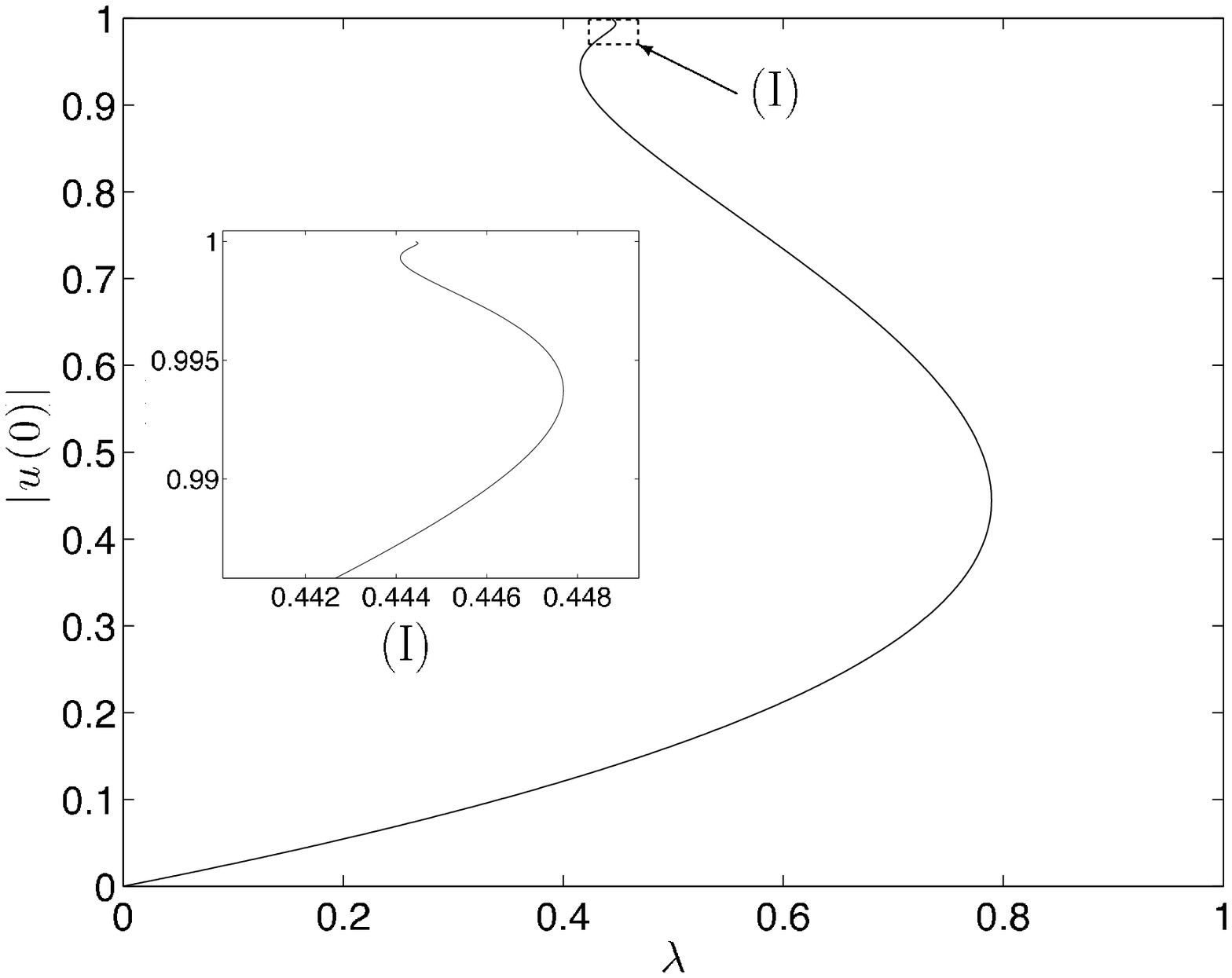}}
\caption{ Bifurcation diagrams of \eqref{eq:stdpde}. In \subref{fig:stdbds1d}, exactly two solutions are present for each $\lambda<\lambda^{\ast}$. In \subref{fig:stdbds2d}, the bifurcation curve undergoes an infinite number of fold point and hence has a infinite number of solutions for ranges of $\lambda$.\label{fig:stdbds}}
\end{figure}

However, when $\eps \neq 0$ the solution set $(\lambda,u)$ of \eqref{eq:govpde} can be markedly different from that of \eqref{eq:stdpde}. The particular focus of this paper is to catalog and analyze some of the profound differences between the solution structure of \eqref{eq:govpde} and \eqref{eq:stdpde} in the singular limit $\|u\|_{\infty}\to 1$. Let us begin by remarking on some differences in the $n=1$ and $n=2$ cases observed from previous studies (c.f. \cite{brubaker2011analysis,brubaker2011nonlinear}).

When $n=1$, it was shown in \cite{brubaker2011analysis} that there exists a critical $\eps^* \approx 0.35$ for which the bifurcation curve of \eqref{eq:govpde} deviates from the qualitative shape given in Fig.~\ref{fig:stdbds1d}. In particular, if $\eps > \eps^*$, then there exists two values $\lambda_{*}(\eps)$ and $\lambda_{**}(\eps)$ such that whenever $\lambda\in (\lambda_{*} ,\lambda_{**})$, the stable minimal solution branch is the only solution of \eqref{eq:stdpde} (see Figure \ref{fig:mcbds1d}). 
\begin{figure}[h]
\centering
\subfigure[$\eps \leq \eps^*$]{\label{fig:mcbds1dleq}\includegraphics[width=0.425\textwidth]{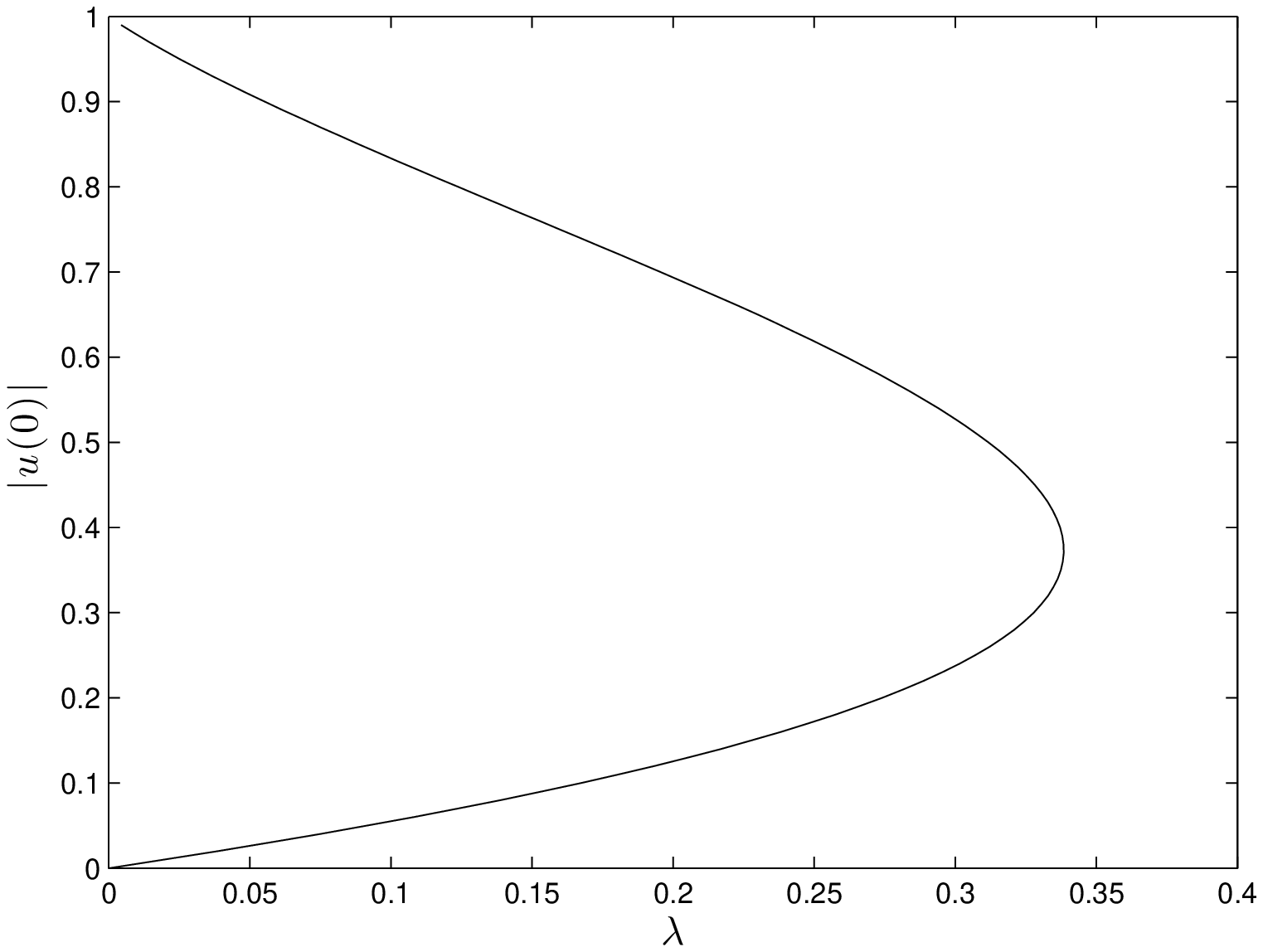}}\qquad
\subfigure[$\eps > \eps^*$]{\label{fig:mcbds1dg}\includegraphics[width=0.425\textwidth]{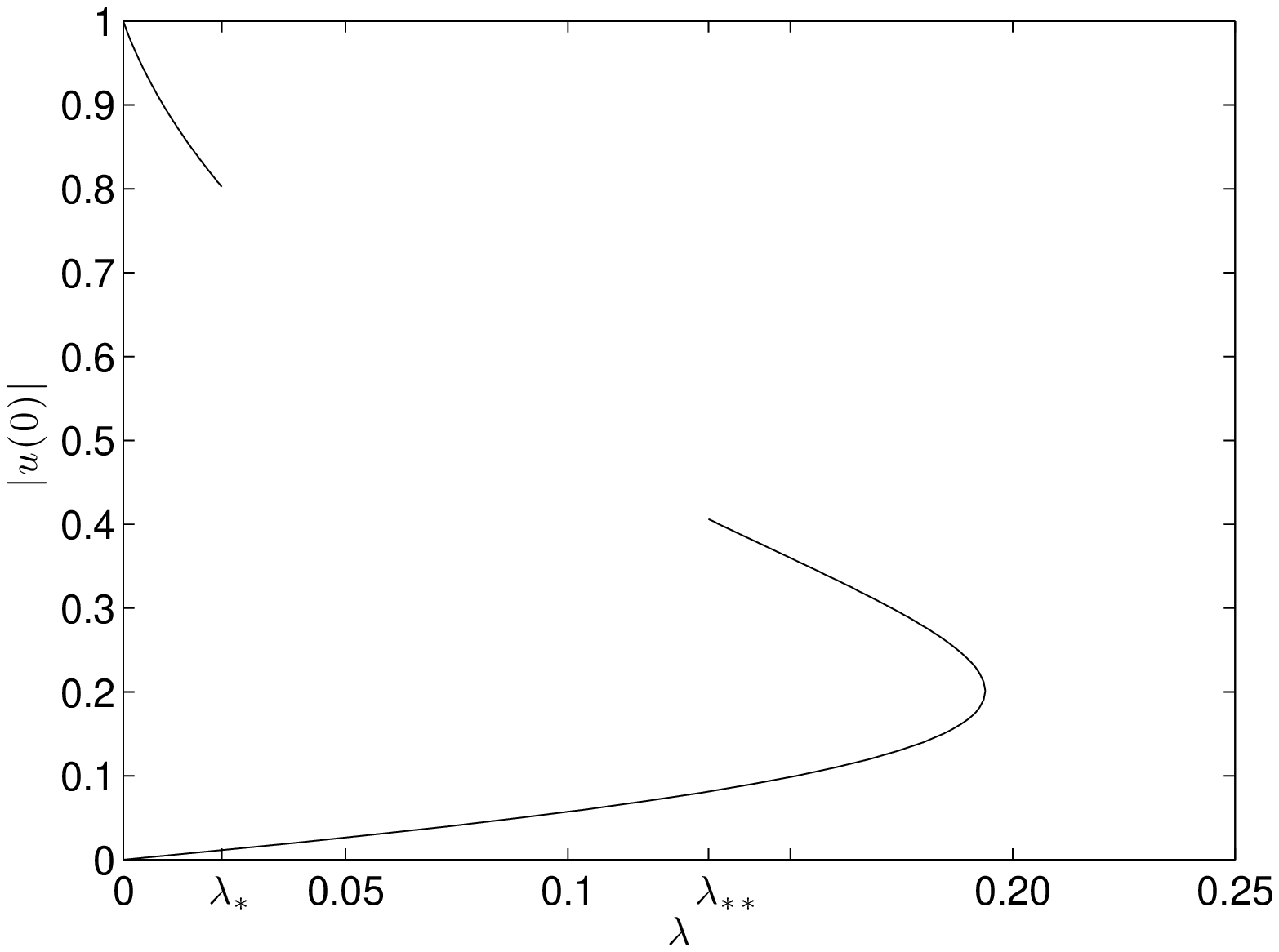}}
\caption{ Bifurcation diagrams of \eqref{eq:govpde} for $n=1$ and $\Omega=\{x \in \R : |x|\leq1\}$ and various $\eps$: \subref{fig:mcbds1dleq} for $\eps=1/2$; \subref{fig:mcbds1dg} for $\eps = 10/3$.\label{fig:mcbds1d}}
\end{figure}
In the case where $\Omega \subset \R^2$ is the unit disk and $u$ is radially symmetric,  
\eqref{eq:govpde} reduces to the ordinary differential equation
\be\label{eq:gov2dradpde}
\begin{aligned}
  \frac{1}{r}\left( \frac{r u' }{\sqrt{1 + \eps^2 (u')^2}}\right)' = \frac{\lambda}{(1+u)^2},   \quad 0<r<1; \quad   u'(0) =u(1)=0,
  \end{aligned}
\ee
where $u(r)\in (-1,0]$ for $0<r<1$. Numerical simulations have indicated the presence of a similar \lq\lq disappearance of solutions\rq\rq~phenomenon, more specifically the bifurcation curve{\emdash}$(\lambda(\eps), |u(0)|)$, where $|u(0)| \in [0,1)${\emdash }undergoes a finite number of folds before terminating at a single \emph{dead-end} point, denoted $(\lambda_*(\eps),|\alpha_*(\eps)|)$ (see Figures \ref{fig:mcbds2d1}--\subref{fig:mcbds2dzoom}). Numerics show that as the bifurcation curve  approaches $(\lambda_*,|\alpha_*|)$, the derivative of the solution becomes unbounded at some internal point, suggesting that multivalued solutions may be continued beyond the dead-end point. 

Therefore to further study \eqref{eq:gov2dradpde}, a parametrization $(r,u(r)) = (r(s),z(s))$ was employed (c.f. \cite{brubaker2011nonlinear}) in terms of an arc length $s$ along the solution curve. The functions $r(s), z(s)$ satisfy the system of ordinary differential equations:
\be\label{eq:arclengthsystem}
 \begin{alignedat}{1}
  &r''  = -\frac{\eps^2 \lambda z'}{(1+z)^2} + \frac{\eps^2 (z')^2}{r} ,  \quad z'' = \frac{\lambda r'}{(1+z)^2}-\frac{r'z'}{r},   \qquad 0<s<\ell, \\
  & r(0) = 0, \quad r'(0) =1, \quad r(\ell) = 1, \quad z'(0) = 0, \quad z(\ell) = 0,
 \end{alignedat}
\ee
where $r>0$ and $z>-1$ for $0<s<\ell$. Note that, by the implicit function theorem, solutions $(r(s),z(s))$ of \eqref{eq:arclengthsystem} can be written as $z=u(r)$, which are solutions of \eqref{eq:gov2dradpde}, if and only if $r'(s) \neq 0$ for all $s\in (0,\ell)$. Numerical simulation of \eqref{eq:arclengthsystem} demonstrates that this arc-length parameterization allows for the bifurcation curve to be continued beyond $(\lambda_*,|\alpha_*|)$ and that an infinite fold points structure, similar to that of \eqref{eq:stdpde} for $n=2$, is recovered (see Figures \ref{fig:pmcbds2d1}--\subref{fig:pmcbds2dzoom}). However, all of the additional solutions of \eqref{eq:arclengthsystem} beyond $(\lambda_*,|\alpha_*|)$ are observed to be multivalued (c.f. \cite{brubaker2011nonlinear}).

\begin{figure}[h]
\centering
\subfigure[]{\label{fig:mcbds2d1}\includegraphics[width=0.425\textwidth]{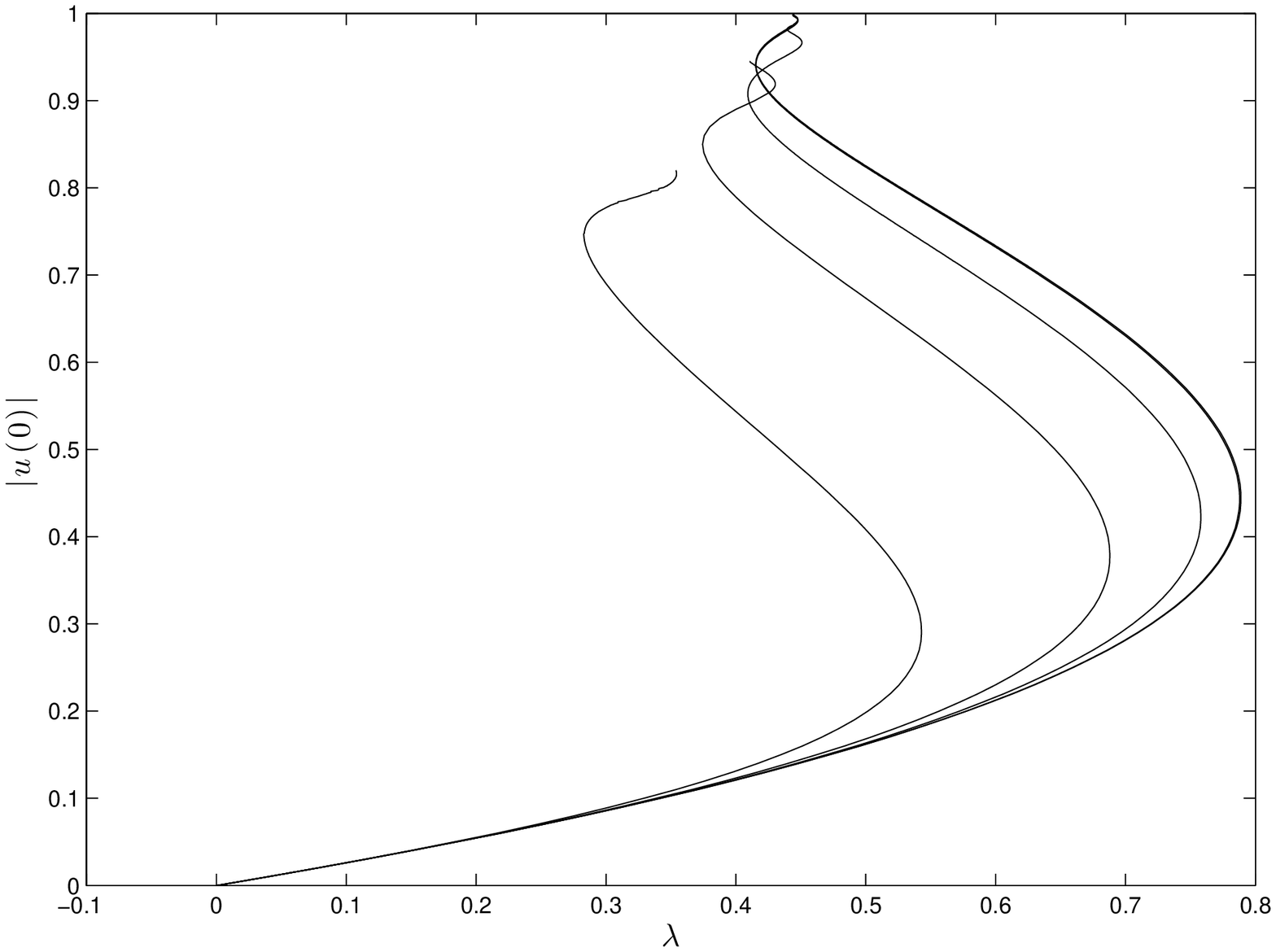}}\qquad
\subfigure[Zoomed]{\label{fig:mcbds2dzoom}\includegraphics[width=0.425\textwidth]{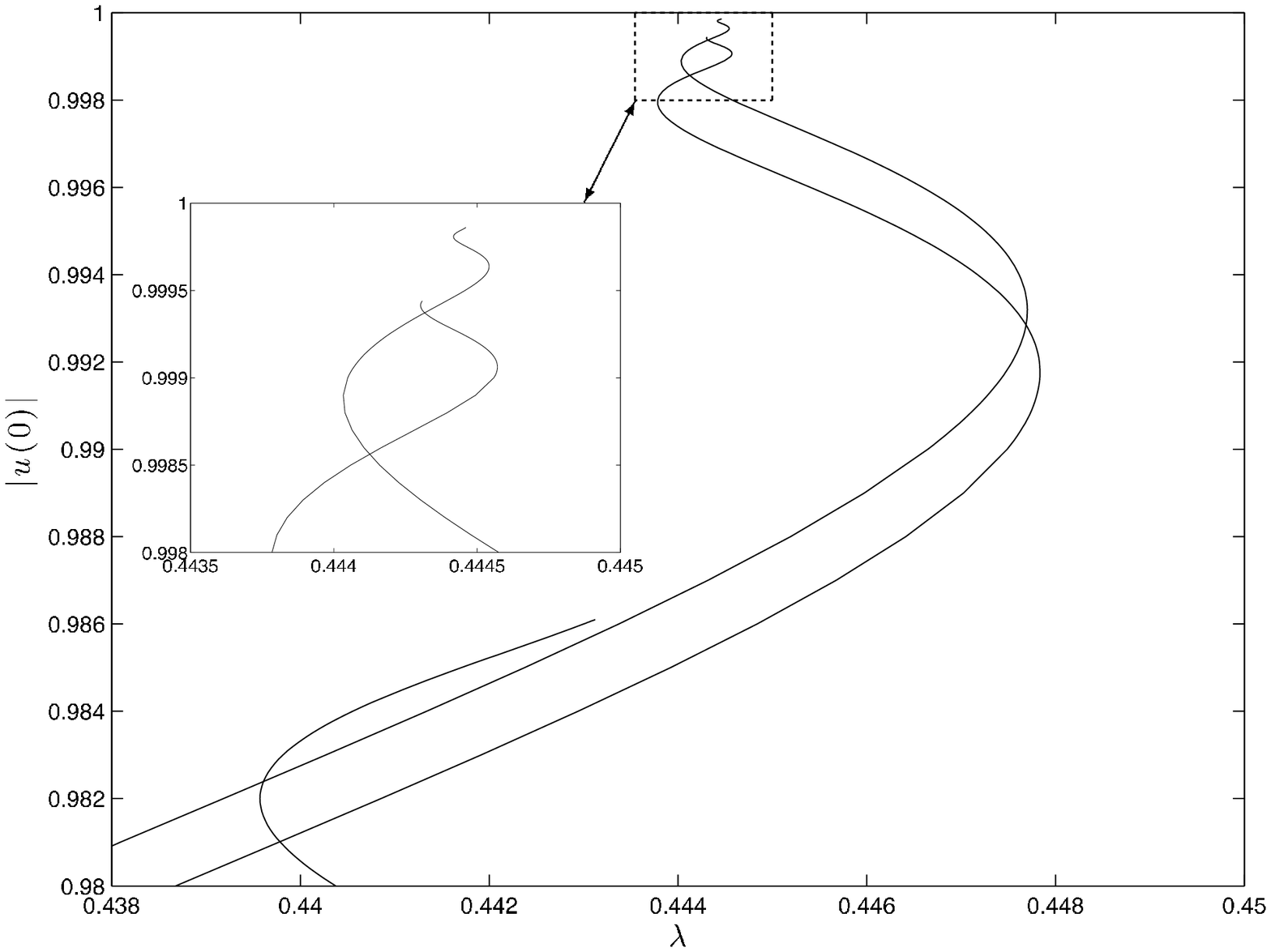}}\\
\subfigure[]{\label{fig:pmcbds2d1}\includegraphics[width=0.425\textwidth]{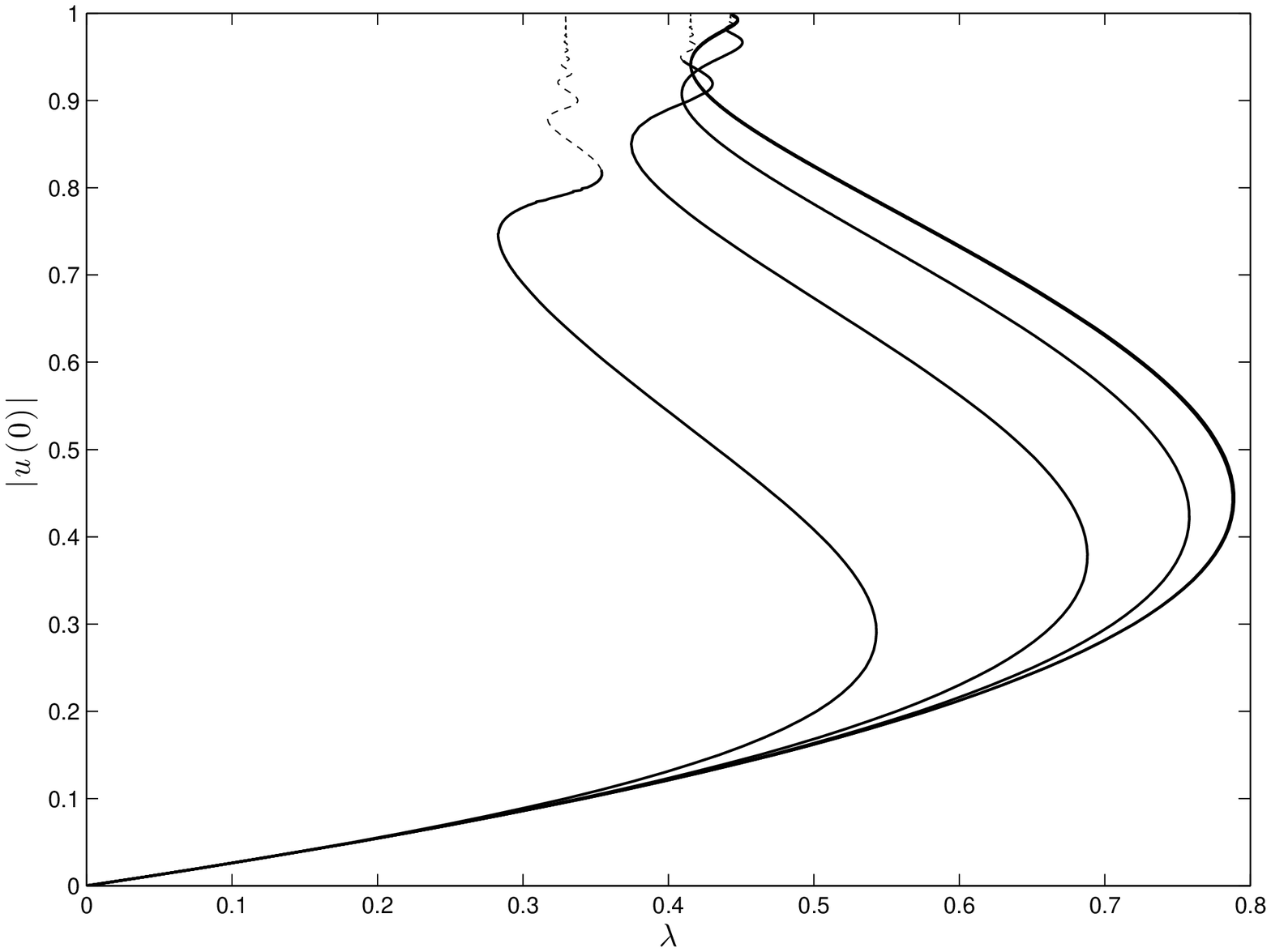}}\qquad
\subfigure[Zoomed]{\label{fig:pmcbds2dzoom}\includegraphics[width=0.425\textwidth]{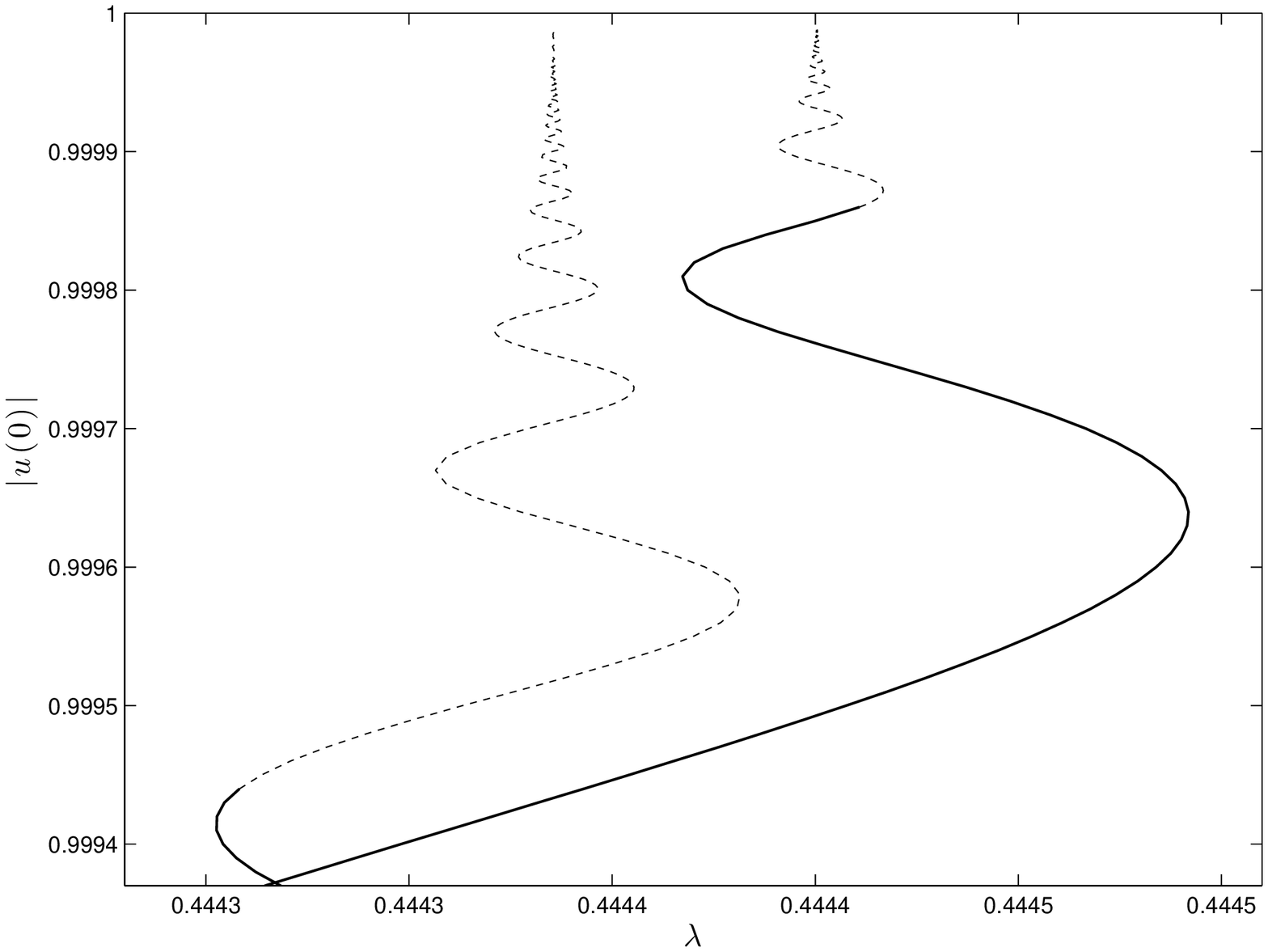}}
\caption{ \subref{fig:mcbds2d1} Bifurcation curves of \eqref{eq:gov2dradpde} computed via numerics for various $\eps$.  From right to left the  curves correspond to $\eps = 0.05,0.1, 0.5,1,2$. Note that at this scale $\eps = 0.05$ and $\eps = 0.1$ appear equal. \subref{fig:mcbds2dzoom} A magnified portion of \subref{fig:mcbds2d1}. Here, the curves for $\eps = 0.05,0.1,0.5$ are seen. \subref{fig:pmcbds2d1} Bifurcation curves of \eqref{eq:arclengthsystem} computed via numerics for various $\eps$. The dashed line represents multi-valued solutions of the system \eqref{eq:arclengthsystem}, which naturally continue on from the final classical solutions of \eqref{eq:gov2dradpde}.  From right to left the curves correspond to $\eps = 0.05, 0.1, 0.5,1,2$. \subref{fig:pmcbds2dzoom} A magnified portion of \subref{fig:pmcbds2d1}, where $\eps = 0.05, 0.1$ correspond to the right and left curves, respectively.\label{fig:mcbds2d}}
\end{figure}

The disappearance of solutions behavior is not isolated to \eqref{eq:govpde} and has arisen in other mean curvature type equation, notably \cite{finn1986equilibrium}, \cite{moulton2008theory} and \cite{pan2009one}, which studied the shape of a pendant drop, electrostatic deflections of a catenoid and a one-dimensional Gelfand-Bratu type problem, respectively. In addition, issues of existence, uniqueness and multiplicity of solutions to problems of general type
\begin{equation}\nonumber
  \mbox{div } \frac{\nabla u }{\sqrt{1 + |\nabla u|^2}} = \lambda f(u,x), \quad x\in \Omega; \qquad  u=0,\quad x\in \partial\Omega
\end{equation}
have been a topic of recent consideration by several authors (see \cite{pan2011radial,pan2011time1, pan2011time2,burns2011steady, Obersnel2010, mellet2010existence, le2008subsupersolution, habets2004positive} and the references there-in). 

The main goal of this paper is to analyze the radially symmetric upper solution branch of \eqref{eq:govpde} in the limit $\|u\|_{\infty}\to1$ for the one- and two-dimensional unit ball. 

For the case $n=1$, we analyze the upper solution branch of \eqref{eq:govpde} for $\Omega = \{x \in \R : |x|<1\}$ using the method of matched asymptotic expansions in the limit $\delta\ce 1+u(0)\to0^{+}$. In doing so, our analysis relies heavily on the use of logarithmic switchback terms (cf. \cite{lindsay2011asymptotics2,hinch1991perturbation}). Our main result for the case $n=1$ recovers the limiting form of the bifurcation diagram as $\delta:= 1+u(0)\to0^{+}$ and is encapsulated in Principal Result \ref{pr:1d_conc}. We remark that Principal Result \ref{pr:1d_conc} is established for any $\eps>0$, and is therefore in agreement with the result (c.f. \cite{brubaker2011analysis}) that \eqref{eq:govpde} admits solutions for $u(0)$ arbitrarily close to $-1$ in the case $\Omega = \{x \in \R : |x|<1\}$.

In \S\ref{section:2d}, radially symmetric solutions of \eqref{eq:govpde} on the unit disc are analyzed with particular focus on the nature of the \emph{dead-end bifurcation}. In \S\ref{subsection:nc}, a rigorous necessary condition is established on solutions of \eqref{eq:gov2dradpde}, namely that for any $\lambda<\lambda^{\ast}$ and $\eps>0$, there exists an $\alpha_{\ast}(\eps,\lambda)\in\mathbb{R}$ such that $|u(0)|<|\alpha_{\ast}|<1$. This result is proved in the Theorems leading up to Corollary \ref{coro:fp_mc} and demonstrates that unlike \eqref{eq:govpde} in the $n=1$ case and \eqref{eq:stdpde} for $n=1,2$, \eqref{eq:gov2dradpde} has no solutions for $u(0)$ arbitrarily close to $-1$. This loss of a classical solution is shown to be due to the formation of a singularity in the derivative at a point internal to the domain.

To complement the aforementioned qualitative result, we employ a novel formal asymptotic analysis to gain insight into the disappearance of solutions at the \emph{dead-end point} and establish a very accurate prediction of its location. The analysis demonstrates that the disappearance of solutions is connected in an intricate way to small values of the parameters $\eps$ and $\delta \ce 1+ u(0)$. Therefore in this case, the perturbation analysis involves two small parameters and must be performed in the distinguished limit $\eps^2 / \delta = \delta_0$ for $\delta_0 = \bigoh(1)$. Consequently, this formal approach allows for an explicit characterization of the upper solution branch in terms of two functions determined by the solution of two associated initial value problems. These two different methods of computation provide a very accurate prediction of the dead-end point, as shown in Principal Result \ref{pr:2drad_np}, by specifying a critical $\delta_0^{\ast}$ beyond which the asymptotic solution fails. In agreement with Corollary \ref{coro:fp_mc}, this implies that asymptotic solutions are not valid as $\delta\to0$. 

In \S\ref{sect:arclen2d}, the arc length parameterization system \eqref{eq:arclengthsystem} is studied and is found to be amenable to the aforementioned asymptotic analysis. An accurate representation of the solution branch of \eqref{eq:arclengthsystem} in the limit $\delta \ce 1+z(0)\to0^{+}$ is accordingly obtained in Principal Result \ref{pr:2drad_p} which is found to be valid for $\delta \to0$. These additional multi-valued solutions therefore provide a natural continuation beyond the \emph{dead-end point} associated with \eqref{eq:gov2dradpde}. Finally, in section \ref{section:conclusion}, a few open problems are discussed.

\section{One-dimensional analysis}\label{sec:onedim}

In this section we construct an asymptotic approximation to the maximal solution of \eqref{eq:govpde} in the case where $\Omega$ is the interval $[-1,1]$. As \eqref{eq:govpde} is symmetric about the origin, we may reduce the domain to $[0,1]$, impose the condition $u'(0)=0$ and solve the resulting problem
\be\label{eq:pmcoriginalsym}
 \begin{alignedat}{2}
  & \left(\frac{u'}{\sqrt{1 + \eps^2(u')^2}}\right)' =  \frac{\lambda}{(1+u)^2}, & \quad &0< x<1; \qquad  u'(0) = u(1) = 0,
 \end{alignedat}
\ee
in the limit $\delta \colonequals 1 - |u(0)| \to0^+$. From the bifurcation diagrams given in Figure \ref{fig:mcbds1d}, we see that $\lambda = \littleoh(1)$ as $\delta \to 0^+$ and, in accordance, assume  
\begin{subequations}\label{eq:origoutuandlamexp}
\begin{equation}\label{eq:lambdaexpansion}
  \lambda = \nu_1(\delta) \lambda_1 + \nu_1^2(\delta) \lambda_1 + \ldots,
\end{equation}
where $\nu_1(\delta)$ is a gauge function to be determined satisfying $\nu_1(\delta) \to 0^{+}$ as $\delta \to 0^+$. Furthermore, the following general form for the initial expansion of $u$
\begin{equation}\label{eq:uouterexpansion} 
 u  = u_0 + \nu_1(\delta)  u_1 + \nu_1^2(\delta) u_2 + \ldots,
\end{equation}
\end{subequations}
is employed. As we will see, the final asymptotic structure of the limiting solutions as $\delta\to0^+$ is not completely obvious, and indeed our initial expansion \eqref{eq:origoutuandlamexp} will require tailoring as the singularity structure of the solution manifests itself. Plugging \eqref{eq:origoutuandlamexp} into \eqref{eq:pmcoriginalsym} and collecting terms, we obtain
\begin{alignat}{2}
  & \left(\frac{u_0'}{\sqrt{1+\eps^2(u_0')^2}}\right)' = 0,  \quad 0< x<1, & \qquad u_0'(0) = u_0(1) = 0, \label{eq:outerord1}\\
  & \left(\frac{u_1'}{(1+\eps^2(u_0')^2)^{3/2}}\right)' = \frac{\lambda_1}{(1+u_0)^2},  \quad 0< x<1, & \qquad u_1'(0) = u_1(1) = 0, \label{eq:outerord2}
\end{alignat}
and
\begin{equation} \label{eq:outerord3}
\begin{alignedat}{2}
  & \left(\frac{u_2'}{(1+\eps^2(u_0')^2)^{3/2}}\right)' = \frac{\lambda_2(1+u_0)-2\lambda_1u_1}{(1+u_0)^3}+ \left(\frac{3\eps^2 u_0' (u_1')^2}{(1+\eps^2(u_0')^2)^{5/2}} \right)',  \ 0< x<1, \\
& \qquad u_2'(0) = u_2(1) = 0.
\end{alignedat}
\end{equation}
In carrying out the differentiation in \eqref{eq:outerord1}, we find that $u_0''=0$ for $x\in(0,1)$, and therefore, $u_0$ is a linear function of $x$. Therefore the solution of \eqref{eq:outerord1} that satisfies $u_0(1)=0$ and $u_0(0) = -1$ is 
\begin{equation}\label{eq:ord1outersoln}
 u_0(x) = -1+x, \qquad x\in(0,1),
\end{equation}
 which cannot satisfy  $u_0'(0)=0$. Note that the leading order outer solution on $[-1,1]$ is $u_0 = -1 + |x|$, which is not differentiable at the origin. This indicates the presence of a boundary layer in the vicinity of $x=0$. 
 
 Then using \eqref{eq:ord1outersoln} in  \eqref{eq:outerord2}--\eqref{eq:outerord3}  and solving theses equations iteratively, with respect to the right boundary condition, gives that
\begin{equation}\label{eq:ord2outersoln}
  u_1(x) = -\lambda_1(1 + \eps^2)^{3/2} \log x+B_1(1 - x)
\end{equation}
 and 
\begin{equation}\label{eq:ord3outersoln}
\begin{aligned}
 u_2(x) & =   (1 + \eps^2)^3 \lambda_1^2\frac{\log x}{x} + \frac{3\lambda_1^2(1 + \eps^2)^{2}- 2 B_1 \lambda_1(1 + \eps^2)^{3/2}  }{2x} \\
	& \qquad +  (B_1 \lambda_1 ( \eps^2-2) \sqrt{1 + \eps^2} -(1 + \eps^2)^{3/2} \lambda_2) \log{x} + C_1 \\
	& \qquad - \left(\frac{3\lambda_1^2(1 + \eps^2)^{2}- 2 B_1 \lambda_1(1 + \eps^2)^{3/2}  }{2} +C_1\right) x,
\end{aligned}
\end{equation}
where $B_1$ and $C_1$ are constants pertaining to an arbitrary solution of the homogeneous equation and are to be determined. Therefore, from \eqref{eq:ord1outersoln}--\eqref{eq:ord3outersoln}, 
\begin{equation}\label{eq:outersoln2orders}
\begin{aligned}
  u & = -1+x + \nu_1(\delta) u_1(x) + \nu_1^2(\delta) u_2  + \ldots \quad \mbox{ as } \delta \to 0^+.
 \end{aligned}
\end{equation}
By introducing the inner variable $y = x/ \gamma$, we find from \eqref{eq:outersoln2orders} that the outer solution has the behavior  $u = -1+\gamma y + \cdots \quad \mbox{as } x \to 0^+. $ Since $u = -1 + \bigoh(\delta)$ in the inner layer, we set $\gamma = \delta$ and define the following inner variables:
\begin{alignat}{2}
   y & = x/ \delta, & \qquad \qquad  u(x) & =-1 + \delta w(y).\label{eq:localvar}
\end{alignat}
Then substituting \eqref{eq:localvar}, along with \eqref{eq:lambdaexpansion}, into \eqref{eq:pmcoriginalsym} we obtain
\begin{equation}\label{eq:innerbalance}
 \left( \frac{w'}{\sqrt{1 + \eps^2(w')^2}}\right)'=  \frac{\nu_1}{\delta}\frac{(\lambda_1 + \cdots )}{ w^2},
\end{equation} 
and a dominant balance requires that $\nu_1 =\delta$. From this we have that $x = \delta y$, which implies that the local behavior of the outer expansion, \eqref{eq:outersoln2orders}, is given by  
\begin{equation}\nonumber
\begin{aligned}
  u & = -1 + \delta \log \delta \left( -\lambda_1(1 + \eps^2)^{3/2} + \frac{(1 + \eps^2)^3 \lambda_1^2}{ y}\right) \\
& \quad +  \delta \left( y-\lambda_1(1 + \eps^2)^{3/2} \log y+ B_1 + (1 + \eps^2)^3 \lambda_1^2\frac{\log y}{ y}\right. \\
& \hspace{2in}+ \left.\frac{3\lambda_1^2(1 + \eps^2)^{2}- 2 B_1 \lambda_1(1 + \eps^2)^{3/2}  }{2 y}\right) \\
& \quad + (\delta^2 \log{\delta}) (B_1 \lambda_1 ( \eps^2-2) \sqrt{1 + \eps^2} -(1 + \eps^2)^{3/2} \lambda_2) \\
& \quad + \delta^2 (-B_1 y +  (B_1 \lambda_1 ( \eps^2-2) \sqrt{1 + \eps^2} -(1 + \eps^2)^{3/2} \lambda_2) \log{y}+  C_1) +  \littleoh(\delta^2) 
 \end{aligned}
\end{equation}
as $\delta \to 0^+$. However, the $\bigoh(\delta \log \delta)$  and $\bigoh(\delta^2 \log \delta)$ terms cannot be matched to the inner expansion; so, we modify the outer expansions assumed in \eqref{eq:origoutuandlamexp} for $u$ and $\lambda$ to include two switchback terms:
\begin{subequations}\label{eq:outuandlamexpsb1}
 \begin{equation}\label{eq:uouterexpnsb1}
  u = u_0 + (\delta \log\delta) u_{1/2} + \delta u_1+ (\delta^2 \log\delta) u_{3/2} + \delta^2 u_2+ \ldots,
 \end{equation}
and
 \begin{equation}\label{eq:lamexpnsb1}
  \lambda = \delta \lambda_1 + (\delta^2 \log \delta) \lambda_{3/2} + \delta^2 \lambda_2 +\ldots.
\end{equation}
\end{subequations}
By substituting \eqref{eq:outuandlamexpsb1} into \eqref{eq:pmcoriginalsym}, we find that $u_{1/2}$ satisfies
\begin{alignat}{2} \nonumber
  & u_{1/2}'' = 0,  \quad 0< x<1; & \qquad  \ u_{1/2}(1) =0.
\end{alignat}
Thus, the solution for $u_{1/2}$ is given by 
\begin{equation}\label{eq:switchbacksolutionu12}
 u_{1/2}(x) = A_{1/2}(x-1),
\end{equation}
where $A_{1/2}$  is a constant chosen to eliminate the order $\bigoh(\delta \log \delta)$ term. Similarly, from using \eqref{eq:outuandlamexpsb1} and \eqref{eq:switchbacksolutionu12} in \eqref{eq:pmcoriginalsym} we find that $u_{3/2}$ satisfies the ordinary differential equation, 
\begin{equation}\nonumber
\begin{alignedat}{2}
  & u_{3/2}'' =\frac{2 A_{1/2} (1 + \eps^2)^{3/2} \lambda_1}{x^3} \\
  	& \qquad +\frac{ \sqrt{1 + \eps^2}(A_{1/2}\lambda_1 (\eps^2-2)  + (1 + \eps^2) \lambda_{3/2})}{x^2},  \quad 0< x<1;\\
  &u_{3/2}(1) =0,
\end{alignedat}
\end{equation}
whose solution is
\begin{equation}\label{eq:switchbacksolutionu32}
\begin{aligned}
 u_{3/2}(x) & = \frac{A_{1/2} \lambda_1(1 + \eps^2)^{3/2}}{x} -\sqrt{1 + \eps^2} (A_{1/2} (\eps^2-2) \lambda_1 + (1 + \eps^2) \lambda_{3/2}) \log {x} \\
	& \qquad -A_{1/2} (1 + \eps^2)^{3/2} \lambda_1 - A_{3/2}(x-1).
\end{aligned}
\end{equation}
Upon substituting \eqref{eq:ord1outersoln}, \eqref{eq:ord2outersoln}, \eqref{eq:switchbacksolutionu12} and \eqref{eq:switchbacksolutionu32} into \eqref{eq:outuandlamexpsb1} and rewriting result in terms of $y=x/\delta$, we have that the local behavior of the outer expansion as $x\to 0^+$  is given by
\begin{equation}\label{eq:localbehavioroutersolution1}
\begin{aligned}
  u &= -1 + \delta \log \delta \left(-A_{1/2} -\lambda_1(1 + \eps^2)^{3/2} + \frac{(1 + \eps^2)^3 \lambda_1^2}{ y} + \frac{A_{1/2} \lambda_1(1 + \eps^2)^{3/2}}{y}\right) \\
	& \quad +  \delta \left( y-\lambda_1(1 + \eps^2)^{3/2} \log y+ B_1 + (1 + \eps^2)^3 \lambda_1^2\frac{\log y}{ y}\right. \\
	& \hspace{2in}+ \left.\frac{3\lambda_1^2(1 + \eps^2)^{2}- 2 B_1 \lambda_1(1 + \eps^2)^{3/2}  }{2 y}\right) \\
	& \quad + (\delta^2 \log^2{\delta})\left(-\sqrt{1 + \eps^2} (A_{1/2} (\eps^2-2) \lambda_1 + (1 + \eps^2) \lambda_{3/2})\right) \\
	& \quad + (\delta^2 \log{\delta}) \left(A_{1/2}y-\sqrt{1 + \eps^2} (A_{1/2} (\eps^2-2) \lambda_1 + (1 + \eps^2) \lambda_{3/2})  \log{ y}  \right. \\
	& \qquad \ \ \left.-A_{1/2} (1 + \eps^2)^{3/2} \lambda_1 + A_{3/2} + B_1 \lambda_1 ( \eps^2-2) \sqrt{1 + \eps^2} -(1 + \eps^2)^{3/2} \lambda_2\right) \\
	& \quad + \delta^2 \left(-B_1 y +  (B_1 \lambda_1 ( \eps^2-2) \sqrt{1 + \eps^2} -(1 + \eps^2)^{3/2} \lambda_2) \log{y}+  C_1\right) + \littleoh(\delta^2).  \\
	& \quad 
 \end{aligned}
\end{equation}
Therefore to eliminate the $\bigoh(\delta\log\delta)$ terms from \eqref{eq:localbehavioroutersolution1}, we choose $A_{1/2}=-\lambda_1(1 + \eps^2)^{3/2}$, and \eqref{eq:localbehavioroutersolution1} becomes
\[
\begin{aligned}
  u &= -1 +  \delta \left( y-\lambda_1(1 + \eps^2)^{3/2} \log y+ B_1 + (1 + \eps^2)^3 \lambda_1^2\frac{\log y}{ y} \right.\\
  	& \hspace{2in}+\left. \frac{3\lambda_1^2(1 + \eps^2)^{2}- 2 B_1 \lambda_1(1 + \eps^2)^{3/2}  }{2 y}\right) \\
	& \ + (\delta^2 \log^2{\delta})\left(\lambda_1^2(1 + \eps^2)^{2} (\eps^2-2) - (1 + \eps^2)^{3/2}\lambda_{3/2})\right) \\
	& \ + (\delta^2 \log{\delta}) \left(A_{1/2}y-\sqrt{1 + \eps^2} (A_{1/2} (\eps^2-2) \lambda_1 + (1 + \eps^2) \lambda_{3/2})  \log{ y} \right. \\
	&  \hspace{0.5in}   - \left. A_{1/2} (1 + \eps^2)^{3/2} \lambda_1 + A_{3/2} + B_1 \lambda_1 ( \eps^2-2) \sqrt{1 + \eps^2} -(1 + \eps^2)^{3/2} \lambda_2\right) \\
	& \ + \delta^2 \left(-B_1 y +  (B_1 \lambda_1 ( \eps^2-2) \sqrt{1 + \eps^2} -(1 + \eps^2)^{3/2} \lambda_2) \log{y}+  C_1\right) + \littleoh(\delta^2).  \\
 \end{aligned}
\]
But the new order $\bigoh(\delta^2 \log^2 \delta)$ term cannot be matched to the inner expansion, and again we must modify our assumed outer asymptotic expansion, \eqref{eq:uouterexpnsb1}, to include another switchback term:
\begin{equation}\label{eq:finaltrueuouterexpansion}
 u = u_0 + (\delta \log\delta) u_{1/2} + \delta u_1+ (\delta^2 \log^2\delta) u_{5/4}+ (\delta^2 \log\delta) u_{3/2} + \delta^2 u_2+ \cdots.
\end{equation}
Inserting \eqref{eq:finaltrueuouterexpansion} and \eqref{eq:lamexpnsb1} into \eqref{eq:pmcoriginalsym}, we obtain
\[
\begin{alignedat}{2}
  & u_{5/4}'' = 0,  \quad 0< x<1; & \qquad  \ u_{5/4}(1) =0, \label{eq:outerord54}
\end{alignedat}
\]
 which upon solving gives 
\[
 u_{5/4}(x) = A_{5/4}(x-1),
\]
where $A_{5/4}$  is a constant chosen to eliminate the order $\bigoh(\delta^2\log^2 \delta)$ term. Specifically, we let
\begin{equation}\nonumber
 A_{5/4}=-\sqrt{1 + \eps^2} (A_{1/2} (\eps^2-2) \lambda_1 + (1 + \eps^2) \lambda_{3/2})
\end{equation}
and the local behavior of the outer solution as $x\to 0^+$ is 
\begin{equation}\label{eq:localbehavioroutersolution2}
\begin{aligned}
  u &= -1 +  \delta \left( y-\lambda_1(1 + \eps^2)^{3/2} \log y+ B_1 + (1 + \eps^2)^3 \lambda_1^2\frac{\log y}{ y} \right.\\
   & \hspace{2in} \left. +\frac{3\lambda_1^2(1 + \eps^2)^{2}- 2 B_1 \lambda_1(1 + \eps^2)^{3/2}  }{2 y}\right) \\
	& \ + (\delta^2 \log{\delta}) \left(-\lambda_1(1 + \eps^2)^{3/2}y +(\lambda_1^2(1 + \eps^2)^{2} (\eps^2-2) - (1 + \eps^2)^{3/2}\lambda_{3/2}))\log{ y}\right. \\
	& \hspace{.7in}   + \left. \lambda_1^2(1 + \eps^2)^{3} + A_{3/2}+B_1 \lambda_1 ( \eps^2-2) \sqrt{1 + \eps^2} -(1 + \eps^2)^{3/2} \lambda_2\right) \\
	& \quad + \delta^2 \left(-B_1 y +  (B_1 \lambda_1 ( \eps^2-2) \sqrt{1 + \eps^2} -(1 + \eps^2)^{3/2} \lambda_2) \log{y}+  C_1\right) +\littleoh(\delta^2)
 \end{aligned}
\end{equation}

Besides the necessity for matching, the behavior given in \eqref{eq:localbehavioroutersolution2}, along with \eqref{eq:localvar}, suggests that the inner solution should be expanded as 
\[
 w = w_0 +  (\delta \log\delta) w_{1/2} + \delta w_1 + \ldots.
\]
From \eqref{eq:pmcoriginalsym}, \eqref{eq:innerbalance}, \eqref{eq:localbehavioroutersolution2} and the limiting behavior $u(0) = -1 + \delta$ as $\delta \to 0^+$, we obtain the following series of inner problems: 
\begin{equation}\label{eq:innerord1}
\begin{alignedat}{1}
  &\left(\frac{w_0'}{\sqrt{1 + \eps^2(w_0')^2)}}\right)' = \frac{\lambda_1}{ w_0^2},  \quad 0< y<\infty;  \qquad w_0'(0) =  0, \ w_0(0) = 1, \\
 & w_0 = y  - \lambda_1(1 + \eps^2)^{3/2} \log y+  B_1 + \littleoh(1)  \quad \mbox{as } y \to \infty,
\end{alignedat}
\end{equation}
\begin{equation}\label{eq:innerord2}
\left.
\begin{alignedat}{1}
  & \calig{L}w_{1/2} = \frac{\lambda_{3/2}}{ w_0^2},  \quad 0< y<\infty;  \quad   w_{1/2}'(0) =  0, \ w_{1/2}(0) = 0, \\
  & \begin{aligned}w_{1/2} = -\lambda_1(1 & + \eps^2)^{3/2}y +(\lambda_1^2(1 + \eps^2)^{2} (\eps^2-2)  - (1 + \eps^2)^{3/2}\lambda_{3/2}))\log{ y}  \\ & +\chi + \littleoh(1) \quad \mbox{as } y \to \infty, \end{aligned}
\end{alignedat}\right.
\end{equation}
 and
\begin{equation}\label{eq:innerord3}
\begin{alignedat}{1}
  & \calig{L}w_1 = \frac{\lambda_2}{ w_0^2},  \quad 0< y<\infty; \qquad  w_1'(0) =  0, \ w_1(0) = 0, \\
 & \begin{aligned} w_1 = - B_1& y  +  (B_1 \lambda_1 ( \eps^2-2) \sqrt{1 + \eps^2} -(1 + \eps^2)^{3/2} \lambda_2) \log{y} \\
 	& + C_1 +\littleoh(1) \ \mbox{as } y \to \infty,\end{aligned}
\end{alignedat}
\end{equation}
where
\begin{equation}\nonumber
 \calig{L}\varphi \colonequals \left(\frac{\varphi'}{(1 + \eps^2(w_0')^2)^{3/2}} \right)' + \frac{2 \lambda_1}{w_0^3} \varphi.
\end{equation}
In particular, the solution of these inner problems uniquely determine $\lambda_1$, $\lambda_{3/2}$, $\lambda_2$, $B_1$, $A_{3/2}$ and $C_1$. Note that the far field condition $w'_0(\infty)=1$ in \eqref{eq:innerord1} fixes the value of $\lambda_1$, which in turn allows the solution of \eqref{eq:innerord1}{\emdash}and accordingly the value of $B_1${\emdash}to be uniquely determined. In \eqref{eq:innerord2}, the now fixed far field condition $w'_{1/2}(\infty)=-\lambda_1(1+\eps^2)^{3/2}$, uniquely determines the solution $(\lambda_{3/2},w_{1/2})$ and, consequently, the value of $\chi$. By comparing \eqref{eq:innerord2} with \eqref{eq:localbehavioroutersolution2}, the following linear equation relating the unknowns $A_{3/2}$, $\chi$ and $\lambda_2$ is obtained
\[
\chi = \lambda_1^2(1 + \eps^2)^{3} + A_{3/2}+B_1 \lambda_1 ( \eps^2-2) \sqrt{1 + \eps^2} -(1 + \eps^2)^{3/2} \lambda_2.
\]
This process can be continued to fix the values of $\lambda_2$ and $C_1$ in \eqref{eq:innerord3}.

To determine $\lambda_1$, we first multiply the ODE of \eqref{eq:innerord1} by $w_0'$, and observe that $w_0$ satisfies the following first integral for $y\in(0,\infty)$:
\begin{equation}\label{eq:1stint}
 - \frac{1}{\eps^2\sqrt{1+\eps^2(w_0'(y))^2}} +\lambda_1\frac{1}{w_0(y)}= - \frac{1}{\eps^2} +\lambda_1.
\end{equation}
Then after taking $y \to \infty$ and using the limiting behavior of $w_0$ given in \eqref{eq:innerord1}, \eqref{eq:1stint} yields
\begin{equation}\nonumber
 \lambda_1  = \frac{\sqrt{1+\eps^2}-1}{\eps^2\sqrt{1+\eps^2}}.
\end{equation}

To determine $B_1$, we first solve for $w_0'$ in \eqref{eq:1stint} to find 
\[
 \D{w_0}{y}  = \frac{\sqrt{\lambda_1}\sqrt{(2 - \eps^2 \lambda_1) w_0^2 -2(1 - \eps^2 \lambda_1) w_0-\eps^2 \lambda_1}}{\eps^2 \lambda_1 + (1 - \eps^2 \lambda_1) w_0} , \ 0<y<\infty; \quad w_0(0)=1.
\]
An integration of this ODE yields 
\begin{equation}\label{eq:integralasymbehavior}
 \frac{1}{\sqrt{\lambda_1}}\int_1^{w_0(y)}\frac{\eps^2 \lambda_1 + (1 - \eps^2 \lambda_1) z}{\sqrt{(2 - \eps^2 \lambda_1) z^2 -2(1 - \eps^2 \lambda_1) z-\eps^2 \lambda_1}}\ud z = y,
\end{equation}
where now the integral on the left-hand side can be explicitly computed and then expanded for $y \gg 1$ (see \ref{app:intexpn}). This computation yields  
\[
\begin{alignedat}{2}
 y   & = \frac{1}{\sqrt{\lambda_1}}\frac{(1 - \eps^2 \lambda_1)}{ \sqrt{2 - \eps^2 \lambda_1}}w_0 + \frac{1}{\sqrt{\lambda_1}\left(2 - \eps^2 \lambda_1\right)^{3/2}}\log{w_0} \\
 & \qquad + \frac{\log\left(4 - 2\eps^2 \lambda_1\right) -(1 - \eps^2 \lambda_1)^2}{\sqrt{\lambda_1}\left(2 - \eps^2 \lambda_1\right)^{3/2}}+ \bigoh\left(w_0^{-1} \right)
\end{alignedat}
\]
as $w_0 \to \infty$, so that 
\begin{equation}\label{eq:w0farfieldbehavior}
 w_0 \sim y - \lambda_1(1 + \eps^2)^{3/2} \log{y} -  \lambda_1(1 + \eps^2)^{3/2}\left(\log\left(4 - 2\eps^2 \lambda_1\right) -\frac{1}{1+\eps^2}\right)
\end{equation}
as $y \to \infty$. Then in comparing \eqref{eq:w0farfieldbehavior} with the far field behavior in \eqref{eq:innerord1},  we find  
\begin{equation}\nonumber
  B_1 = -  \lambda_1(1 + \eps^2)^{3/2}\left(\log\left(4 - 2\eps^2 \lambda_1\right) -\frac{1}{1+\eps^2}\right).
\end{equation}

Next we determine $\lambda_{3/2}$. From Green's second identity and a differentiation of \eqref{eq:innerord1}, we obtain
\[
 \lim_{R\to \infty}\int_0^R \left( w_0' \calig{L}w_{1/2} - w_{1/2} \calig{L}w_0'   \right) \ud x = \lim_{R\to\infty} \left.  \frac{(w_0'w_{1/2}'-w_{1/2} w_0'')}{(1+\eps^2 (w_0')^2)^{3/2}}\right|_{y=0}^{y=R}
\]
and $\calig{L} w_0' = 0$ in $(0,\infty)$, respectively. Therefore, plugging the latter into the former{\emdash}along with the boundary conditions and limiting behavior of $w_0$ and $w_1${\emdash}gives 
\begin{equation}\nonumber
 \lambda_{3/2} \lim_{R\to \infty}\int_0^R \left(\frac{w_0'}{ w_0^2}\right) \ud x =-\lambda_1,
\end{equation}
and consequently, $\lambda_{3/2} = -\lambda_1$. Similarly, to determine $\lambda_2$, we have from Green's second identity, \eqref{eq:innerord1} and \eqref{eq:innerord3} that
\begin{equation}\nonumber
 \lim_{R\to \infty}\int_0^R \left( w_0' \calig{L}w_1 - w_1 \calig{L}w_0'   \right) \ud x = \lim_{R\to\infty} \left.  \frac{(w_0'w_1'-w_1 w_0'')}{(1+\eps^2 (w_0')^2)^{3/2}}\right|_{y=0}^{y=R} =   \frac{-B_1}{(1+\eps^2)^{3/2}}.
\end{equation}
Hence,
\begin{equation}\nonumber
 \lambda_2 =   - \frac{B_1}{(1+\eps^2)^{3/2}} =\lambda_1\left(\log\left(4 - 2\eps^2 \lambda_1\right) -\frac{1}{1+\eps^2}\right).
\end{equation}
The determination of $C_1$ from \eqref{eq:innerord3}, which is tedious and gives no special insight into the expansion for $\lambda$, is omitted.

We now summarize the preceding analysis for the maximal solution branch of \eqref{eq:pmcoriginalsym}.

\begin{pr}\label{pr:1d_conc}
 For $\delta = 1 + u(0) \to 0^+$, the three-term asymptotic expansion for the maximal solution branch of \eqref{eq:pmcoriginalsym} is given by 
\begin{equation}\label{eq:1dasymmaximalbranch}
 \lambda = \delta \lambda_1 - (\delta^2 \log{\delta})\lambda_1 + \delta^2 \lambda_1\left(\log\left(4 - 2\eps^2 \lambda_1\right) -\frac{1}{1+\eps^2}\right) + \littleoh(\delta^2)
\end{equation}
where $$\lambda_1  = \frac{\sqrt{1+\eps^2}-1}{\eps^2\sqrt{1+\eps^2}}.$$ 
\end{pr}

Figure \ref{fig:mcbds1dasymp} compares the asymptotic result of \eqref{eq:1dasymmaximalbranch} and previous numerical results for bifurcation diagram $\lambda$ versus $|u(0)|$ of \eqref{eq:pmcoriginalsym}. As seen, the three-term result is quite accurate. We also remark that this asymptotic formulation does not predict the dead-end point $\lambda_{\ast}$ seen in Fig.~\ref{fig:mcbds1d}.

\begin{figure}[h]
\centering
\subfigure[$\eps =1/2$]{\label{fig:mcbds1dleqasymp}\includegraphics[width=0.425\textwidth]{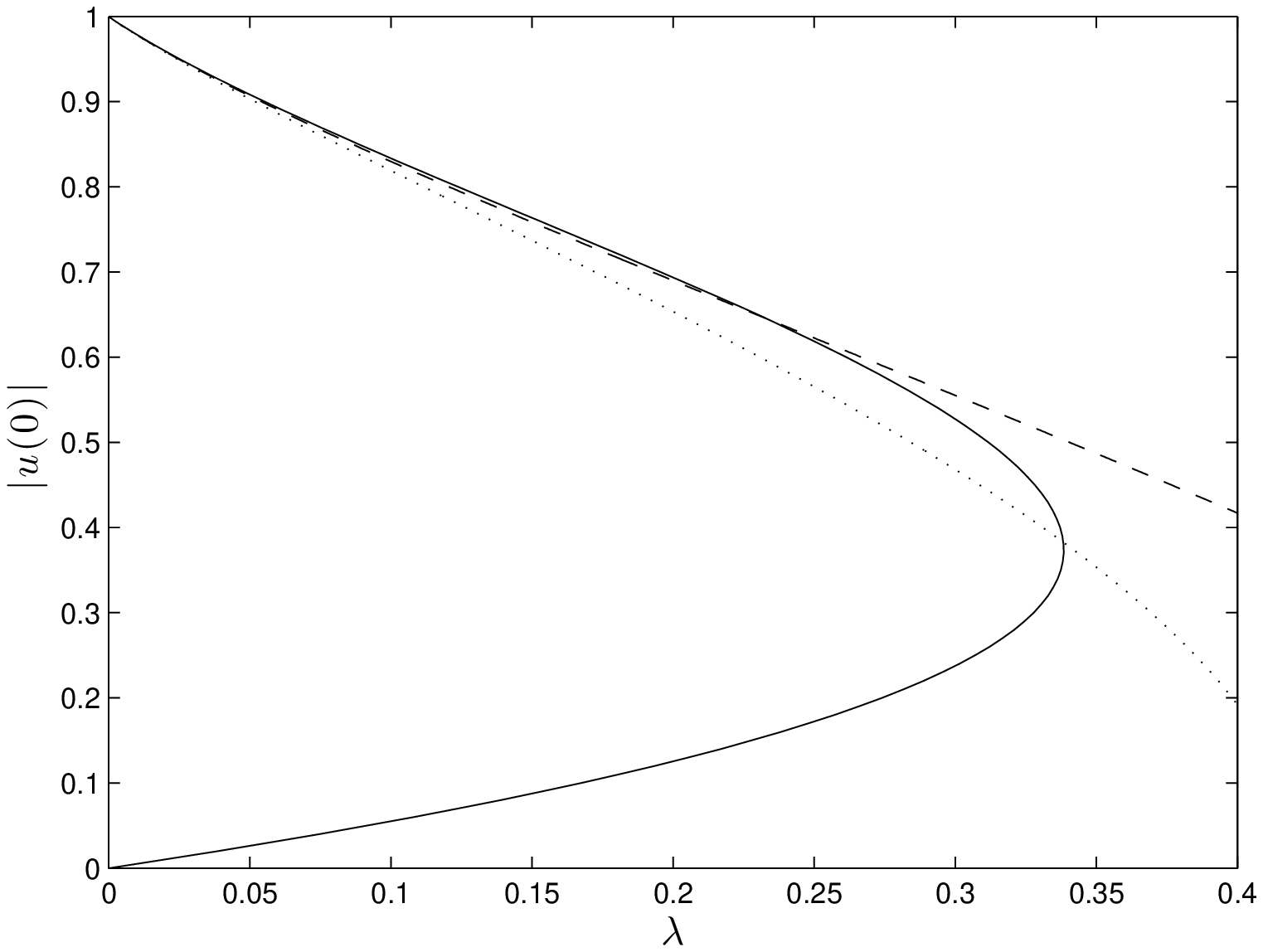}}\qquad
\subfigure[$\eps =10/3$]{\label{fig:mcbds1dgasymp}\includegraphics[width=0.425\textwidth]{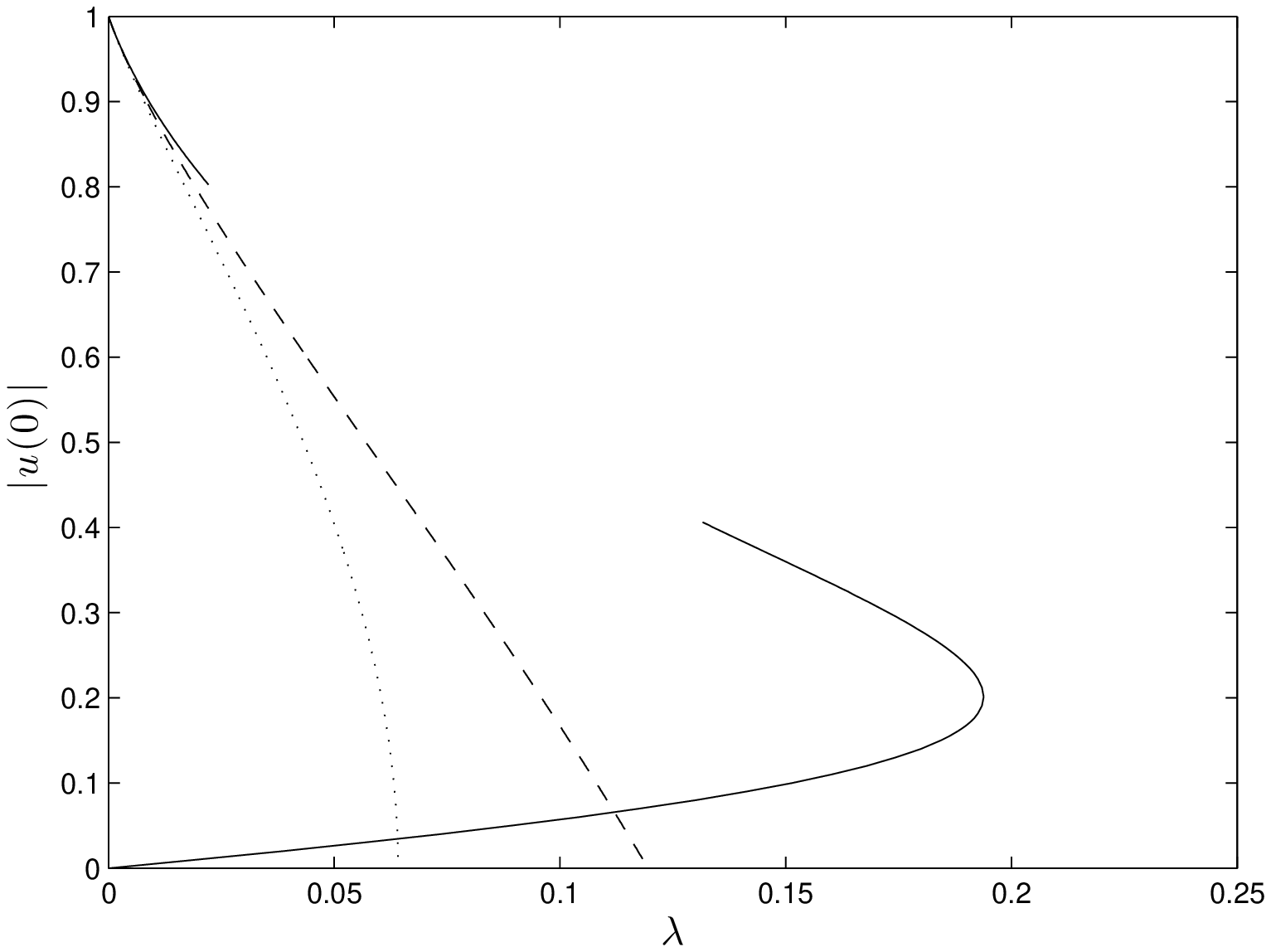}}
\caption{The two (dotted) and three (dashed) term asymptotic expansions from \eqref{eq:1dasymmaximalbranch} of the maximal solution branch  of \eqref{eq:pmcoriginalsym} compared with the full numerics (solid).\label{fig:mcbds1dasymp}}
\end{figure}

\section{Two-dimensional analysis}\label{section:2d}

\subsection{A necessary condition for existence of solutions}\label{subsection:nc}

In this section, we investigate the disappearance of classical solutions of \eqref{eq:gov2dradpde} observed in \cite{brubaker2011nonlinear} using techniques similar  to ones  introduced in \cite{finn1986equilibrium} to study the behavior of pendant drops. For convenience, we introduce the change of variable $\rho = r/\eps$, which from \eqref{eq:gov2dradpde} yields the nonlinear ordinary differential equation
\bsub\label{eq:gov2dradpde_cov}
\be\label{eq:gov2dradpde_cov_ode}
  \frac{1}{\rho} \left( \frac{\rho  u'  }{\sqrt{1 + (u')^2}}\right)' = \frac{\eps^2 \lambda }{(1+u)^2},   \quad 0<\rho <\eps^{-1}; \qquad   u'(0) = 0, \ u(\eps^{-1})=0,
\ee
with 
\be\label{eq:gov2dradpde_cov_bound}
  -1<u(\rho)<0, \qquad 0<\rho <\eps^{-1}.
\ee
\esub
Here, $'$ now represents differentiation with respect to $\rho$. In this form, the ordinary differential equation in \eqref{eq:gov2dradpde_cov_ode} admits the very advantageous geometrical interpretation:  
\[
 \left(\rho \sin \psi \right)'  = \frac{\eps^2 \lambda \rho}{(1+u)^2},
\]
where $\psi$ is the angle of inclination of the solution curve $u$, measured counterclockwise from the positive $\rho$-axis to its tangent. It is important to note that these ODEs are equivalent on any interval in which $|u'(\rho)| < \infty$.

Now to study the non-existence of solutions of \eqref{eq:gov2dradpde_cov}, we look at a corresponding initial value problem,
\be\label{eq:gov2dradpde_cov_ivp_ode}
  \left(\rho \sin \psi \right)'  = \frac{\eps^2 \lambda \rho }{(1+u)^2},   \quad \rho>0; \qquad   u'(0) = 0, \ u(0)=\alpha,
\ee
where $\alpha\in(-1,0)$. We begin by proving the following  lemma about solutions of \eqref{eq:gov2dradpde_cov_ivp_ode}.

\begin{lem}\label{lem:fp_ivp_increasing}
 If $u$ is a solution to \eqref{eq:gov2dradpde_cov_ivp_ode} for $ \rho \in [0,a)$, then $\sin \psi > 0$ for $ \rho \in (0,a)$, which implies that $u$ is increasing on that interval. Furthermore, we have the following bound:
\be\label{eq:bound1}
 \frac{\eps^2 \lambda}{2(1+ u(\rho))^2} < \frac{\sin \psi}{\rho} <\frac{\eps^2 \lambda}{2(1+\alpha)^2}.
\ee
\end{lem}
\begin{proof}
An integration of the differential equation in \eqref{eq:gov2dradpde_cov_ivp_ode}  yields
\be\label{eq:integral1}
 \sin \psi  =  \frac{\eps^2 \lambda}{\rho} \int_0^\rho \frac{\xi}{(1+u(\xi))^2} \ud{\xi} > 0, 	\quad \rho \in (0,a).
\ee
Now since $u$ is increasing on $(0,a)$, we have $\alpha < u(\rho) < u(a)$ for $\rho \in (0,a)$, and  \eqref{eq:integral1} gives \eqref{eq:bound1}.
\end{proof}

Next we prove a crucial lemma about the solutions the solutions of \eqref{eq:bound1}; however, first we state a comparison principle \cite{finn1986equilibrium}, which we use in our proof.

\begin{lem}[Comparison Principle]\label{lem:fp_comp}
  Let $\calig{F}(\xi,\eta)$ be a function such that for all $\xi>0$,
  \begin{equation}\nonumber
   \pd{\calig{F}}{\eta}(\xi,\eta)>0.
  \end{equation} 
  Furthermore, assume that $w_1(\xi)$ and $w_2(\xi)$ are functions defined for $\xi \in [a,b]$ such that 
  \[
   \pd{}{\xi} \left[\calig{F}\left(\xi,\pd{w_1}{\xi} \right)\right]  \geq \pd{}{\xi} \left[\calig{F}\left(\xi,\pd{w_2}{\xi} \right)\right]
  \]
on $[a,b]$. Suppose that 
  \[
   \pd{w_1}{\xi}(a) \geq \pd{w_2}{\xi}(a).
  \]
  Then 
  \be\label{eq:ip_inequalities}
   \pd{w_1}{\xi}(b) \geq \pd{w_2}{\xi}(b), \qquad {w_1}(b) -{w_2}(b) \geq {w_1}(a) - {w_2}(a).
  \ee
  Moreover, equality of \eqref{eq:ip_inequalities} holds if and only if $w_1(\xi) = w_2(\xi) + \mbox{constant}$ on $[a,b]$.
\end{lem}

Now we can prove the following theorem.

\begin{lem}\label{lem:fp_main}
 Assume that $u(\rho;\lambda)$ is a solution to \eqref{eq:gov2dradpde_cov_ivp_ode} on an initial interval, where $\lambda>0$. Also, let $\eps>0$ be fixed. If
 \bsub\label{eq:fp_u0_lambda_bound}
 \be\label{eq:fp_u0_lambda_bound_a}
   \alpha < -1 + \frac{\left(3-2\sqrt{2}\right)}{2}\eps^2 \lambda, \quad \mbox{when} \quad 0<\lambda\leq  \frac{4+3 \sqrt{2}}{\eps^2},
\ee
or
\be\label{eq:fp_u0_lambda_bound_b}
    \alpha <  \frac{\eps \sqrt{\lambda(\eps^2 \lambda-8)}-\eps^2 \lambda-4}{4 (\eps^2 \lambda+1)} < -1 + \frac{\left(3-2\sqrt{2}\right)}{2}\eps^2 \lambda, \ \ \mbox{when} \ \ \lambda > \frac{4+3 \sqrt{2}}{\eps^2},
\ee
 \esub
 then there exists a $\rho_1$ such that $u(\rho)$ cannot be continued beyond as a solution of \eqref{eq:gov2dradpde_cov_ivp_ode}. Moreover, at this point $(\rho_1,u_1)$, where $u_1 \ce u(\rho_1)$, the slope of $u'$ becomes vertical and 
\begin{align}
 \frac{2 (1+\alpha)^2}{\eps^2 \lambda } & < \rho_1 \leq \frac{2\left(1+M\right)^{2}}{\eps^2 \lambda},\label{eq:fp_bound_rho1}\\
 \alpha + \frac{2 (1+\alpha)^2}{\eps^2 \lambda}  &  < u_1<  M <0,\label{eq:fp_bound_u1}
\end{align}
where 
\be\label{eq:fp_def_M}
 M \ce \frac{3  \alpha+1}{2}-\frac{ (1+\alpha)^2}{\eps^2 \lambda}- \fr{(1+\alpha)\sqrt{\eps^4 \lambda^2-12 \eps^2 \lambda (\alpha+1)+4 (\alpha+1)^2}}{2 \eps^2 \lambda}.
\ee
\end{lem}

\begin{proof} 
First, we note that since $u$ is a solution of \eqref{eq:gov2dradpde_cov_ivp_ode}, 
\be\label{eq:gov2dradpde_angle_ode_split}
 \frac{\sin \psi}{\rho}  + (\sin \psi)'   = \frac{\eps^2 \lambda}{(1+u)^2}
\ee
on the initial interval. Moreover, by Lemma \ref{lem:fp_ivp_increasing}, $u$ is increasing on this initial interval and we may use it as independent variable; thus, using 
 \be\label{eq:fp_curvature_relation}
  \D{}{r}(\sin \psi) = - \D{}{u} (\cos{\psi}),
 \ee
in \eqref{eq:gov2dradpde_angle_ode_split} and then integrating the result with respect to $u$, we obtain
\[
 \frac{\eps^2 \lambda}{2(1+\alpha)^2} (u-\alpha)  + (1 - \cos{\psi(u)})   > \eps^2 \lambda \left[\frac{1}{1+\alpha}-\frac{1}{1+u}\right],
\]
where we have used \eqref{eq:bound1}; or, equivalently
\[
 1 - \cos{\psi(u)}   > \frac{\eps^2 \lambda(u-\alpha) }{(1+\alpha)(1+u)} - \frac{\eps^2 \lambda(u-\alpha) }{2(1+\alpha)^2}.
\]
Hence if \eqref{eq:fp_u0_lambda_bound_a} or \eqref{eq:fp_u0_lambda_bound_b} is true, then a vertical slope appears at a value
\[
 u_1<  M,
\] 
where $M$ is defined in \eqref{eq:fp_def_M}. Furthermore, the requirements in \eqref{eq:gov2dradpde_cov_ivp_ode} imply that $M<0$.

First from \eqref{eq:bound1} we have that $u$ can be continued at least until $\rho = \eps^2 \lambda/(2 (1+\alpha)^2)$, which implies that $\rho_1 >\eps^2 \lambda/(2 (1+\alpha)^2)$.

Next, let $w$ be defined as
\[
 w(\rho) \ce \alpha + \beta - \sqrt{\beta^2 - \rho^2}, \qquad \beta \ce \frac{2\left(1+M\right)^{2}}{\eps^2 \lambda},
\]
so $w(0)= \alpha $. Now, 
\[
\begin{aligned}
 \left(\calig{F}(\rho,w')\right)' & = \frac{\eps^2 \lambda \rho}{\left(1+M\right)^{2}} < \frac{\eps^2 \lambda \rho}{(1+u_1)^2} < \frac{\eps^2 \lambda \rho}{(1+u)^2}  = \left(\calig{F}(\rho,u')\right)',
 \end{aligned}
\]
for $\rho \in (0,R)$, where $R = \min\{\rho_1,\beta\}$ and 
\be\label{eq:calig_F}
 \calig{F}(\xi,\eta) \ce \xi \fr{\eta}{\sqrt{1+\eta^2}},
\ee
which by Lemma \ref{lem:fp_comp} implies
\begin{equation}\nonumber
   u'(\rho) < w'(\rho), \qquad {u}(\rho) < {w}(\rho), \quad \rho \in (0,R)
\end{equation}
Therefore, if $\beta < \rho_1$, then $w'(\rho) \to \infty$ as $\rho \to R^-$, which is a contradiction, and hence, $\rho_1 \leq \beta.$

To get the last bound, we recall that for the solution graph $(\rho,u)$ of \eqref{eq:gov2dradpde_cov_ivp_ode} $u$ can be used as the independent variable for the solution graph $(\rho,u)$. 
Therefore using \eqref{eq:bound1} and \eqref{eq:fp_curvature_relation} in \eqref{eq:gov2dradpde_angle_ode_split} gives 
\be\label{eq:fp_ind_u_angle}
 (\cos \psi)_u> -   \frac{\eps^2 \lambda}{2(1+u)^2}
\ee
Now let $(\til{\rho}(u),u)$ be a comparison surface for $u \in [\alpha, \alpha + 2(1+\alpha)^2/(\eps^2 \lambda)]$, defined as
\be\label{eq:fp_ind_u_comp_surf}
(\til{\rho}(u),u), \qquad \til{\rho}(u) \ce \sqrt{\frac{4 (1+\alpha)^4}{\eps^4 \lambda^2 } - \left(\left[\alpha + \frac{2 (1+\alpha)^2}{\eps^2 \lambda }\right] -u\right)^2 },
\ee
with corresponding angle of inclination $\varphi$, measured counterclockwise from the positive $\til{\rho}$-axis to its tangent. Therefore, 
\[
 (\cos{\varphi})_u = - \frac{\eps^2 \lambda }{2 (1+\alpha)^2}
\]
and from \eqref{eq:fp_ind_u_angle}, we have $(\cos \psi)_u>  (\cos{\varphi})_u $ for each $u\in (\alpha, U)$, where $U \ce \min\{u_1, \alpha+ 2(1+\alpha)^2/(\eps^2\lambda)\}$. Therefore,  $ \cos \psi> \cos\varphi$ for $u$ in the same interval. Hence the solution graph $(\rho,u)$ of \eqref{eq:gov2dradpde_cov_ivp_ode} can be continued vertically until the comparison surface \eqref{eq:fp_ind_u_comp_surf} becomes vertical, i.e., 
\[
 u_1 > w\left(\frac{2 (1+\alpha)^2}{\eps^2 \lambda}\right) = \alpha + \frac{2 (1+\alpha)^2}{\eps^2 \lambda}.
\]
\end{proof}

Therefore, we have that for $\alpha$ satisfying \eqref{eq:fp_u0_lambda_bound}, if $u$ is a solution to \eqref{eq:gov2dradpde_cov_ivp_ode} then its derivative blows-up in finite time. Furthermore, from the requirements on $\alpha$ in \eqref{eq:fp_u0_lambda_bound}, the blow-up point $(\rho_1,u_1)$ must happens for $u_1 <0$. In using this crucial fact, we can establish the following theorem, which rigorously proves, for all $\eps>0$, the disappearing solution behavior of \eqref{eq:gov2dradpde} observed in \cite{brubaker2011nonlinear}.

\begin{thm}\label{thm:fp_mainthm}
 Let $\eps >0$ be fixed. Moreover, let $u(r;\lambda)$ be a solution to \eqref{eq:gov2dradpde}, where $u(0) = \alpha$. 
 \begin{enumerate}
  \item[(a)] If  $\lambda\leq (4+3 \sqrt{2})/{\eps^2}$,  then 
 \be\label{eq:fp_alpha_bd_1}
   \alpha \geq -1 + \frac{\left(3-2\sqrt{2}\right)}{2}\eps^2 \lambda > -1.
\ee
  \item[(b)] If $\lambda >  (4+3 \sqrt{2})/{\eps^2}$, then 
\be\label{eq:fp_alpha_bd_2}
    \alpha \geq  \frac{\eps \sqrt{\lambda(\eps^2 \lambda-8)}-\eps^2 \lambda-4}{4 (\eps^2 \lambda+1)} >-1.
\ee
 \end{enumerate}
\end{thm}
\begin{proof}
 We first prove part (a). For contradiction assume that 
\[
   \alpha < -1 + \frac{\left(3-2\sqrt{2}\right)}{2}\eps^2 \lambda.
\]
Now, since $u(r)$ is a solution of \eqref{eq:gov2dradpde}, then $u(\rho)$, where $\rho = r/\eps$, is a solution of \eqref{eq:gov2dradpde_cov}, which in turn is a solution of \eqref{eq:gov2dradpde_cov_ivp_ode} on an initial interval; then by the previous lemma, we have that $u$ can only be continued to $(\rho_1,u(\rho_1))$, where $u(\rho_1) <0$. Therefore, since $u$ is increasing on $(0,\rho_1)$, we obtain $u(\rho) < u(\rho_1) < 0$, which is a contradiction, because $u \ne 0$ at $\rho = \eps^{-1}$. Therefore, our assumption must be wrong, which implies that \eqref{eq:fp_alpha_bd_1} is true.

The proof of part (b) follows similarly, except \eqref{eq:fp_alpha_bd_2} is negated instead of \eqref{eq:fp_alpha_bd_1}.
\end{proof}

An immediate corollary to this theorem is the following.
\begin{coro}\label{coro:fp_mc} 
  Let $\eps >0$ be fixed and $u(r;\lambda)$ be a solution of \eqref{eq:gov2dradpde}. Then there exists an $\alpha_*(\eps,\lambda) >-1$ such that if $u$ is a solution of \eqref{eq:gov2dradpde}, then $u(0) > \alpha_*$.
\end{coro}

An illustration of Theorem \ref{thm:fp_mainthm} is shown in Figure \ref{fig:fp_mainthm}.

\begin{figure}[h]
\centering
\subfigure[$\eps=0.5$]{\includegraphics[width=0.425\textwidth]{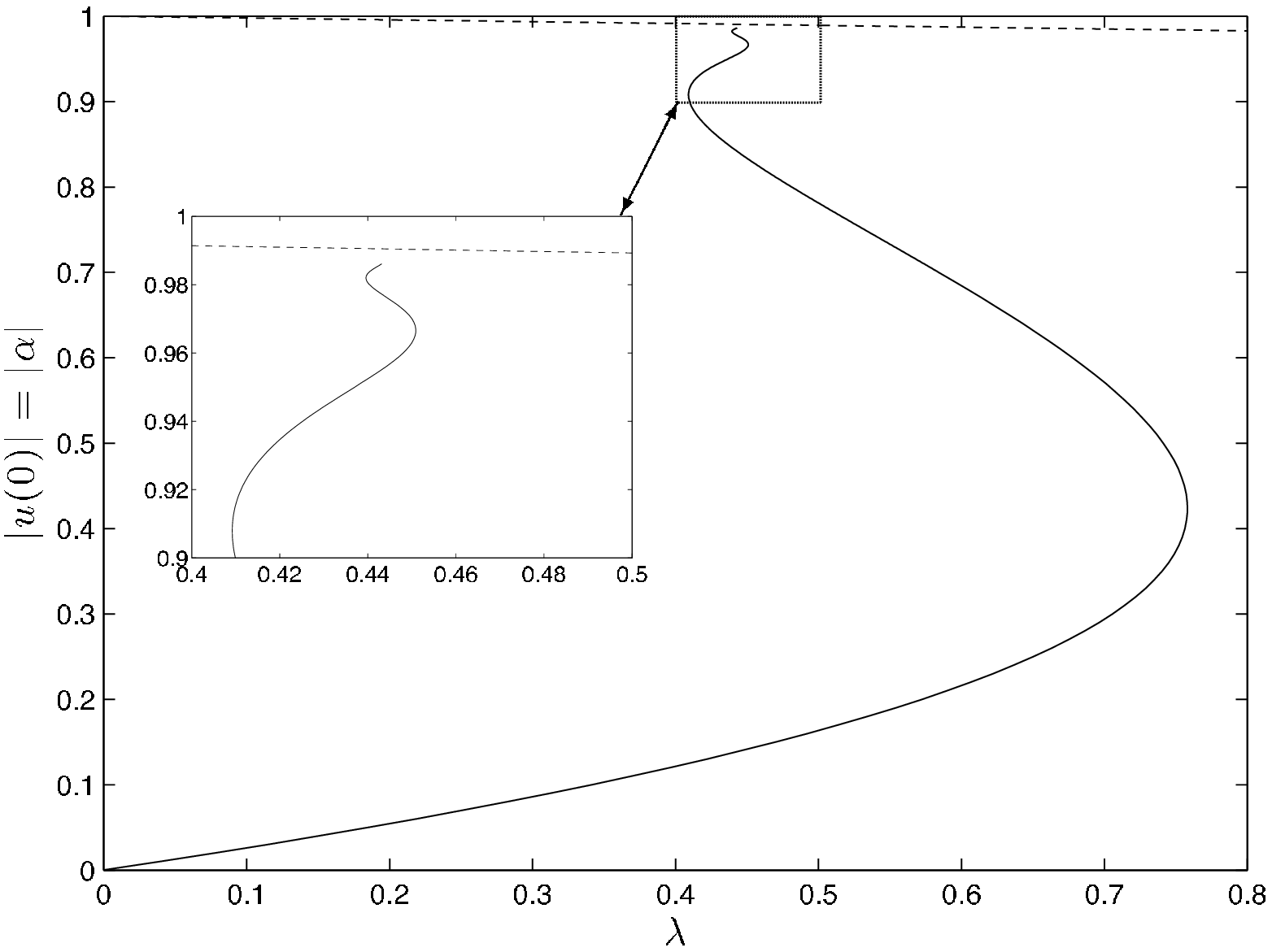}\label{fig:fp_thmfig1}}\qquad
\subfigure[$\eps=1$]{\includegraphics[width=0.425\textwidth]{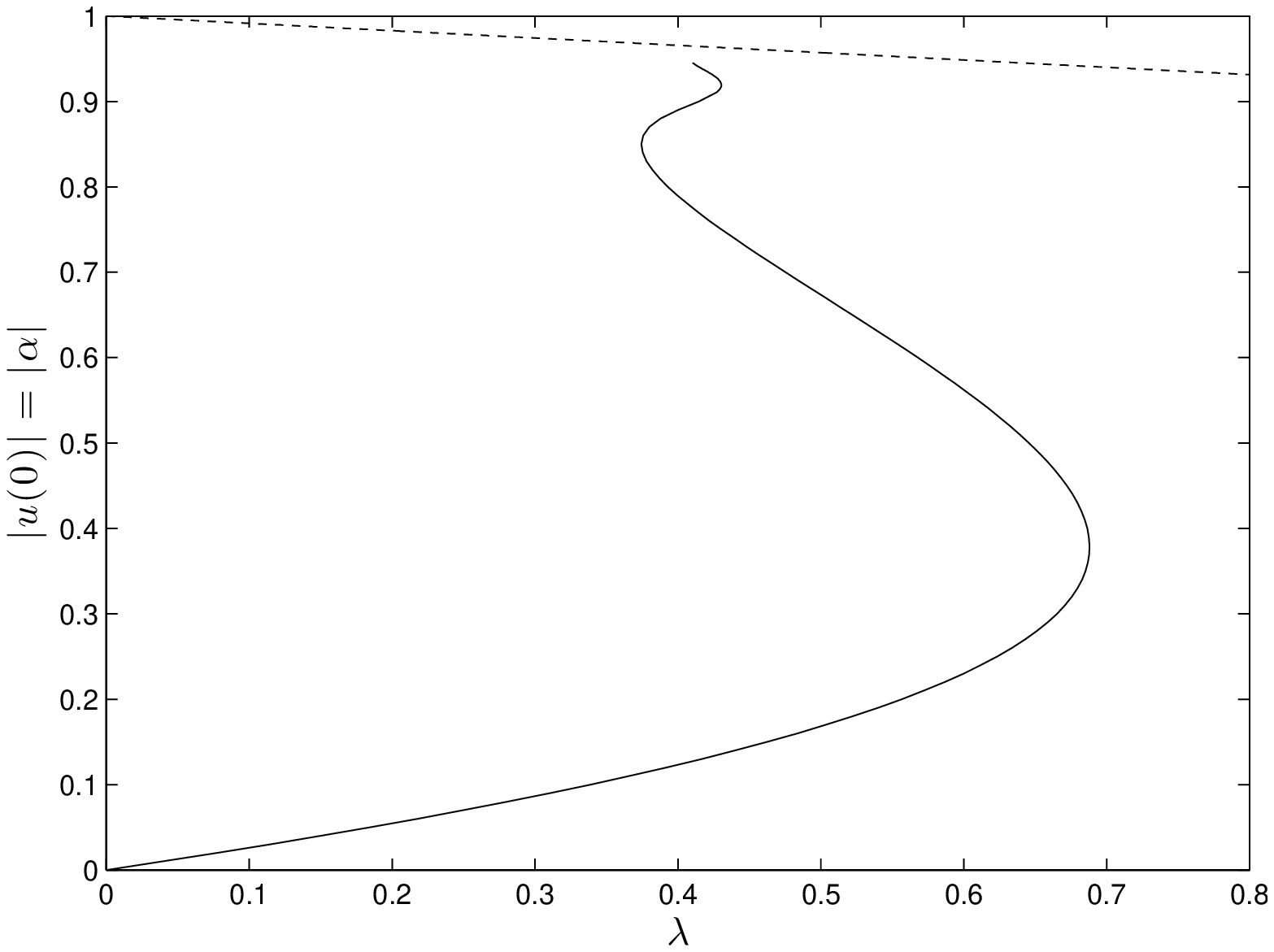}\label{fig:fp_thmfig2}}
\caption{An illustration of Theorem \ref{thm:fp_mainthm}.  On or below the dash line is the region, given in Theorem \ref{thm:fp_mainthm}, where $|u(0)|$ must be if $u$ is a solution of \eqref{eq:gov2dradpde}. As seen, this region keeps the bifurcation diagram bound away from $|u(0)|=-1$. \subref{fig:fp_thmfig1} For $\eps=0.5$; \subref{fig:fp_thmfig2} For $\eps =1$.\label{fig:fp_mainthm}}
\end{figure}

\subsection{Asymptotic analysis}\label{subsec:p_asym_analysis}

In this section, we use a similar analysis as in section \ref{sec:onedim} to analyze the upper solution branch of \eqref{eq:stdpde} in the two-dimensional unit ball for $\eps \ll 1$. To this end, we consider the equation
\begin{equation}\tag{\ref{eq:gov2dradpde}}
\frac{1}{r} \left( \frac{r u' }{\sqrt{1+\eps^2 u'^2}}\right)' = \frac{\lambda}{(1+u)^2}, \quad 0<r<1; \qquad u'(0) = u(1) = 0,
\end{equation}
in the limits $\eps \to 0^+$ and $\delta \to 0^+$, where $\delta \colonequals 1 + u(0)$. In particular, the analysis will reveal that these two small parameters must be related together in order to facilitate the matching. In these limits, \eqref{eq:gov2dradpde} is a singular perturbation problem with an inner layer at $r =0$; therefore, in the outer region away from $r=0$, we expand $u$ and $\lambda$ as
\be
u = u_0 + \eps^2 u_1 + \mathcal{O}(\eps^4), \qquad \lambda = \lambda_0 + \eps^2 \lambda_1 + \mathcal{O}(\eps^4),
\ee
and gather terms of similar order to find 
\bsub\label{main_outer}
\begin{align}
\label{main_outer_a} \Delta u_0 &= \frac{\lambda_0}{(1+u_0)^2}, \qquad u_0(0) = -1, \quad u_0(1) = 0,\\[5pt]
\label{main_outer_b} \Delta u_1 + \frac{2\lambda_0}{(1+u_0)^3} u_1 &= \frac{\lambda_1}{(1+u_0)^2} +\frac{3\lambda_0u_0'^2}{2(1+u_0)^2} -\frac{u_0'^3}{r},  \qquad u_1(1) = 0,
\end{align}
\esub
where $\Delta \ce \partial_{rr} + r^{-1}\partial_r$ denotes the two-dimensional radial Laplacian. The general solution of \eqref{main_outer} is
\begin{equation}\label{main_outer_solutions_a}
u_0 = -1 + r^{2/3}, \quad \lambda_0 = \frac{4}{9}; \qquad u_1 = \frac{\lambda_1}{3\lambda_0} r^{2/3} + A\sin(\omega\log r + \phi) 
\end{equation}
for constants $A$, $\phi${\emdash}which will be determined by matching{\emdash}and $\omega \colonequals (2 \sqrt{2})/3$. Note that $u_0'$ is not finite at $r=0$, so the condition $u'(0)=0$ will need to be enforced in a boundary layer centered around $r=0$. The value of $\lambda_1$ will eventually be fixed by the boundary condition $u_1(1) = 0$.

 Next we analyze the boundary layer near $r=0$ by introducing the inner variables 
\[
\rho = r/\gamma, \qquad  u = -1 + \delta w(\rho)
\]
where $\gamma\ll1$ is the scale of the boundary layer. Substituting these equations into \eqref{eq:gov2dradpde} gives the following equation for $w(\rho)$:
\[
  \frac{1}{\rho}\left( \frac{\rho w'}{\sqrt{1+ \eps^2\delta^2\gamma^{-2} (w')^2}}\right)' = \frac{\gamma^2}{\delta^3} \frac{\lambda}{w^2},  \quad 0< \rho < \infty; \qquad w(0) = 1, \ w'(0) = 0.
\]
A dominant balance requires that 
\begin{equation}\nonumber
\gamma = \delta^{3/2}, \qquad \frac{\eps^2}{\delta} = \delta_0
\end{equation}
where $\delta_0$ is an $\mathcal{O}(1)$ constant. We next expand $w$ as $
 w = w_0 + \littleoh(1)$, for $\delta \to 0^+$, and find that the ordinary differential equation  for $w_0(\rho)$ is 
\begin{equation}\label{mems_inner_reduced}
 \frac{1}{\rho}\left( \frac{\rho w_0'}{\sqrt{1+ \delta_0 (w_0')^2}}\right)' = \frac{\lambda_0}{w_0^2},  \quad 0< \rho < \infty; \qquad w_0(0) = 1, \ w_0'(0) = 0.
\end{equation}
The matching condition, from \eqref{main_outer_solutions_a}, provides the leading order far field behavior for $w_0$: $w_0\sim \rho^{2/3}$ as $\rho\to\infty$. To find the next order correction, we look for perturbations about this leading order form; specifically, we let $w_0 = \rho^{2/3} + v(\rho) + \ldots $  as  $\rho \to \infty$, where $v \ll \rho^{2/3}$, and retain all the linear terms to obtain the ordinary differential equation,
\[
\Delta v + \frac{2\lambda_0}{\rho^2} v +\frac{2\delta_0}{\rho^{4/3}} v' = 0,
\]
which via WKB analysis has the far field behavior 
\begin{equation}\nonumber
v \sim \til{A}(\delta_0) \sin( \omega\log\rho + \til{\phi} (\delta_0)) , \qquad \mbox{as } \rho\to\infty.
\end{equation}
Therefore $w_0 \sim \rho^{2/3} + v + \cdots$ as $\rho\to\infty$, which augments \eqref{mems_inner_reduced} to give the full specification of $w_0$ as
\bsub\label{mems_inner_full}
\begin{align}
\label{mems_inner_full_a} &\frac{1}{\rho}\left( \frac{\rho w_0'}{\sqrt{1+ \delta_0 (w_0')^2}}\right)' = \frac{\lambda_0}{w_0^2},  \quad 0< \rho < \infty; \qquad w_0(0) = 1,\ w_0'(0) = 0,\\[5pt]
\label{mems_inner_full_b} & w_0 \sim \rho^{2/3} + \til{A}(\delta_0) \sin ( \omega\log\rho + \til{\phi} (\delta_0) ) , \qquad \mbox{as} \qquad \rho\to\infty.
\end{align}
\esub
 
 To carry out matching, we introduce the intermediate variable $r_\eta = r/\eta(\eps)$, where $\eps^3 \ll \eta \ll 1$ as $\eps \to 0^+$, and the corresponding order $\bigoh(\eps^2)$ condition
 \be\label{eq:matchingcondition}
  \lim_{\substack{\eps\to 0^+ \\ r_\eta \ \mathrm{fixed}}} \frac{1}{\eps^2}\left(u_0(\eta r_\eta) + \eps^2 u_1(\eta r_\eta) + 1 -  \frac{\eps^2}{\delta_0} w_0(\delta_0^{3/2} \eta r_\eta/\eps^3) ) \right)= 0.
 \ee
From \eqref{main_outer_solutions_a} and \eqref{mems_inner_full} we have 
 \[
  \begin{alignedat}{1}
   \frac{1}{\eps^2} u_0(\eta r_\eta)  &= -\frac{1}{\eps^2} + \frac{\eta}{\eps^2} r_\eta^{2/3} \\
    u_1(\eta r_\eta) &= \frac{\lambda_1}{3\lambda_0} \eta^{2/3} r_\eta^{2/3} + A\sin(\omega\log \eta  r_\eta  + \phi) \\ 
   \frac{1}{\delta_0}w_0\left(\frac{\delta_0^{3/2} \eta r_\eta}{\eps^3}\right) & =  \frac{\eta^{2/3}}{\eps^{2}} r_\eta^{2/3} + \frac{\til{A}(\delta_0)}{\delta_0}\sin ( \omega\log \eta r_\eta +   \omega\log  \frac{\delta_0^{3/2}}{ \eps^3} + \til{\phi} (\delta_0) )  + \littleoh(1)
  \end{alignedat}
 \]
as $\eps\to 0^+$, where $r_\eta$ is fixed, and \eqref{eq:matchingcondition} yields
\[
A = \frac{\til{A}(\delta_0)}{\delta_0}, \qquad \phi = \omega\log  \frac{\delta_0^{3/2}}{ \eps^3} + \til{\phi} (\delta_0).
\]
Finally applying the boundary condition $u_1(1) = 0$ in \eqref{main_outer_solutions_a} gives
\[
  \lambda_1  =- 3\lambda_0 \frac{\til{A}(\delta_0)}{\delta_0} \sin( \omega\log  \delta_0^{3/2}\eps^{-3} + \til{\phi} (\delta_0)) 
\]
and hence, 
\[
\begin{aligned}
\lambda & = \lambda_0 + \eps^2 \lambda_1 + \bigoh(\eps^4) = \lambda_0 - \delta 3 \lambda_0 \til{A}(\delta_0) \sin( -\sqrt{2}\log \delta + \til{\phi} (\delta_0)).
\end{aligned}
\]
At this stage, we may fix the value of $\eps$ in the main equation \eqref{eq:gov2dradpde} and write $\delta_0 = \eps^2/\delta$ with $\eps^2$ fixed but still $\bigoh(\delta)$. This leads to the following asymptotic result regarding the upper solution branch of the bifurcation curve for \eqref{eq:gov2dradpde}. 

\begin{pr}\label{pr:2drad_np}
 For solutions of \eqref{eq:gov2dradpde}, there is a regime where both $\eps \ll 1$ and $\delta \ll 1$, with $\eps^2/\delta = \mathcal{O}(1)$, such that the upper solution branch of the bifurcation curve has the asymptotic parameterization
\bsub\label{PR1}
\begin{equation}\label{PR1_a}
|u(0)|= 1- \delta, \qquad \lambda = \frac{4}{9} - \delta  \frac{4}{3} \til{A}\left(\frac{\eps^2}{\delta}\right) \sin\left[ -\sqrt{2}\log \delta + \til{\phi} \left(\frac{\eps^2}{\delta}\right)\right] + \mathcal{O}(\delta^2).
\end{equation}
where $\til{A}(\delta_0)$ and $\til{\phi}(\delta_0)$ are functions determined by the initial value problem
\begin{align}
\label{PR1_b} &\frac{1}{\rho}\left( \frac{\rho w_0'}{\sqrt{1+ \delta_0 (w_0')^2}}\right)' = \frac{4}{9}w_0^{-2},  \quad 0< \rho < \infty; \qquad w_0(0) = 1,\ w_0'(0) = 0\\
\label{PR1_c} &w_0 = \rho^{2/3} + \til{A}(\delta_0)\sin\left(\frac{2\sqrt{2}}{3}\log\rho + \til{\phi}(\delta_0)\right) + \littleoh(1) \qquad \rho\to\infty.
\end{align}
\esub
\end{pr}

The asymptotic parameterization \eqref{PR1_a} of the upper solution branch of \eqref{eq:gov2dradpde} appears outwardly to be defined for $|u(0)|$ arbitrarily close to $1$, potentially contradicting Corollary \ref{coro:fp_mc}. However, the parameterization assumes that the quantities $\til{A}(\eps^2/\delta)$ and $\til{\phi}(\eps^2/\delta)$ are well defined as $\delta\to0^{+}$ and so one can expect that $\til{A}(\delta_0)$ and $\til{\phi}(\delta_0)$ will not be defined for $\delta_0$ sufficiently large. Therefore before observing the predictive accuracy of \eqref{PR1}, let us first consider the existence of solutions to \eqref{PR1_b}, for $\delta_0$ sufficiently large.

To do so, we follow a similar procedure outline in previous section and introduce a change of variables{\emdash}specifically, $\xi = \rho/\sqrt{\delta_0}$, with $w_0(\rho) = v(\xi)${\emdash}so that \eqref{PR1_b} becomes
\begin{align}\label{eq:ip_ivp_rescaled}
 &\frac{1}{ \xi}\left( \frac{ \xi v'}{\sqrt{1+(v')^2}}\right)' = \frac{4\delta_0}{9}  v^{-2},  \quad 0< \xi < \infty; \qquad v(0) = 1,\ v'(0) = 0.
\end{align}
Hence, the mean curvature operator is isolated on the left-hand side, and \eqref{eq:ip_ivp_rescaled} yields the geometric representation  
\bsub\label{eq:ip_ivp_angle}
\be\label{eq:ip_ivp_angle_ode}
 \frac{\left(\xi \sin\psi\right)'}{\xi}= \frac{4\delta_0}{9}  v^{-2}
\ee
where $\psi$ is the angle of inclination of the solution curve. Noting that \eqref{eq:ip_ivp_rescaled} and \eqref{eq:ip_ivp_angle_ode} are equivalent on any interval in which $v'(\xi)$ is bounded, we look at \eqref{eq:ip_ivp_angle_ode}, coupled with the initial condition 
\be\label{eq:ip_ivp_angle_ic}
 v(0)=1,
\ee
\esub
to study the nonexistence of solutions of \eqref{eq:ip_ivp_rescaled}. Also note that if $v$ satisfies \eqref{eq:ip_ivp_angle}, then the condition $v'(0)=0$ is redundant, which can be seen by integrating \eqref{eq:ip_ivp_angle_ode} and then taking $\xi \to 0^+$. 

\begin{lem}\label{lem:ip_increasing_v}
 If $v$ satisfies \eqref{eq:ip_ivp_angle} on $[0,a)$, then $\sin \psi > 0$ for $ \xi \in (0,a)$, which implies that $v$ is increasing on that interval. Furthermore, we have the following bound:
\be\label{eq:ip_increasing_v1}
 \frac{2\delta_0}{9v(\xi)^2}< \frac{\sin\psi(\xi)}{\xi} < \frac{2\delta_0}{9}.
\ee
\end{lem}
\begin{proof}
 Integrating \eqref{eq:ip_ivp_angle_ode} yields
\[
 \sin\psi(\xi)= \frac{4\delta_0}{9}  \frac{1}{\xi}\int_0^\xi \eta \  v(\eta)^{-2} \ud\eta > 0.
\]
for all $\xi \in (0,a)$, and the results follow as in Lemma \ref{lem:fp_ivp_increasing}.
\end{proof}

Now, we can prove the main lemma which leads to our desired main result for the nonexistence of solutions of \eqref{PR1_b}.

\begin{lem}\label{lem:ip_main}
 Assume that $v$ is a solution of \eqref{eq:ip_ivp_angle_ode} on an initial interval with $v(0)=1$. If
 \be\label{eq:delta0_bound}
  \delta_0 \geq \bar{\delta}_0 \ce \frac{9(2\sqrt{2}+3)}{2},
 \ee
 then there exists a value $\xi_1$ in which $v(\xi)$ cannot be continued beyond as a solution of \eqref{eq:ip_ivp_angle}. Furthermore, at this point $(\xi_1,v_1)$, where $v_1 \ce v(\xi_1)$, the slope of the solution curve is vertical and the following bounds hold:
\begin{align}
  \frac{9}{2\delta_0 } <\xi_1 & \leq \frac{9}{2 \delta_0} \left(\frac{6 \delta_0 -9  - \sqrt{4 \delta_0^2  - 108 \delta_0 + 81}}{4 \delta_0}\right)^2, \\
  1 + \frac{9}{2\delta_0} < v_1 &< \frac{6 \delta_0 -9  - \sqrt{4 \delta_0^2  - 108 \delta_0 + 81}}{4 \delta_0}.
\end{align}
\end{lem} 
\begin{proof}
By Lemma \ref{lem:ip_increasing_v}, we know that on an initial segment $v$ is increasing and consequently, may use it $v$ as an independent variable; thus, \eqref{eq:ip_ivp_angle} gives
\be\label{eq:ip_c2cos}
 \frac{\sin\psi}{\xi} - (\cos{\psi})_v =\frac{4\delta_0}{9}  v^{-2},
\ee
Then integrating with respect to $v$ and using \eqref{eq:ip_increasing_v1} yields
\[
 1- \cos{\psi(v)}>\frac{4\delta_0}{9}  \left(1- \frac{1}{v}\right) - \frac{2\delta_0}{9} (v-1),
\]
or
\[
 \frac{2\delta_0}{9} \left(v+ \left[\frac{9}{2\delta_0}-3\right]+\frac{2}{v} \right)>\cos{\psi(v)};
\]
hence if  \eqref{eq:delta0_bound} is true, then a vertical slope exists at a value $v_1$ where 
\[
 v_1 < \frac{6 \delta_0 -9  - \sqrt{4 \delta_0^2  - 108 \delta_0 + 81}}{4 \delta_0}.
\]

First from \eqref{eq:ip_increasing_v1}, we see that $v$ can be continued to at least $\xi={9}/(2\delta_0)$, which implies that $\xi_1 > 9/(2 \delta_0)$.

Next, we compare the solution to
\begin{equation}\nonumber
 W(\xi)\ce 1 + \beta -\sqrt{\beta^2 -\xi^2}, \qquad \beta \ce  \frac{9}{2\delta_0 } \left(\frac{6 \delta_0 -9  - \sqrt{4 \delta_0^2  - 108 \delta_0 + 81}}{4 \delta_0}\right)^2.
\end{equation}
Note that 
\begin{equation}\nonumber
\begin{aligned}
	(\calig{F}(\xi,W'))'  = \frac{2}{\beta} \xi  < \frac{4 \delta_0}{9}\xi \, v_1^{-2}  <  \frac{4 \delta_0}{9}\xi \, v^{-2}= (\calig{F}(\xi,v'))' 
\end{aligned}
\end{equation}
for $\xi\in(0,\min\{\xi_1,\beta\})$ and $W(0) = 1$. Therefore, by Lemma \ref{lem:fp_comp},
\begin{equation}\nonumber
  {W}(\xi) < {v}(\xi) , \qquad  W'(\xi) < v'(\xi), \qquad \xi \in(0,\min\{\xi_1,\beta\}).
\end{equation}
Now if $\beta < \xi_1$, then $W'(\xi) < v'(\xi)$, for $\xi \in(0,\beta)$, which implies that $v'(\beta) = \infty$ and yields a contradiction. Therefore, we must have $\xi_1 \leq \beta$.

To get the last bound we note that for the solution graph $(\xi,v)$ of \eqref{eq:ip_ivp_angle} $v$ may be used as the independent variable.  Hence
using  \eqref{eq:ip_increasing_v1} in \eqref{eq:ip_c2cos} gives
\be\label{eq:ip_ind_v_angle}
 (\cos{\psi})_v >- \frac{2\delta_0}{9}  v^{-2}.
\ee
Then in introducing a comparison surface
\be\label{eq:ip_ind_v_comp_surf}
 (\til{\xi}(v),v), \qquad  \til{\xi} (v) \ce \sqrt{\frac{81}{4 \delta_0^2} - \left[1 + \fr{9}{2 \delta_0} - v\right]^2} 
\ee
for $v \in(1,1+9/(2\delta_0))$ with corresponding angle of inclination $\varphi$, we obtain  $(\cos \varphi)_v = -{9}/({2 \delta_0})$; thus, from \eqref{eq:ip_ind_v_angle}
\[
 (\cos{\psi})_v > (\cos\varphi)_v
\]
for each $v \in(0, \min\{v_1, 1+9/(2\delta_0)\})$, which implies  $\cos{\psi} >\cos\varphi$ for each $v$ in that interval, and therefore, the solution $(\xi,v)$ of \eqref{eq:ip_ivp_angle} can be continued vertically until the comparison surface \eqref{eq:ip_ind_v_comp_surf} becomes vertical. That is, 
\begin{equation}\nonumber
 v_1 > w(9/(2\delta_0)) = 1+9/(2\delta_0).
\end{equation}
\end{proof}

Therefore, since \eqref{eq:ip_ivp_angle} and \eqref{eq:ip_ivp_rescaled} are equivalent on $(0,\xi_1)$, then the derivative of \eqref{eq:ip_ivp_rescaled} also blows-up in finite time, which after changing variables back to $\rho$ yields the following result concerning \eqref{PR1_b}. 
\begin{thm}
 There exists a value $\delta_0^*$ such that for $\delta_0 \geq \delta_0^*$ no solutions of \eqref{PR1_b} exist. Furthermore, 
 \begin{equation}\nonumber
 \delta_0^* \leq  \bar{\delta}_0 \ce \frac{9(2\sqrt{2}+3)}{2}\approx 26.2279.
 \end{equation}
\end{thm}
\textit{Remark.} An integration of the \eqref{PR1_b} gives $\delta^{\ast}_0 \approx 18.142468.$ indicating that $\bar{\delta}_0$ is not a particularly tight upper bound on $\delta^{\ast}_0$.

As a result of this theorem, an expansion for the dead-end point $(\lambda_*(\eps),|\alpha_*(\eps)|)$ can now be  extracted from \eqref{PR1_a}; specifically, since \eqref{PR1_b} has no solutions for $\delta_0 \geq \delta_0^*$, the asymptotic approximation \eqref{PR1_a} fails at $\eps^2/\delta  = \delta_0^* $, or $\delta = \eps^2/\delta_0^* $. Therefore, using these values in \eqref{PR1_a}, the following asymptotic result for the dead-point of \eqref{eq:gov2dradpde} is established.
\vspace{0.25in}
\begin{pr}\label{pr:2drad_de_pt}
For $\eps \ll 1$, the dead-end point of the upper solution branch of the bifurcation curve of \eqref{eq:gov2dradpde} has the asymptotic expansion 
\be\label{eq:asymdeadendpt}
\begin{aligned}
  |\alpha_*(\eps)|& = 1 -\frac{\eps^2}{\delta_0^*} + \bigoh(\eps^4),\\
 \lambda_*(\eps) &= \frac{4}{9} - \eps^2  \frac{4}{3} \frac{\til{A}\left(\delta_0^*\right)}{\delta_0^*} \sin\left[ -\sqrt{2}\log \frac{\eps^2}{\delta_0^*} + \til{\phi} \left(\delta_0^*\right)\right] + \bigoh(\eps^4).
 \end{aligned}
\ee
\end{pr}

We remark that the ability of Principal Result \ref{pr:2drad_de_pt} to predict the dead-end point associated with radially symmetric solutions of \eqref{eq:govpde} in the case $n=2$, \emph{is not matched} by the asymptotic analysis leading to Principal Result \ref{pr:1d_conc} for the $n=1$ case. This discrepancy suggests that the dead-end phenomena exhibited in both the $n=1$ and $n=2$ are not qualitatively similar.

In order to study the quantitative accuracy of Principal Result \ref{pr:2drad_de_pt}, it is necessary to obtain the functions $\til{A}(\delta_0)$ and $\til{\phi}(\delta_0)$, which are readily acquired by solving \eqref{PR1_b} numerically, then subtracting off the growth term $\rho^{2/3}$ and applying a least squares fit to the remainder (see Figure \ref{fig:A_phi}). In Figure \ref{fig:Comp}, comparisons of the full numerical solution of the upper branch of the bifurcation curve and asymptotic prediction of \eqref{PR1_a}  are displayed; furthermore, the agreement is observed to be very good.
\begin{figure}[h]
\centering
\subfigure[$\til{A}(\delta_0)$]{\includegraphics[width=0.425\textwidth]{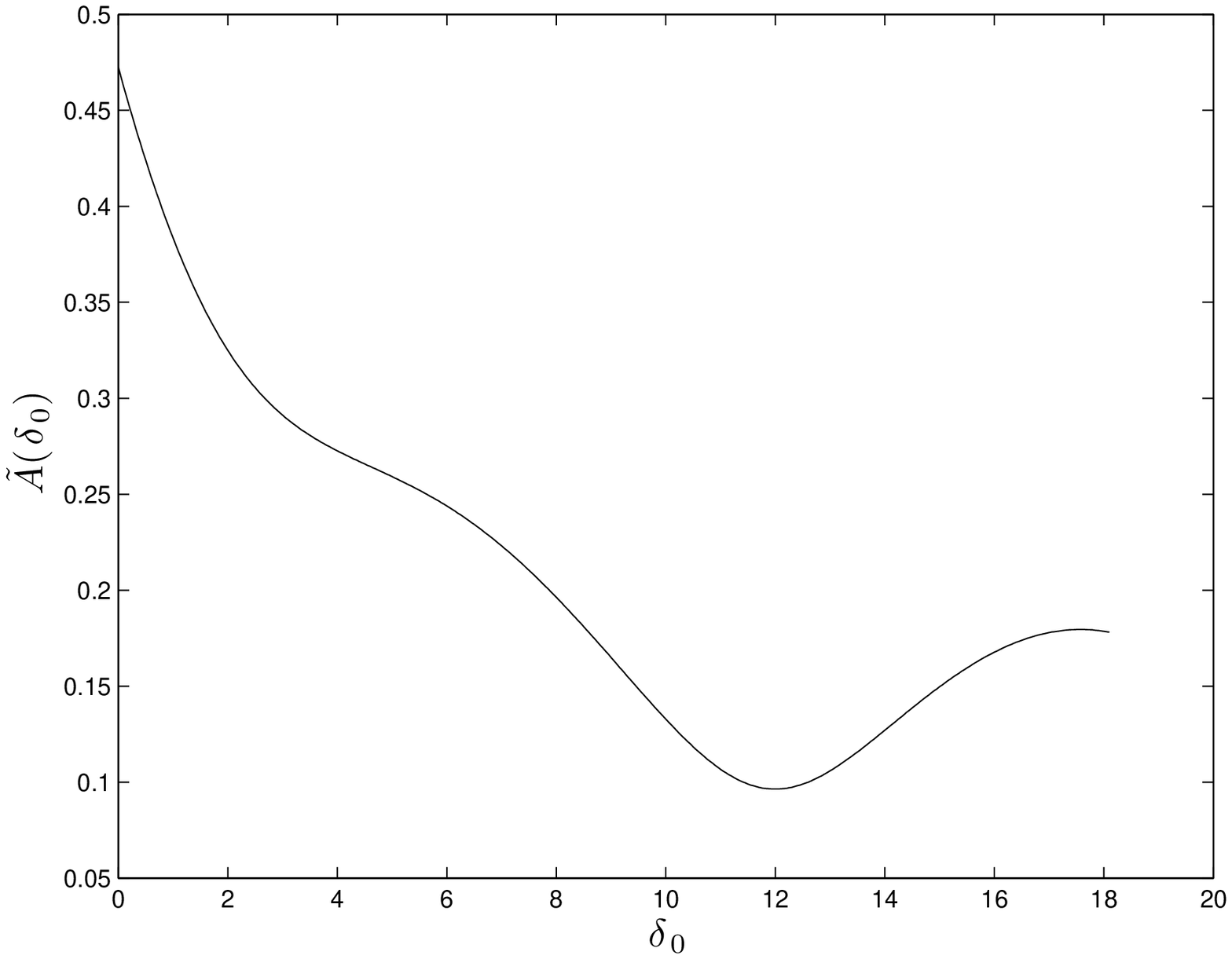}}\qquad
\subfigure[$\til{\phi}(\delta_0)$]{\includegraphics[width=0.425\textwidth]{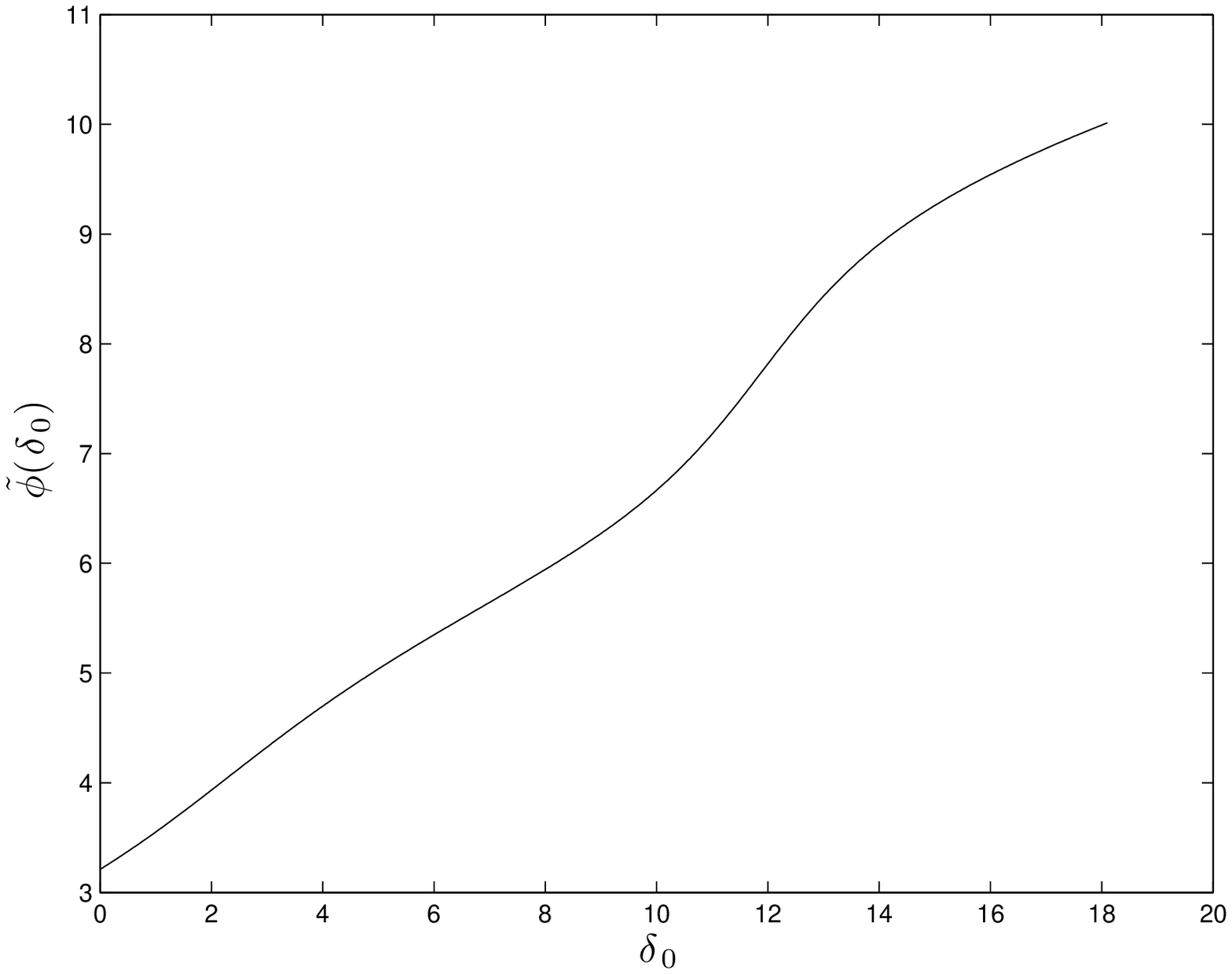}}
\caption{Graphs of $\til{A}(\delta_0)$ and $\til{\phi}(\delta_0)$ against $\delta_0$. The numerical integration fails abruptly at roughly $\delta_0 = \delta^{\ast}_0 \approx 18.142468$. \label{fig:A_phi} }
\end{figure}

\begin{figure}[H]
\centering
\subfigure[$\eps^2 = 0.025$]{\includegraphics[width=0.4125\textwidth]{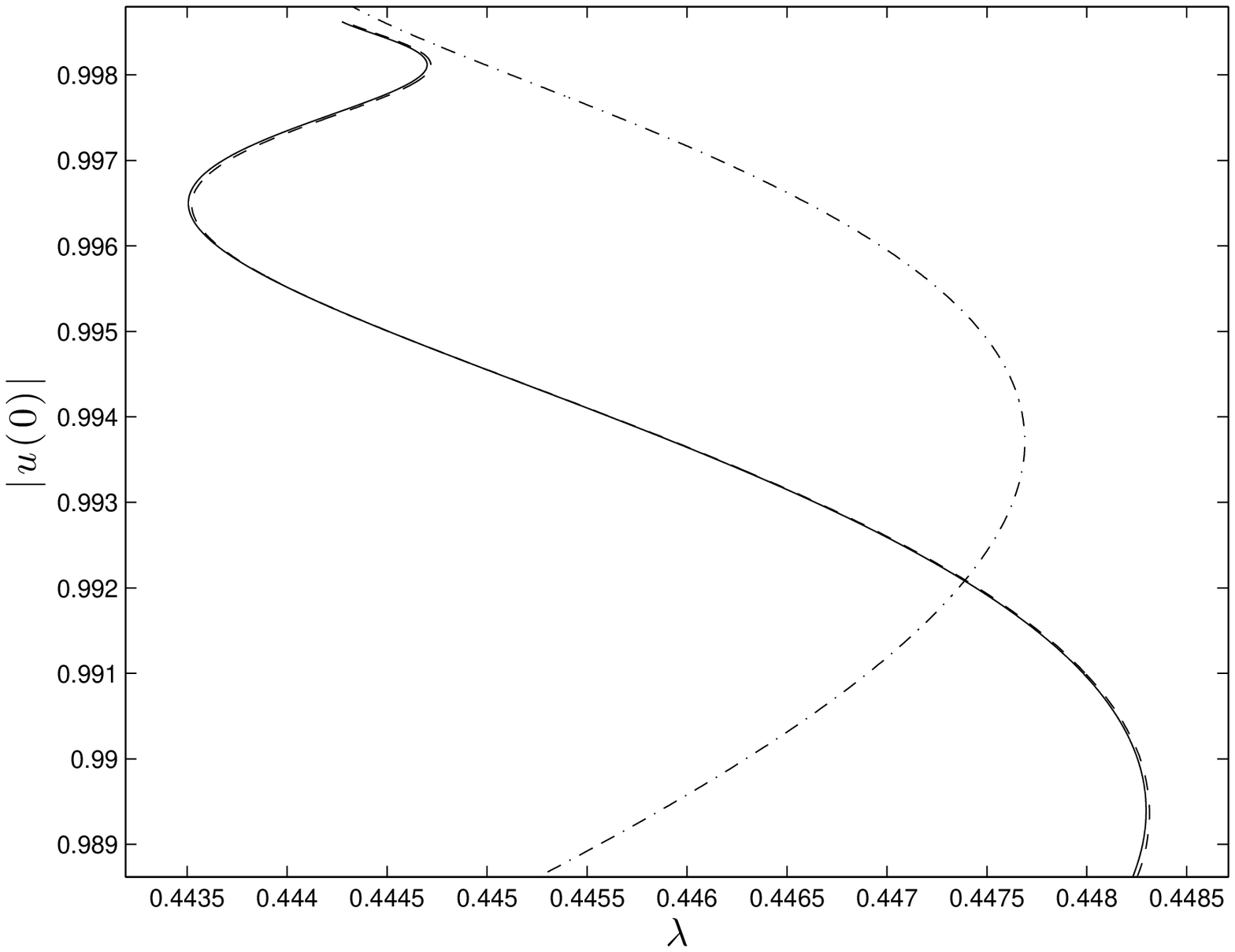}}\qquad
\subfigure[$\eps^2 = 0.1$]{\includegraphics[width=0.425\textwidth]{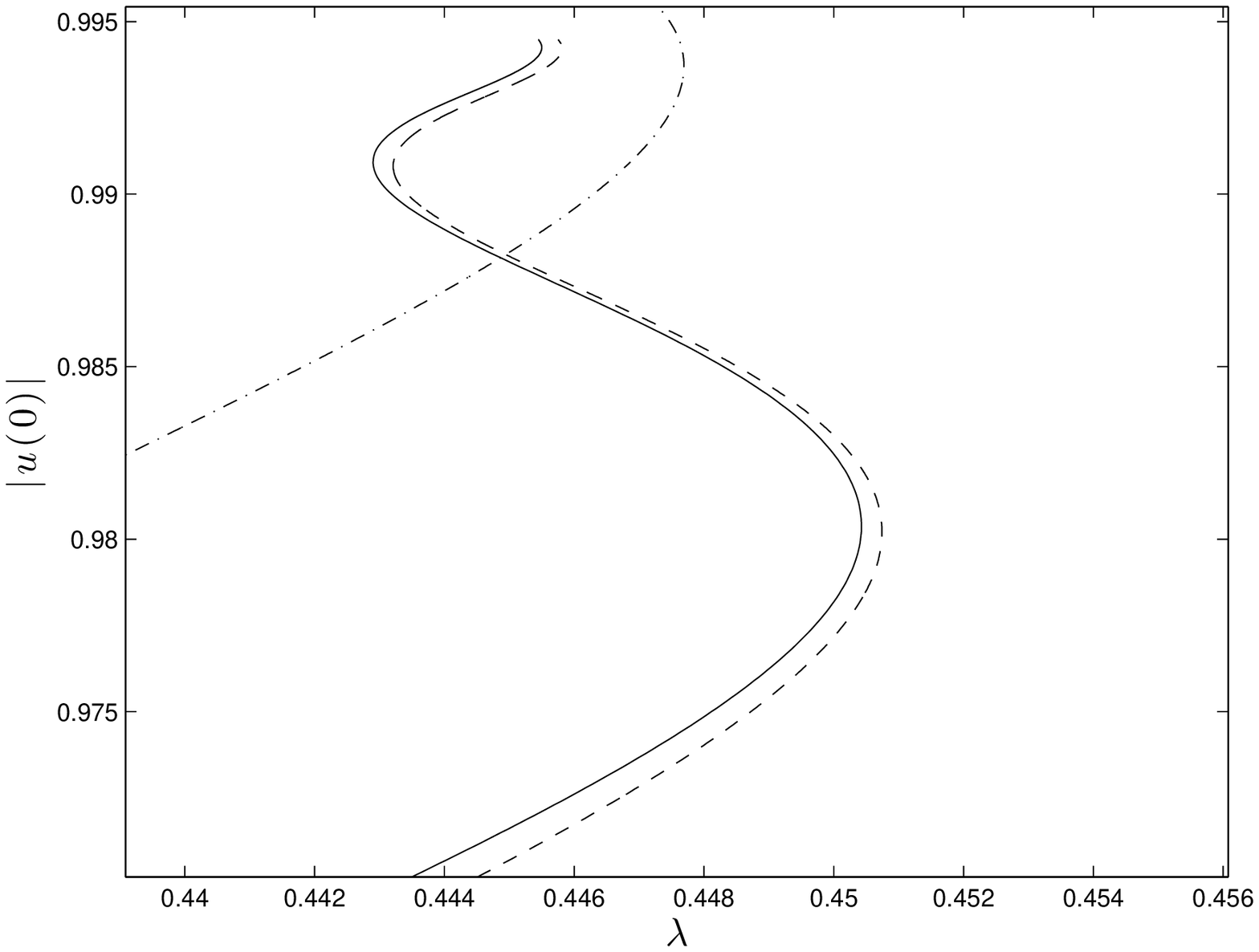}}
\caption{Comparison of the full numerical solution of the bifurcation curve for \eqref{eq:gov2dradpde} (solid) with asymptotic formula \eqref{PR1_a} (dashed) and the solution curve of \eqref{eq:gov2dradpde} for $\delta =0$ (dash dot). \label{fig:Comp} }
\end{figure} 

  In Figure \ref{fig:crit_comp}, a comparison of the numerical and asymptotic values for the location of the dead-end point is shown; note that agreement is very good, as in each case the asymptotic error is $\mathcal{O}(\eps^4)$.

In Figure \ref{fig:blowup_inner_a}, the numerical (solid) and global asymptotic (dashed) solutions of \eqref{eq:gov2dradpde} at the dead-end point for $\eps^2=0.2$ are displayed. As expected (see section \ref{subsection:nc}), the tangent of the solution curve is almost vertical, indicating that the derivative of the solution is becoming unbounded. Indeed, when solutions $w_0'(\rho;\delta_0)$ of \eqref{PR1_b} are plotted for several $\delta_0$ a blow-up in $w_0'(\rho;\delta_0)$ as $\delta_0\to\delta_0^*$ is observed (see Figure \ref{fig:blowup_inner_b}). This suggests that beyond the dead-end point solutions of \eqref{eq:gov2dradpde} cannot be represented by a function of a single variable. Therefore in the next section \eqref{eq:gov2dradpde} is re-parameterized in terms of arc length along the solution curve $s$, and consequently becomes a system of coupled ODES.  An asymptotic study of this coupled system reveals that multivalued solutions of \eqref{eq:govpde} are present beyond the dead-end point of the bifurcation diagram.

\begin{figure}[h]
\centering
\subfigure[$\eps^2/\delta_0^* = 1-|\alpha_*(\eps)|$]{\label{eq:npradial_alpha_star}\includegraphics[width=0.425\textwidth]{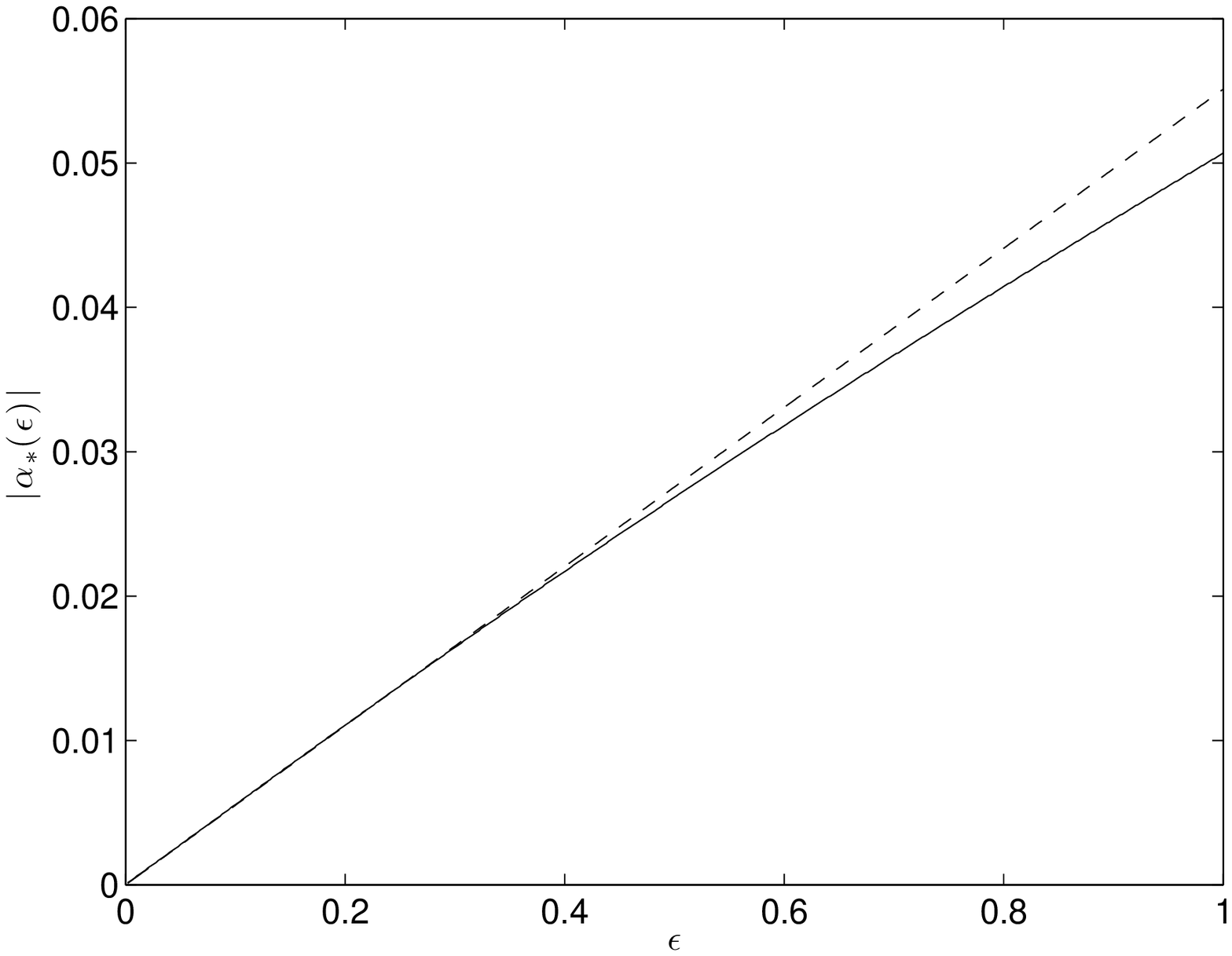}}\qquad
\subfigure[$\lambda_*(\eps)$]{\label{fig:npradial_lambda_star}\includegraphics[width=0.415\textwidth]{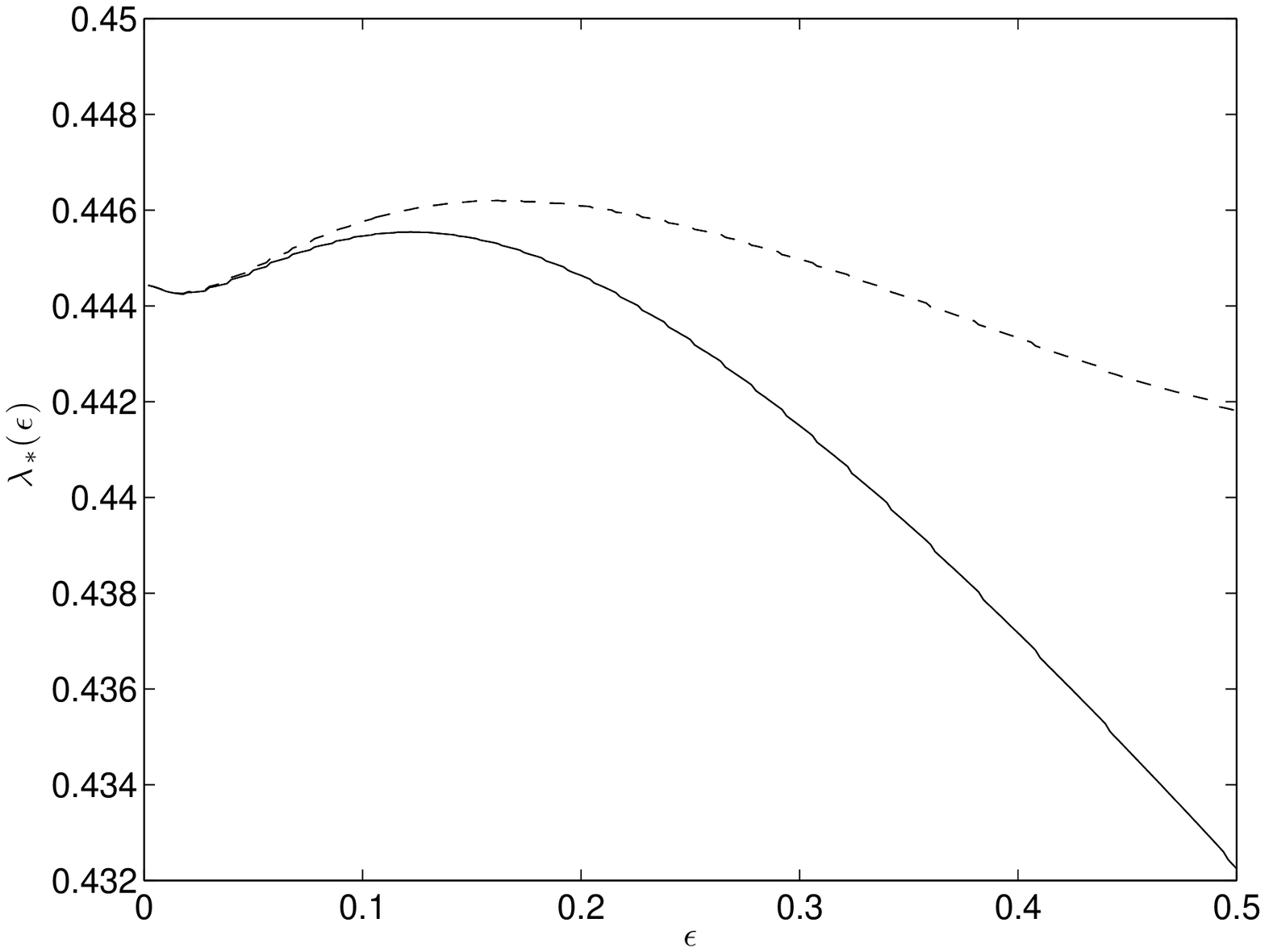}}
\caption{Comparison of the asymptotic prediction, \eqref{eq:asymdeadendpt}, (dashed line) of the dead-end point $(\lambda_*(\eps),|\alpha_*(\eps)|)$ with full numeric computations (solid) for: \subref{eq:npradial_alpha_star} the  $\bigoh(\eps^2)$ correction of $|\alpha_*(\eps)|$;   \subref{fig:npradial_lambda_star} $\lambda_*(\eps)$. Notice that the scale on the $y$-axis of the right figure is quite fine and so the agreement for $\lambda_*(\eps)$ is in fact better than the figures makes it appear. \label{fig:crit_comp} }
\end{figure}

\begin{figure}[H]
\centering
\subfigure[Global Approximation -- $\eps^2 = 0.2$.]{\includegraphics[width=0.425\textwidth]{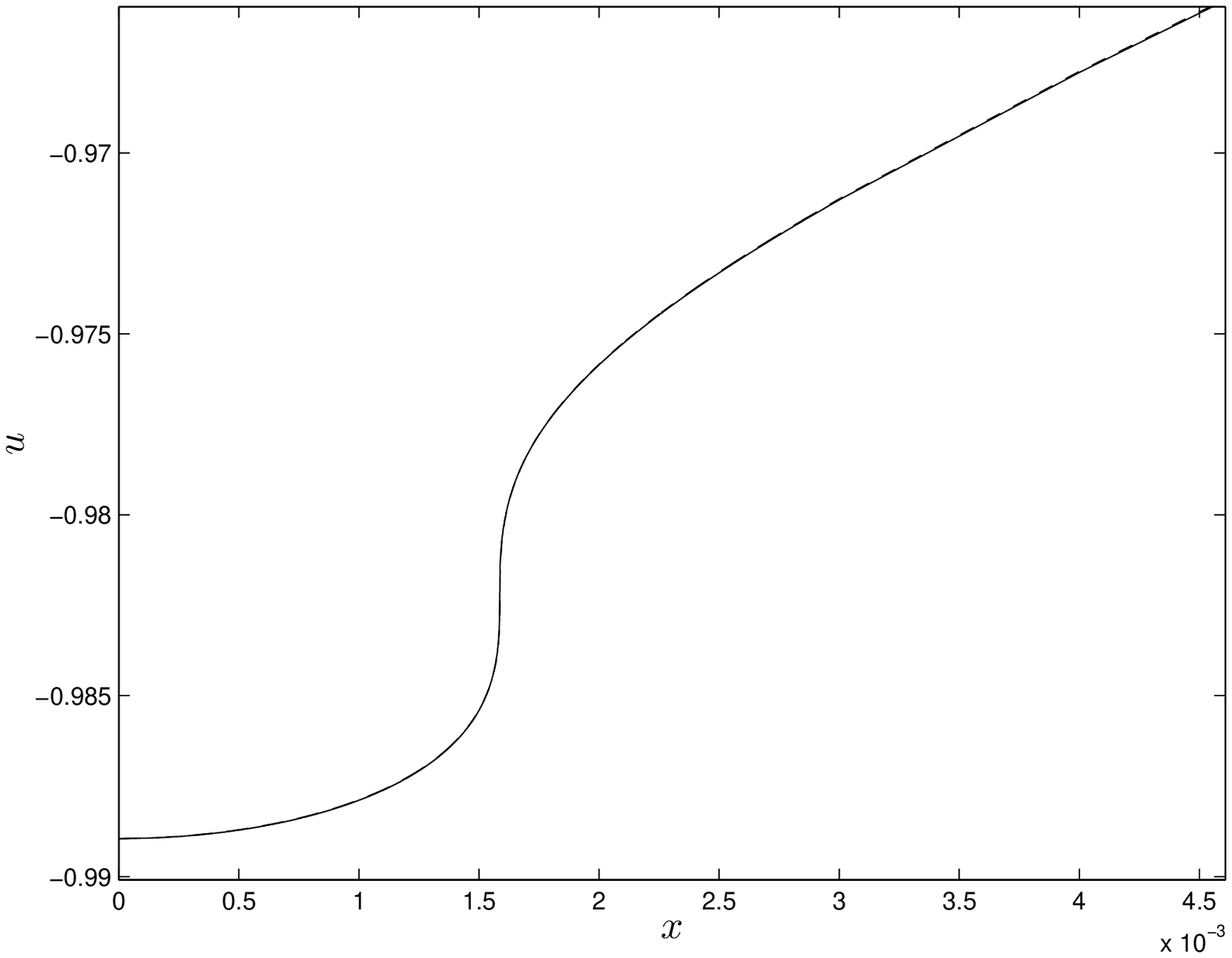}\label{fig:blowup_inner_a}}\qquad
\subfigure[Blow up of $w'_0(\rho;\delta_0)$ as $\delta_0\to\delta_0^{\ast}$.]{\includegraphics[width=0.425\textwidth]{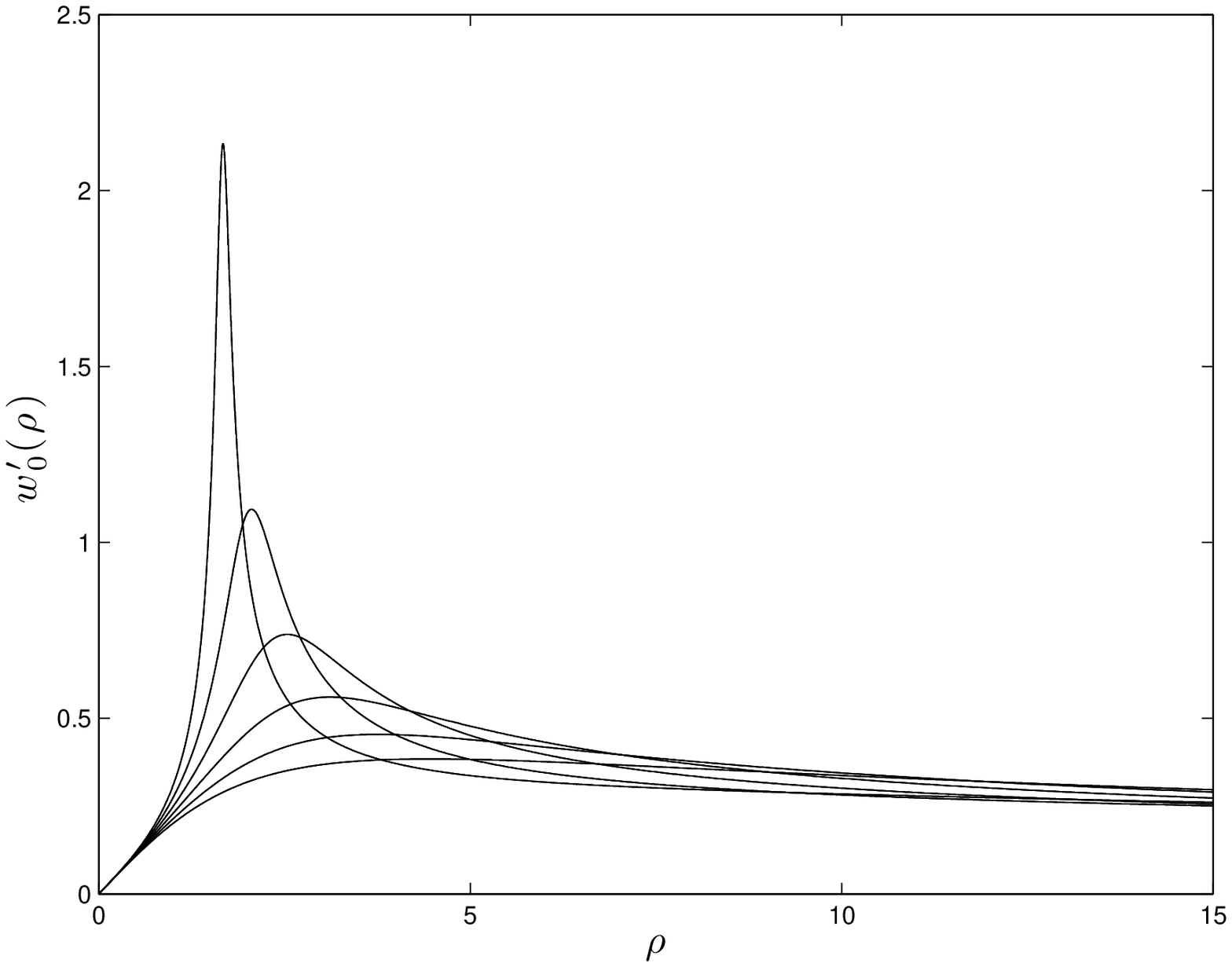}\label{fig:blowup_inner_b}}
\caption{\subref{fig:blowup_inner_a} The numerical (solid) and global asymptotic (dashed) profile in the boundary layer. The curve is almost vertical indicating that the derivative is becoming infinite. \subref{fig:blowup_inner_b} The derivative function $w'_0(\rho;\delta_0)$ plotted for several $\delta_0$. As $\delta_0\to \delta_0^{*}$, it observed that $w_0'(\rho;\delta_0)$ appears to develop a singularity at a finite $\rho^*$. \label{fig:blowup_inner}}
\end{figure}

\subsection{Arc length asymptotic analysis}\label{sect:arclen2d}
In this section, we analyze parameterized solutions $(r(s),u(r(s))) = (r(s),z(s))$ of \eqref{eq:gov2dradpde} where the arc length $s$ satisfies the relationship $ d r^2 + \eps^2 d z^2 = ds^2$. When reformulated in terms of $r(s)$ and $z(s)$, \eqref{eq:gov2dradpde} becomes a coupled set of ODEs:
\begin{equation}\tag{\ref{eq:arclengthsystem}}\label{eq:paramemsode}
 \begin{alignedat}{2}
  &r''  = -\frac{\eps^2 \lambda z'}{(1+z)^2}+\frac{\eps^2 (z')^2}{r}, \qquad  z'' + \frac{r' z'}{r} = \frac{\lambda r'}{(1+z)^2}, & \qquad & 0<s<\ell; \\
 & r(0)=0, \ r'(0)=1, \ r(\ell) = 1, \ z'(0)=0, \ z(\ell) = 0,
 \end{alignedat}
\end{equation}
where $0< \eps^2 \ll 1$ and $\lambda$ and $\ell$ are unknown parameters to be determined. Furthermore, to facilitate the analysis of the upper solution branch, we impose the condition $z(0) = -1 + \delta$ and study \eqref{eq:paramemsode} in the limits $\eps\to 0^+$ and $\delta \to 0^+$, where the relationship  between these two small parameters is to be determined. In the outer region away from $s=0$, we expand $r$, $z$, $\lambda$ and $\ell$ as 
\begin{subequations}\label{eq:outerexpan} 
\begin{equation} \label{eq:outerexpansub1}
 \begin{aligned}
  r(s;\eps) & = r_0(s) + \eps^2 r_1(s) + \eps^4 r_2(s) + \bigoh(\eps^6), \\
   z(s;\eps) & = z_0(s) + \eps^2 z_1(s) + \eps^4 z_2(s) +  \bigoh(\eps^6),
 \end{aligned}
 \ee
 and 
 \be
 \lambda(\eps)  = \lambda_0 + \eps^2 \lambda_1  + \eps^4 \lambda_2 + \bigoh(\eps^6),  \qquad \ell(\eps)  = \ell_0 + \eps^2 \ell_1 + \eps^4 \ell_2 + \bigoh(\eps^6), \label{eq:outer_def_lambda_ell}
\end{equation}
\end{subequations}
which upon substituting into \eqref{eq:paramemsode} gives 
\begin{equation}\label{eq:outord1}
 \begin{alignedat}{1}
  & r_0''  = 0, \qquad z_0'' + \frac{r_0' z_0'}{r_0} =\frac{\lambda_0 r_0'}{(1+z_0)^2},  \qquad  0<s<1; \\
 & r_0(0)=0, \  r_0'(0)=1, \ r_0(\ell_0) = 1, \ z(0)=-1, \ z_0'  (0)=0, \ z_0(\ell_0) = 0,
 \end{alignedat}
\end{equation}
at order $\bigoh(1)$. Therefore in solving \eqref{eq:outord1} we find
\be\label{eq:outord1soln}
 r_0(s)=s , \quad  z_0(s) = -1 + s^{2/3}, \quad \lambda_0 = \frac{4}{9},  \quad \ell_0 =1.
\ee
However, $z_0'(0) \not=0$ and therefore, we have a boundary layer at $s = 0$ for $z(s)$. Next, from \eqref{eq:paramemsode}, \eqref{eq:outerexpan} and \eqref{eq:outord1soln} we have 
\begin{equation}\label{eq:outord2}
 \begin{alignedat}{1}
  &r_1''  = \frac{4}{27} s^{-5/3}, \qquad  z_1'' + \frac{1}{s} z_1' + \frac{8}{9s^2} z_1 = \frac{2}{3} \frac{r_1}{s^{7/3}} + \frac{9 \lambda_1 - 2 r_1'}{9 s^{4/3}} ,\\
 & r_1(0)=0, \  r_1'(0)=0, \ r_1(1) = -\ell_1, \ z_1(1) = -\frac{2}{3}\ell_1,
 \end{alignedat}
\end{equation}
at order $\bigoh(\eps^2)$. The solution for $r_1(s)$ is 
\be\label{eq:ord2routersoln} 
 r_1(s) = \left(\frac23 -\ell_1\right) s -\frac{2}{3} s^{1/3} 
\ee
where the condition $r_1'(0)=0$ will be enforced in the boundary layer at $s=0$. Then using \eqref{eq:ord2routersoln} in \eqref{eq:outord2}, we deduce
\[
 \begin{alignedat}{1}
  &z_1'' + \frac{1}{s} z_1' + \frac{8}{9s^2} z_1 =-\frac{32}{81s^2} + \left( \lambda_1 + \frac{8}{27} - \frac{4}{9} \ell_1\right)\frac{1}{s^{4/3}},\qquad z_1(1) = 0,
 \end{alignedat}
\]
which upon solving gives
\be\label{eq:outordereps2soln}
 z_1(s) = \frac{(27\lambda_1 + 8 -12 \ell_1)}{36}s^{2/3} - \frac{4}{9} +A_1 \sin\left( \omega \log{s} + \phi_1 \right),  
\ee
where $\omega \ce 2\sqrt{2}/3$. Here, $A_1$ and $\varphi_1$ are constants that will be determined by matching and the value of $\lambda_1$ will be determined later by applying the condition $z_1(1)=-2 \ell_1/3$. In order to fix the value of $\ell_1$, an expansion to higher order is required. Accordingly, we use \eqref{eq:outerexpan} in \eqref{eq:paramemsode} to find a system of ODEs at order $\bigoh(\eps^4)$ (see \ref{app:secondordouterODE}), which upon solving gives 
\[
  r_2(s) =  \frac{K_1}{s^{1/3}} + K_2 +\frac{6 \ell_1 -4 + 36 A_1 \sin \phi_1 }{27} s^{1/3} +C_2 s,
\]
where 
\[
 \begin{aligned}
  K_1& \ce \frac{2 + 4 \sqrt{2} A_1 \cos\left(\omega \log{s} + \phi_1 \right) - 16 A_1 \sin\left(\omega \log{s} + \phi_1 \right)}{27 },\\
  K_2 & \ce \left(\frac{2}{27} - \frac{2 \ell_1}{3} + \ell_1^2 -\ell_2 - C_2- \frac{2}{9}\omega A_1 \cos{\phi_1} - \frac{20}{27} A_1 \sin{\phi_1} \right),
 \end{aligned}
\]
and 
\[
  z_2(s) = \frac{K_3}{s^{2/3}} + \frac{K_4}{s^{1/3}} + \frac{4}{3} A_1 \sin\phi_1+A_2 \sin\left(\omega\log{s}+\phi_2 \right) + K_5 s^{2/3},
\]
where 
\[
 \begin{aligned}
  K_3 & \ce  \frac{2}{81} + \frac{A_1^2}{2} - \frac{58 \sqrt{2}}{81} A_1 \cos\left(\omega\log{s}+\phi_1\right)- \frac{20}{81} A_1 \sin\left(\omega\log{s}+\phi_1\right)  \\
  	& \qquad + \frac{5}{38} A_1^2 \cos\left(2\omega\log{s} + 2 \phi_1\right) +  \frac{2\sqrt{2}}{19} A_1^2 \sin\left( 2\omega\log{s}+2 \phi_1\right),\\
	& \qquad \\
  K_4 & \ce \frac{4 - 36 \ell_1 + 54 \ell_1^2 - 54 \ell_2 - 54 C_2- 8 \sqrt{2} A_1 \cos{\phi_1}- 40 A_1 \sin{\phi_1}}{18}, \\
  K_5 & \ce \frac{48 \ell_1 - 36 \ell_1^2 -16 - 162 A_1^2 +  243 \lambda_2 + 108 C_2}{324}\\
  	& \qquad + \frac{9 A_1^2 \cos{2 \phi_1} - 
  2 A_1 (3 \ell_1-2) \sin{\phi_1}}{18}.
 \end{aligned}
\]

Next we introduce the inner variable $\rho = s/\gamma$, which after plugging into $z_0$ gives the near field behavior $z = -1 + \gamma^{2/3} \rho + \ldots$ as $s \to 0^+$. Moreover, $z = -1+ \bigoh(\delta)$ in the inner layer, and as a result, we choose $\gamma = \delta^{3/2}$. Then for matching we write the outer solution, \eqref{eq:outerexpansub1}, in terms of the inner variable, $\rho = s/\delta^{3/2}$, to obtain 
\begin{subequations}\label{eq:outersolnmatch2}
\be\label{eq:outersolnmatch2a}
 \begin{aligned}
 r & = \delta^{3/2} \left(\rho  -  \frac{2 \delta_0}{3} \rho^{1/3} + \frac{\delta_0^2 K_1}{\rho^{1/3}} \right) + \delta^2 \delta_0^2 K_2  \\
 & \qquad + \delta^{5/2} \left(\delta_0\left(\frac{2}{3} - \ell_1 \right)\rho + \frac{\delta_0^2(6 \ell_1 -4 + 36 A_1 \sin \phi_1) }{27} \rho^{1/3}  \right)   + \bigoh(\delta^{7/2})
 \end{aligned}
 \ee
 \be\label{eq:outersolnmatch2b}
 \begin{aligned}
 z& = -1 +  \delta\left( \rho^{2/3}  - \frac{4\delta_0}{9} + \delta_0 A_1 \sin\left( \omega\log{\rho} +\sqrt{2} \log{\delta}  + \phi_1 \right)+ \frac{\delta_0^2 K_3}{\rho^{2/3}}\right) \\
 &\qquad +\delta^{3/2}  \frac{\delta_0^2 K_4}{\rho^{1/3}} \\
 &\qquad+ \delta^2 \left(\left( \frac{(4-6 \ell_1)\delta_0}{9} - \delta_0 A_1 \sin{\phi_1}\right) \rho^{2/3}  \right. \\
 & \hspace{.75in} \left.+ \frac{4 \delta_0^2}{3} A_1 \sin\phi_1+ \delta_0^2 A_2 \sin\left(\omega\log{(\delta^{3/2}\rho)}+\phi_2 \right)\right)  +  \bigoh(\delta^3)
\end{aligned}
\ee
\end{subequations}
as $s \to 0^+$. It will be seen that the inner solution cannot be matched to the order $\bigoh(\delta^2)$ of \eqref{eq:outersolnmatch2a}. For this reason, the constant $C_2$ is chosen such that $K_2$ vanishes, which causes $K_4$ to vanish and  gives the reduced local behavior
\begin{subequations}\label{eq:outersolnmatch_red2}
\be\label{eq:outersolnmatch_red2a}
 \begin{aligned}
 r   =  \delta^{3/2} &\left[ \rho \vphantom{\frac{2}{3}}  -  \frac{2 \delta_0}{3} \rho^{1/3} +\mathcal{O}(\rho^{-1/3})  \right.  \\
 &  \quad+ \left.  \delta \left(\delta_0\left(\frac{2}{3} - \ell_1 \right)\rho + \frac{\delta_0^2(6 \ell_1 -4 + 36 A_1 \sin \phi_1) }{27} \rho^{1/3}  \right)   + \bigoh(\delta^{2}) \right],
 \end{aligned}
 \ee
 \be\label{eq:outersolnmatch_red2b}
 \begin{aligned}
 z = -1 +  \delta & \left[  \rho^{2/3}   \vphantom{\frac{2}{3}} - \frac{4\delta_0}{9} + \delta_0 A_1 \sin\left( \omega\log{\rho} +\sqrt{2} \log{\delta}  + \varphi_1 \right)+ \mathcal{O}(\rho^{-2/3})\right.\\
  & \quad +  \left.\delta \left(\left( \frac{(4-6 \ell_1)\delta_0}{9} - \delta_0 A_1 \sin{\phi_1}\right) \rho^{2/3}  + \mathcal{O}(1) \right)  +\bigoh(\delta^2)\right],
\end{aligned}
\ee
\end{subequations}
as $s \to 0^+$. As a consequence, these local expansions motivate us to introduce the following local variables within the vicinity of $s=0$,
\be\label{eq:innervariables}
 \rho = s/\delta^{3/2}, \quad r(s) = \delta^{3/2} R(\rho), \qquad z(s) = -1+\delta Z(\rho),
\ee
which transform \eqref{eq:paramemsode} into
\be\label{eq:innerproblem} 
 \begin{alignedat}{2}
  & R''= \frac{\eps^2}{\delta}  \left(-\frac{\lambda Z'}{Z^2}+ \frac{ (Z')^2}{R}\right), \qquad Z'' + \frac{R'  Z'}{R} = \frac{\lambda R'}{ Z^2}.
 \end{alignedat}
\ee
Here the dominant scalings require $\eps^2/\delta= \delta_0$, where $\delta_0 = \bigoh(1)$. Therefore, expanding $R$, $Z$ and $\lambda$ as 
\be\label{eq:innerexpansions}
 R = R_0 + \delta R_1 + \bigoh(\delta^2) , \quad Z=Z_0 + \delta Z_1 + \bigoh(\delta^2), \quad \lambda = \lambda_0 + \delta \delta_0 \lambda_1 + \bigoh(\delta^2),
\ee respectively, we find that the leading order problem for the inner solution is 
\begin{equation}\label{eq:insolnord1}
 \begin{alignedat}{2}
  & R_0''= -\delta_0 \frac{ \lambda_0 Z_0'}{Z_0^2}+ \delta_0 \frac{ (Z_0')^2}{R_0}, \qquad Z_0'' + \frac{R_0'  Z_0'}{R_0} = \frac{\lambda_0 R_0'}{ Z_0^2}, \qquad 0< \rho < \infty; \\
  & R_0(0) = 0, \ R_0'(0)=1, \ Z_0(0)= 1, \ Z_0'(0) = 0.
 \end{alignedat}
\end{equation}
To find the far field behavior of $R_0$ and $Z_0$, we assume $R_0 \sim \rho + V$ as $\rho \to \infty$, where $V \ll \rho$, and $Z_0 \sim \rho^{2/3} + W$ as $\rho \to \infty$, where $W \ll \rho^{2/3}$. Substituting these relations into \eqref{eq:insolnord1}, gives asymptotic differential equations for $V(\rho)$ and $W(\rho)$, 
\begin{equation}\nonumber
 V'' \sim \frac{4\delta_0}{27} \rho^{-5/3}, \quad W'' + \frac{W'}{\rho} + \frac{2 \lambda_0}{\rho^2} W \sim -\frac{32 \delta_0}{81}\frac{1}{\rho^2},  \quad \mbox{ as } \rho \to \infty,
\end{equation}
whose solution is $V \sim -(2 \delta_0/3) \rho^{1/3}$, $W \sim - \delta_0\lambda_0 + \til{A}_1(\delta_0) \sin(\omega \log{\rho} + \til{\phi}_1(\delta_0) )$ as $\rho \to \infty$. Hence, the far field behavior for the solution, $R_0$ and $Z_0$ of \eqref{eq:insolnord1} is 
\be\label{eq:inord1ff}
 \begin{aligned}
  R_0(\rho) &= \rho -\frac{2 \delta_0}{3} \rho^{1/3} + \littleoh(1), \\
  Z_0(\rho) & = \rho^{2/3}  -  \delta_0 \lambda_0 + \til{A}_1(\delta_0) \sin\left( \omega \log{\rho} + \til{\phi}_1(\delta_0) \right) +\littleoh(1),
\end{aligned}
 \quad \mbox{as } \rho \to \infty.
\ee
Proceeding to $\mathcal{O}(\delta)$ terms, we substitute \eqref{eq:innerexpansions} into \eqref{eq:innerproblem} and collect the $\delta$ terms to find that $R_1$ and $Z_1$ satisfy
\be\label{eq:order2innerprob}
 \begin{aligned}
  &R_1''  = \delta_0 \left( - \frac{\delta_0 \lambda_1 Z_0' + \lambda_0 Z_1 }{Z_0^2}+ \frac{2 \lambda_0 Z_1 Z_0'}{Z_0^3} - \frac{R_1 Z_0'^2}{R_0^2}  +  \frac{2 Z_0' Z_1' }{R_0}\right), \ \ 0<\rho<\infty, \\
  &Z_1''  = \frac{\delta_0 \lambda_1 R_0'+ \lambda_0 R_1'}{Z_0^2}- \frac{2 \lambda_0 Z_1 R_0'}{Z_0^3}  + \frac{R_1 R_0' Z_0'}{R_0^2} - \frac{R_1' Z_0' + R_0'Z_1'}{R_0},  \ \ 0<\rho<\infty, \\
  &R_1(0) = 0, \ R_1'(0)=0, \ Z_1(0)= 0, \ Z_1'(0) = 0.
 \end{aligned}
 \ee
 From \eqref{eq:outersolnmatch_red2} we expect the far field behavior of both $R_1$ and $Z_1$ to grow algebraically. Therefore, we assume $R_1 \sim  a \rho^{\alpha} $ and $Z_1 \sim b \rho^{\beta}$ as $\rho \to \infty$ and substitute this behavior into \eqref{eq:order2innerprob} along with the far field behavior of $R_0$ and $Z_0$. After a dominant balance, this yields
 \[
 \begin{aligned}
  a \alpha (\alpha-1)  \rho^{\alpha-2} & \sim -\frac{2 \delta_0^2 \lambda_1}{3}\rho^{-5/3}+ \left(\frac{4b \beta \delta_0}{3} +  \frac{4b \delta_0 \lambda_0}{3} - b \beta \delta_0 \lambda_0 \right) \rho^{\beta-7/3}, \\
  b \beta (\beta-1)  \rho^{\beta-2} & \sim \delta_0 \lambda_1 \rho^{-4/3} -  b \left(\beta  +  2  \lambda_0\right) \rho^{\beta-2},
  \end{aligned}
\]
 as $\rho \to \infty$. Consequently, 
 \begin{equation}\nonumber
  a = -\delta_0^2 \lambda_1, \quad \alpha = \frac{1}{3}, \quad b = \frac{3 \delta_0 \lambda_1}{4}, \quad \beta = \frac{2}{3},
 \end{equation}
which implies that 
\be\label{eq:inord2fflo}
 R_1 \sim  -\delta_0^2 \lambda_1 \rho^{1/3}, \qquad Z_1 \sim \frac{3 \delta_0 \lambda_1}{4} \rho^{2/3}, \quad \mbox{as } \rho \to \infty.
\ee
As a result, \eqref{eq:innerexpansions}, \eqref{eq:inord1ff} and \eqref{eq:inord2fflo} give the following far field behavior of the inner solution:
\be\label{eq:outersolnmatch}
 \begin{aligned}
 R & = \left(\rho -\frac{2 \delta_0}{3} \rho^{1/3} + \ldots\right) +\delta\left(-\delta_0^2 \lambda_1 \rho^{1/3} + \ldots \right) + \mathcal{O}(\delta^2),\\
 Z & =  \left(\rho^{2/3} -  \delta_0 \lambda_0 + \til{A}_1(\delta_0) \sin\left(\omega \log{\rho} + \til{\phi}_1(\delta_0) \right)+\ldots\right)\\& \qquad + \delta \left(\frac{3 \delta_0 \lambda_1}{4} \rho^{2/3}+\ldots\right) + \bigoh(\delta^2)
   \end{aligned}
\ee
as $\delta\to 0^+$ and $\rho \to \infty$. Then to match we compare \eqref{eq:innervariables}, using \eqref{eq:outersolnmatch}, with \eqref{eq:outersolnmatch_red2} to get
\be\label{eq:matchingCAphi}
 \ell_1=\frac{2}{3}, \quad A_1 = \frac{\til{A}_1(\delta_0)}{\delta_0}, \quad   \phi_1= \til{\phi}_1(\delta_0) - \sqrt{2}\log \delta, \quad \lambda_1 = -\frac{4}{3}A_1 \sin\phi_1.
\ee

Note that the boundary condition $z_1(1)=-2 \ell_1/3$ is automatically satisfied by the value of $\lambda_1$ determined in \eqref{eq:matchingCAphi}. By returning to the definition of $\lambda$ made in \eqref{eq:outer_def_lambda_ell} and recalling that $\eps^2/ \delta = \delta_0$ , a two term expansion of $\lambda$ is now given by 
\[
\lambda = \frac{4}{9} - \delta\,\frac{4}{3} \, \til{A}_1\!\!\left(\frac{\eps^2}{\delta}\right) \sin\left[\til{\phi}_1\!\!\left(\frac{\eps^2}{\delta}\right) - \sqrt{2}\log \delta \right] + \cdots.
\]
Next, we fix $\eps$ in the governing equation, \eqref{eq:paramemsode}; therefore, for our asymptotic analysis to remain valid, we need $\eps^2/ \delta = \delta_0 = \bigoh(1)$, with $\eps$ fixed, which leads to the following asymptotic result regarding the upper solution branch of the bifurcation diagram of \eqref{eq:paramemsode}. 
\begin{pr}\label{pr:2drad_p}
For solutions of \eqref{eq:paramemsode}, there is a regime where both $\eps \ll 1$ and $\delta \ll 1$, with $\eps^2/\delta = \mathcal{O}(1)$, such that the upper solution branch of the bifurcation curve has the asymptotic parameterization, $(\lambda(\delta;\eps),|z(0)|),$ where 
\be\label{eq:PRsct2_2_a}
 \begin{aligned}
 & |z(0)| = 1 - \delta, \quad \ \lambda = \lambda_0 - \delta\,\frac{4}{3} \, \til{A}_1\!\!\left(\frac{\eps^2}{\delta}\right) \sin\left[\til{\phi}_1\!\!\left(\frac{\eps^2}{\delta}\right) - \sqrt{2}\log \delta \right] + \bigoh(\delta^2). 
 \end{aligned} 
\ee
Moreover, $\lambda_0 = 4/9$, and $\til{A}_1(\delta_0)$ and $\til{\phi}_1(\delta_0)$ are functions determined by the far field behavior of $Z_0$,
\bsub\label{eq:2drad_np_innerprob}
\be
  Z_0(\rho)  = \rho^{2/3}  -\frac{4 \delta_0}{9} + \til{A}_1(\delta_0) \sin\left( \omega \log{\rho} + \til{\phi}_1(\delta_0) \right) +\littleoh(1) \quad \mbox{ as } \rho \to \infty,
\ee
of the initial value problem
\begin{equation}
 \begin{alignedat}{2}
  & R_0''= -\frac{4 \delta_0}{9} \frac{Z_0'}{Z_0^2}+ \delta_0 \frac{ (Z_0')^2}{R_0}, \qquad Z_0'' + \frac{R_0'  Z_0'}{R_0} = \frac{4}{9}\frac{ R_0'}{ Z_0^2}, \qquad 0< \rho < \infty; \\
  & R_0(0) = 0, \ R_0'(0)=1, \ Z_0(0)= 1, \ Z_0'(0) = 0,
 \end{alignedat}
\end{equation}
where
\be
 R_0(\rho) = \rho -\frac{2 \delta_0}{3} \rho^{1/3}+ \littleoh(1) \quad \mbox{as } \rho \to \infty.
\ee
\esub
\end{pr}

To study the accuracy of this result, we again need to compute the functions $\til{A}_1(\delta_0)$ and $\til{\phi}_1(\delta_0)$, which is done by subtracting the growth term and applying a least squares fit to the remainder. The graphs of these functions are displayed in Figure~\ref{fig:A_phi_2}. As expected, these new functions are continuations of the old functions found in Principle Result \ref{pr:2drad_np}.
\begin{figure}[h]
\centering
\subfigure[$\til{A}_1(\delta_0)$]{\label{fig:Avseps02}\includegraphics[width=0.425\textwidth]{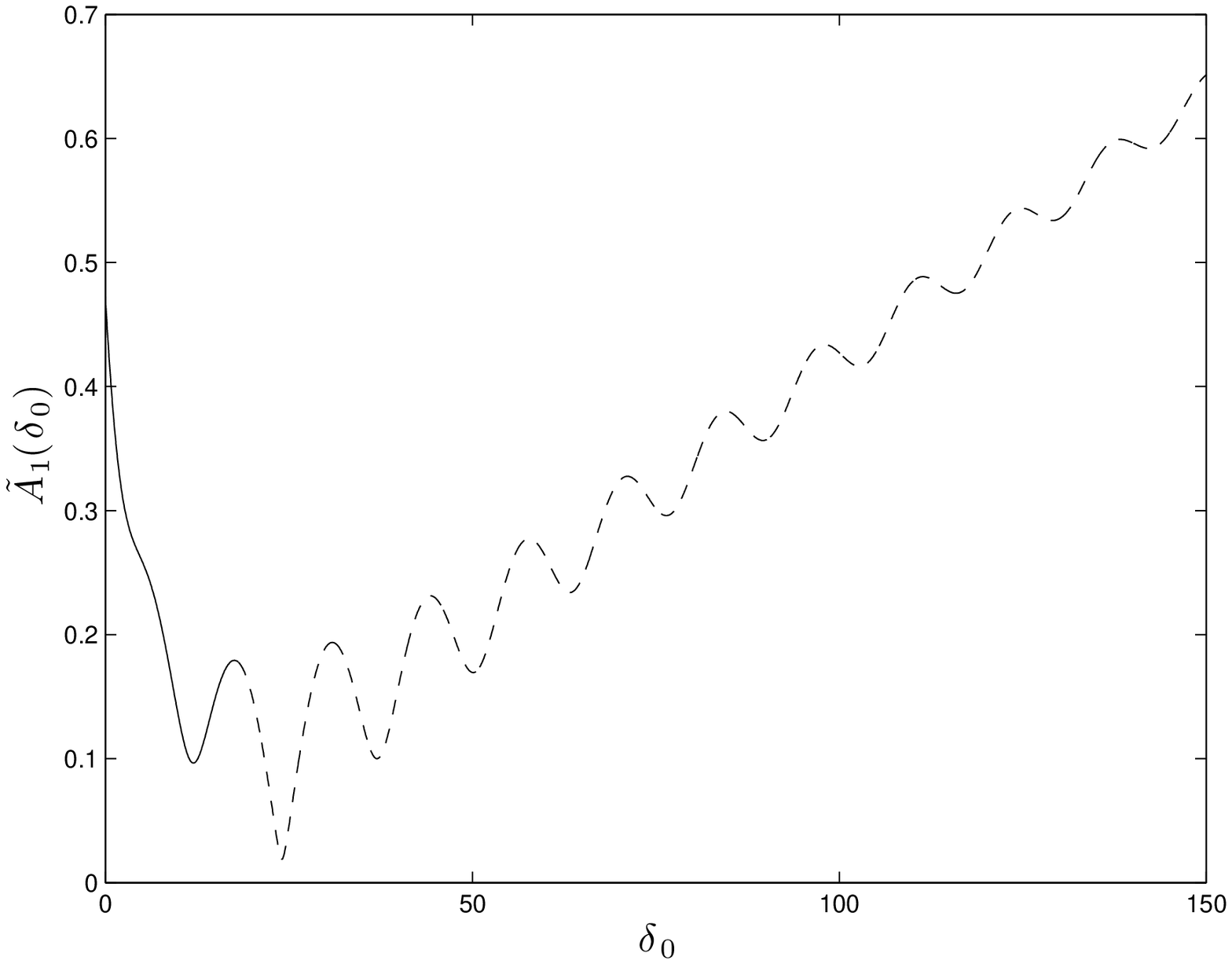}}\qquad
\subfigure[$\til{\phi}_1(\delta_0)$]{\includegraphics[width=0.425\textwidth]{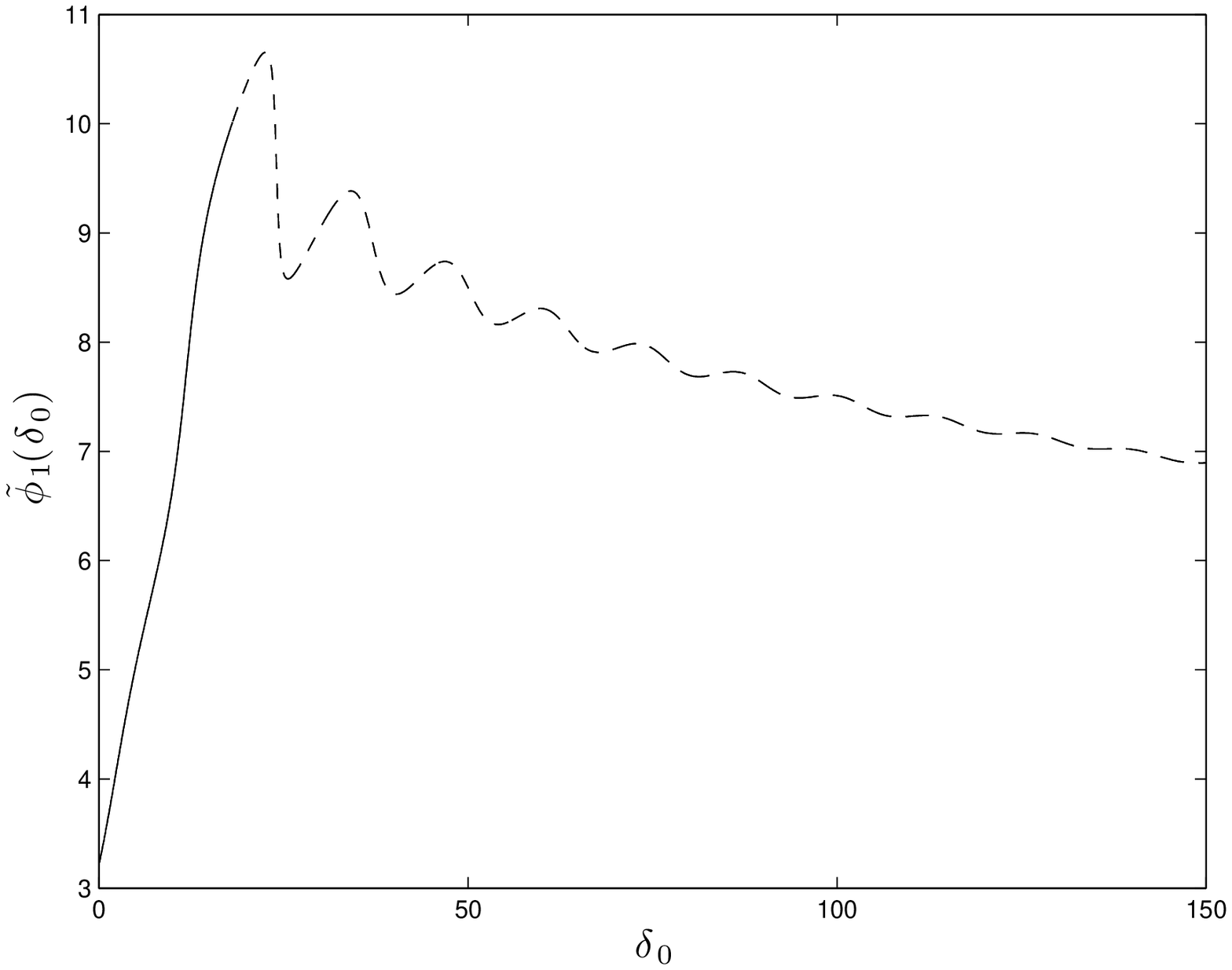}}
\caption{Graphs of $\til{A}_1(\delta_0)$ and $\til{\phi}_1(\delta_0)$ against $\delta_0$ computed from \eqref{eq:2drad_np_innerprob}. The solid line indicates where the coefficients agree with those compute from \eqref{PR1}. \label{fig:A_phi_2} }
\end{figure}

A combination of the asymptotic formula \eqref{eq:PRsct2_2_a} and the numerically obtained functions $\til{A}_1(\delta_0)$ and $\til{\phi}_1(\delta_0)$ allow for a reconstruction of the bifurcation diagram (see Figure \ref{fig:Comp_parametric}). Numerically, $\til{A}_1(\delta_0)$ appears to grow linearly as $\delta_0\to\infty$, which would indicate that $\delta \til{A}_1(\eps^2/\delta)$ is finite as $\delta\to0$. Therefore the analysis predicts that the upper solution branch of \eqref{eq:paramemsode} undergoes infinitely many fold points in a way similar to the upper branch of \eqref{eq:stdpde} for the two-dimensional unit disk. This prompts the following conjecture.

\begin{conj}\label{conj_1}
 For $\eps >0$ fixed and sufficiently small, the upper solution branch of the the bifurcation diagram of \eqref{eq:arclengthsystem} undergoes infinitely many folds and as $|z(0)| \to 1^-$, $\lambda>0$ goes to a finite value that is bounded away from zero.
\end{conj}

In Figure~\ref{fig:Comp_parametric}, \eqref{eq:PRsct2_2_a} is compared with the numerically computed bifurcation diagram of \eqref{eq:arclengthsystem}. From this we see that the observed agreement is very good.

\begin{figure}[h]
\centering
\subfigure[$\eps = 0.05$]{\label{fig:asympara_eps_pt_05}\includegraphics[width=0.4125\textwidth]{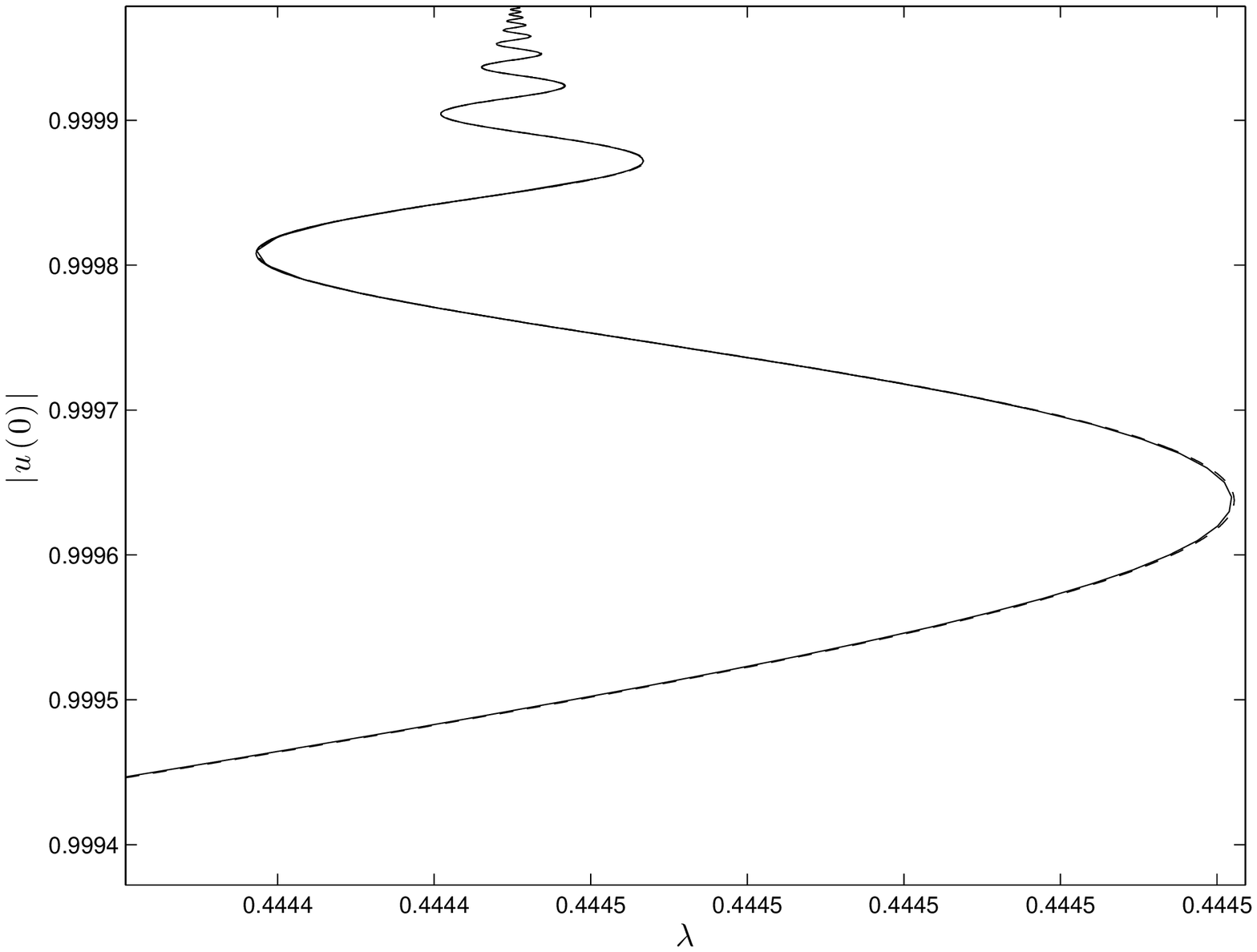}}\qquad
\subfigure[$\eps = 0.1$]{\label{fig:asympara_eps_pt_1}\includegraphics[width=0.425\textwidth]{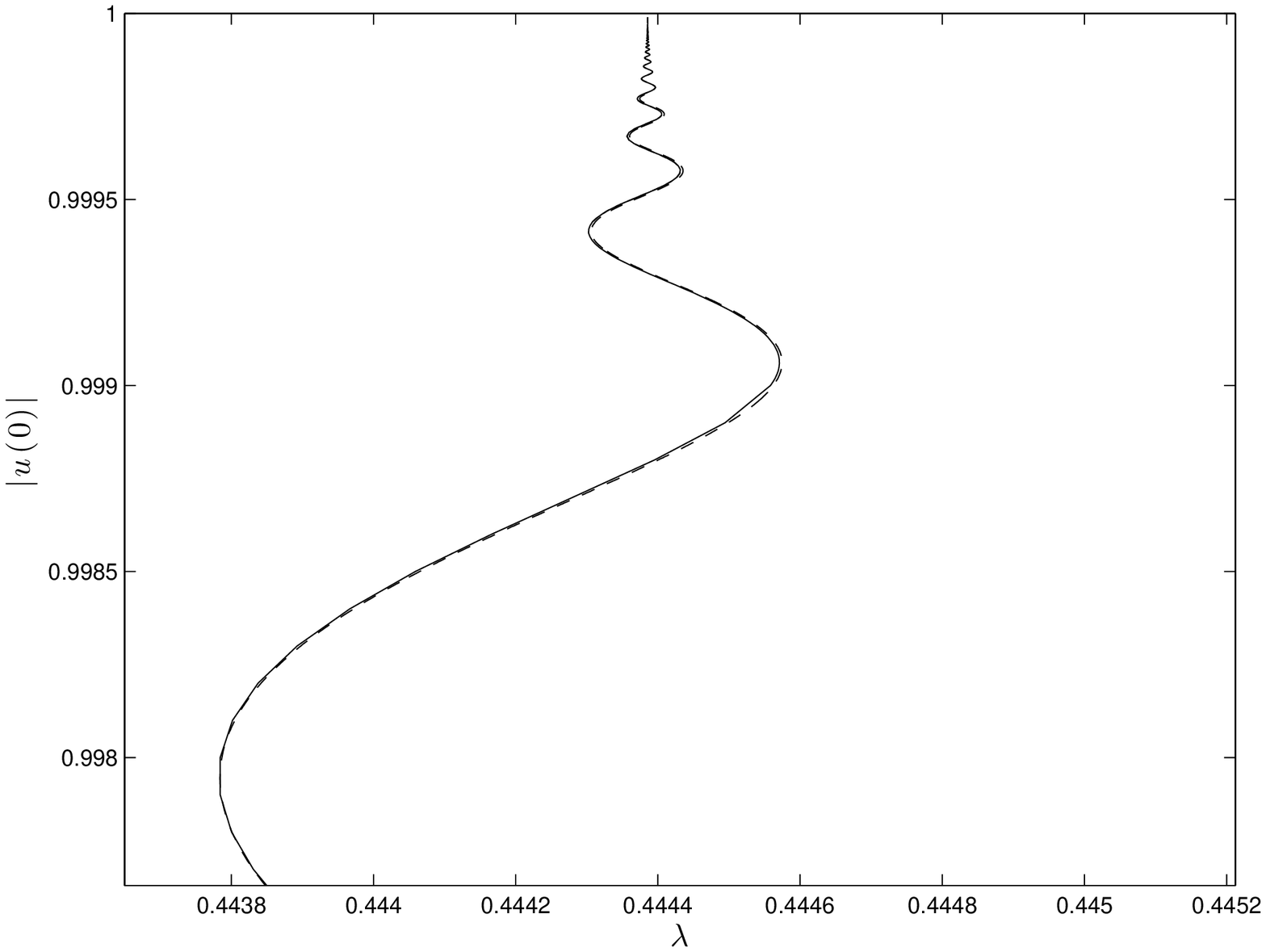}}
\caption{ Comparison of the full numerical solution of the bifurcation curve for \eqref{eq:gov2dradpde} (solid) with asymptotic formula \eqref{PR1_a} (dashed) for: \subref{fig:asympara_eps_pt_05} $\eps = 0.05$; \subref{fig:asympara_eps_pt_1} $\eps=0.1$. Note that in both \subref{fig:asympara_eps_pt_05}  and \subref{fig:asympara_eps_pt_1}  the asymptotic prediction agrees extremely well with the full numerical solution.} \label{fig:Comp_parametric}
\end{figure}

\section{Conclusion}\label{section:conclusion}

In this work, we have analyzed the \emph{upper branch} of solutions to equation \eqref{eq:govpde} in the limit as $|u(0)|\to1^-$ with particular focus on one and two spatial dimensions. In both cases, there are marked differences between the solution structure for $\eps=0$ and $\eps>0$, specifically a disappearance of solutions is observed, i.e., \eqref{eq:govpde} does not necessarily have a solution for all $|u(0)| \in[0,1)$. 

For the case $n=1$,  there are two qualitatively different bifurcations structures associated with \eqref{eq:govpde} which are separated by the cases $\eps\leq\eps^{\ast}$ and $\eps>\eps^{\ast}$ (c.f. Fig.~\ref{fig:mcbds1d}) for a critical value $\eps^{\ast}\approx0.35$. In the case where $\eps>\eps^{\ast}$, the upper solution branch of equation \eqref{eq:govpde} is defined only for $|u(0)|\in [0,\alpha_{**}(\eps))\cup(\alpha_*(\eps),1)$, where $0<\alpha_{**}(\eps)<\alpha_*(\eps)<1$, while for $\eps\leq\eps^{\ast}$ equation \eqref{eq:govpde} has a solution for any $|u(0)|\in [0,1)$. Therefore, in the case $n=1$, equation \eqref{eq:govpde} has solutions for $|u(0)|$ arbitrarily close to $1$, and these solutions have been constructed using singular perturbation techniques in the limit $|u(0)|\to1^-$. The limiting form of the bifurcation diagram, encapsulated in Principal Result \ref{pr:1d_conc}, has been found to be highly accurate, however, in contrast to the asymptotic analysis of symmetric solutions of \eqref{eq:govpde} in $n=2$, it does not predict the dead-end point $(\lambda_*(\eps),\alpha_{\ast}(\eps))$.

In the case $n=2$, we have shown that solutions of \eqref{eq:gov2dradpde} do not exist for $|u(0)|$ arbitrarily close to $1$ for any $\eps>0$. It is observed that as the upper solution branch is traversed, a singularity in the first derivative of the solution develops interior to the domain and at this singularity, the branch of solutions ends abruptly at a single dead-end point. Our asymptotic analysis allows for an accurate prediction of this point to be made by relating it to a singularity in an associated initial value problem. The analysis predicts that the singularity occurs for a fixed value of $\eps^2/(1+|u(0)|)$ and therefore establishes a relationship between a given $\eps$ and the dead-end point.  In each case, the asymptotic parameterizations obtained for the solution branch compare very well with full numerical solutions. The ability of the asymptotic analysis to predict the dead-end point when $n=2$ case but not when $n=1$, suggests a different underlying mechanism is responsible for the phenomena in each case.

Finally, by employing an arc length parameterization of solutions to \eqref{eq:gov2dradpde}, we find and analyze a new family of solutions emanating from the dead-end point. These solutions are found to be multi-valued and provide a natural continuation of the bifurcation curve beyond the dead-end point which retains the infinite fold points feature of the $\eps=0$ problem.

The main limitation of our study is that for $n=2$, we deal only with radially symmetric domains. Though our analysis has revealed interesting structure associated with solutions of \eqref{eq:gov2dradpde}, an investigation of \eqref{eq:govpde} on more general domains would be desirable; specifically, can a result like Theorem  \ref{thm:fp_mainthm} be formulated, where $|u(0)|$ is replaced by $\|u\|_\infty$.

Additionally, it would be interesting to study solutions of \eqref{eq:govpde} in higher spatial dimensions. Rigorously it has been shown that when $\Omega$ is the $n$-dimensional unit ball and $2\leq n\leq7$, the bifurcation diagram of \eqref{eq:stdpde} exhibits the infinite fold points structure \cite{guo02008infinitely}. What then is the effect of positive $\eps$ on the bifurcation structure of \eqref{eq:govpde} when $n\geq3$? Another interesting avenue for future investigation is the dynamic version of \eqref{eq:govpde}, namely the equation
\begin{equation}\label{conc_eq1}
\begin{array}{c}
u_t = \mbox{div } \ds\frac{\nabla u }{\sqrt{1 + \eps^2 |\nabla u|^2}} - \ds\frac{\lambda}{(1+u)^2}, \quad x\in \Omega;
\\[10pt]
 u(x,0) = 0, \quad x\in\Omega; \qquad  u = 0, \quad x\in\partial\Omega.
\end{array}
\end{equation}
Is there an equivalent of disappearance of solutions for \eqref{conc_eq1}, i.e., does $u$ or its derivatives exhibit a singularity at some finite $t$ before reaching $u=-1$?

\section*{Acknowelegments}
N.D.B. would like to thank J.\:A. Pelesko for many useful discussions.

\appendix\section*{}\label{app:intexpn}
Here we compute the integral given in \eqref{eq:integralasymbehavior} and then expand for $y \gg 1$. First, we split the aforementioned integral into two parts:
\begin{equation}\label{eq:appint}
I \colonequals \int_1^{w_0(y)}\frac{\eps^2 \lambda_1 + (1 - \eps^2 \lambda_1) z}{\sqrt{(2 - \eps^2 \lambda_1) z^2 -2(1 - \eps^2 \lambda_1) z-\eps^2 \lambda_1}}\ud z = I_1 + I_2
\end{equation}
where
\begin{equation}\nonumber
\begin{aligned}
I_1 & \colonequals \eps^2 \lambda_1\int_1^{w_0(y)}\frac{1}{\sqrt{(2 - \eps^2 \lambda_1) z^2 -2(1 - \eps^2 \lambda_1) z-\eps^2 \lambda_1}}\ud z, \\
I_2 & \colonequals (1 - \eps^2 \lambda_1) \int_1^{w_0(y)}\frac{z}{\sqrt{(2 - \eps^2 \lambda_1) z^2 -2(1 - \eps^2  \lambda_1) z-\eps^2 \lambda_1}}\ud z.
 \end{aligned}
\end{equation}
In computing $I_1$ and $I_2$, respectively, we obtain
\begin{equation}\nonumber
\begin{aligned}
I_1 & \colonequals \eps^2 \lambda_1\sqrt{2 - \eps^2 \lambda_1}\int_1^{w_0(y)}\frac{1}{\sqrt{ \left((2 - \eps^2 \lambda_1)z- (1 - \eps^2 \lambda_1)\right)^2- 1}}\ud z  \\
	& = \frac{\eps^2 \lambda_1}{\sqrt{2 - \eps^2 \lambda_1}}\int_1^{(2 - \eps^2 \lambda_1)w_0- (1 - \eps^2 \lambda_1)}\frac{1}{\sqrt{ \xi^2- 1}}\ud \xi \\
	& = \frac{\eps^2 \lambda_1}{\sqrt{2 - \eps^2 \lambda_1}} \cosh^{-1}\left((2 - \eps^2 \lambda_1)w_0- (1 - \eps^2 \lambda_1)\right).
 \end{aligned}
\end{equation}
and
\begin{equation}\nonumber
\begin{aligned}
I_2 & \colonequals (1 - \eps^2 \lambda_1)\sqrt{2 - \eps^2 \lambda_1} \int_1^{w_0(y)}\frac{z}{\sqrt{\left((2 - \eps^2 \lambda_1)z- (1 - \eps^2 \lambda_1)\right)^2- 1}}\ud z \\
	& = \frac{(1 - \eps^2 \lambda_1)}{ (2 - \eps^2 \lambda_1)^{3/2}} \int_1^{(2 - \eps^2 \lambda_1)w_0- (1 - \eps^2 \lambda_1)}\frac{\xi + (1 - \eps^2 \lambda_1)}{\sqrt{\xi^2- 1}}\ud \xi \\
	& = \frac{(1 - \eps^2 \lambda_1)}{ (2 - \eps^2 \lambda_1)^{3/2}} \sqrt{\left((2 - \eps^2 \lambda_1)w_0- (1 - \eps^2 \lambda_1)\right)^2- 1} \\
	& \qquad +  \frac{(1 - \eps^2 \lambda_1)^2}{ (2 - \eps^2 \lambda_1)^{3/2}} \cosh^{-1}\left((2 - \eps^2 \lambda_1)w_0- (1 - \eps^2 \lambda_1)\right)
 \end{aligned}
\end{equation}
Therefore, since $y \gg 1$, $w_0 \gg 1$, which implies
\begin{equation}\nonumber
\begin{aligned}
 \cosh^{-1}\left((2 - \eps^2 \lambda_1)w_0- (1 - \eps^2 \lambda_1)\right) &=  \log{w_0} + \log\left(4 - 2\eps^2 \lambda_1\right)+ \bigoh\left(w_0^{-1}\right) \\
 \sqrt{\left((2 - \eps^2 \lambda_1)w_0- (1 - \eps^2 \lambda_1)\right)^2- 1} &= (2 - \eps^2 \lambda_1)w_0 - (1 - \eps^2 \lambda_1)+ \bigoh\left({w_0}^{-1}\right)
 \end{aligned}
\end{equation}
and hence,
\begin{equation}\label{eq:appint1}
\begin{aligned}
I_1 & =  \frac{\eps^2 \lambda_1}{\sqrt{2 - \eps^2 \lambda_1}}\log{w_0} +  \frac{\eps^2 \lambda_1\log\left(4 - 2\eps^2 \lambda_1\right)}{\sqrt{2 - \eps^2 \lambda_1}} + \bigoh\left(w_0^{-1}\right),\\
I_2 & = \frac{(1 - \eps^2 \lambda_1)}{ \sqrt{2 - \eps^2 \lambda_1}}w_0 + \frac{(1 - \eps^2 \lambda_1)^2}{ (2 - \eps^2 \lambda_1)^{3/2}} \log{w_0} \\
	& \qquad  +  \frac{(1 - \eps^2 \lambda_1)^2\log\left(4 - 2\eps^2 \lambda_1\right)-(1 - \eps^2 \lambda_1)^2}{ (2 - \eps^2 \lambda_1)^{3/2}} + \bigoh\left(w_0^{-1} \right)
 \end{aligned}
\end{equation}
as $w_0 \to \infty$. Therefore, from \eqref{eq:appint} and \eqref{eq:appint1} we obtain
\begin{equation}\nonumber
 I  = \frac{(1 - \eps^2 \lambda_1)}{ \sqrt{2 - \eps^2 \lambda_1}}w_0 +  \frac{1}{\left(2 - \eps^2 \lambda_1\right)^{3/2}}\log{w_0} +   \frac{\log\left(4 - 2\eps^2 \lambda_1\right) -(1 - \eps^2 \lambda_1)^2}{\left(2 - \eps^2 \lambda_1\right)^{3/2}}+ \bigoh\left(w_0^{-1} \right).
\end{equation}

\section*{}\label{app:nummethod}
To solve \eqref{eq:paramemsode} numerically, we use a shooting method. That is, we impose the initial conditions 
\be\label{eq:appbbcs}
  r(0)=0, \  r'(0)=\ell,  \ z'  (0)=0, \ u(0) = \alpha,
\ee
where $\alpha \in (-1,0)$ and find $(\lambda,\ell)$ such that $F(\lambda,\ell) \colonequals \begin{bmatrix}  r(1;\lambda,\ell) - 1 &  z(1;\lambda,\ell) \end{bmatrix}^T= \mathbf{0}.$ To do so, we apply Newton's method and iterate as 
\begin{equation}\nonumber
\mbox{{\boldmath $\lambda$}}_{n+1} = \mbox{{\boldmath $\lambda$}}_{n}  - 
 \begin{bmatrix}  \displaystyle \pd{r}{\lambda}(1;\lambda_n,\ell_n) &\displaystyle  \pd{r}{\ell}(1;\lambda_n,\ell_n) \vspace{5pt} \\  \displaystyle \pd{z}{\lambda}(1;\lambda_n,\ell_n) &\displaystyle  \pd{z}{\ell}(1;\lambda_n,\ell_n)  \end{bmatrix}^{-1} \begin{bmatrix}  r(1;\lambda_n,\ell_n) - 1 \\  (1;\lambda_n,\ell_n) \end{bmatrix}
\end{equation}
where $\mbox{{\boldmath $\lambda$}}_{n} = \begin{bmatrix}  \lambda_n & \ell_n \end{bmatrix}^T$. Therefore at each step we need to find $r_\lambda(1;\lambda_n,\ell_n),$  $r_\ell(1;\lambda_n,\ell_n),$  $z_\lambda(1;\lambda_n,\ell_n)$ and $z_\ell(1;\lambda_n,\ell_n)$. To this end, we differentiate the ode given in \eqref{eq:paramemsode} and the initial conditions \eqref{eq:appbbcs} with respect to $\lambda$ and separately with respect to $\ell$ to get two auxiliary problem for $(r_\lambda,z_\lambda)$ and $(r_\ell,z_\ell)$, whose solutions evaluated at $\xi=1$ yield our desired result.

\section*{}\label{app:secondordouterODE}
Here are the order $\bigoh(\eps^4)$ outer ODEs for \eqref{eq:paramemsode}:
\be
\begin{aligned}
 &r_2''  =  \frac{M_1}{s^{7/3}} + \frac{M_2}{ s^{5/3}},\\
 & r_2(1) = \ell_1 (\ell_1 - \lambda_0) - \ell_2, \\
 &  z_2'' + \frac{1}{s}z_2' + \frac{2\lambda_0}{s^2}z_2 = \frac{2 r_2}{3s^{7/3}} - \frac{2 r_2'}{9 s^{4/3}} - \frac{M_3}{s^{2}} + \frac{M_4}{s^{4/3}} - \frac{M_5}{s^{8/3}}\\
 & z_2(1) = \frac{\ell_1 (15 \ell_1 - 8)}{27} - \frac{2 \ell_2}{3} + \frac{2 A_1 \ell_1}{3} ( \sin{\phi_1}-\sqrt{2}\cos{\phi_1} )
 \end{aligned} 
\ee
where 
\begin{equation}\nonumber
 \begin{aligned}
  M_1 & \colonequals \frac{144 \sqrt{2} A_1 \cos\left(\frac{2 \sqrt{2}}{3}\log{s} + \phi_1 \right) + 144 A_1 \sin\left(\frac{2 \sqrt{2}}{3}\log{s}+\phi_1\right)+ 8}{243} \\
  M_2 & \colonequals \frac{8 - 12 \ell_1 - 72 A_1 \sin \phi_1}{243}\\
  M_3 & \colonequals \frac{32 (3-2 \ell_1 - 9 A_1 \sin{\phi_1})}{729 } \\ 
  M_4 & \colonequals \frac{227 + 48 \ell_1 - 36 \ell_1^2 - 36 A_1 (3 \ell_1-2 ) \sin\phi_1- 324 A_1^2 \sin^2\phi_1}{243 } \\
  M_5 & \colonequals \frac{4 \left(4 + 54 \sqrt{2}A_1 \cos\left(\frac{2\sqrt{2}}{3}\log{s}+\phi_1\right) + 180 A_1 \sin\left(\frac{2\sqrt{2}}{3}\log{s}+\phi_1\right)\right)}{729 }  \\
  	& \qquad - \frac{4  A_1^2 \sin^2\left(\frac{2\sqrt{2}}{3}\log{s}+\phi_1\right) }{3 }
 \end{aligned}
\end{equation}
and we have simplified the result using  \eqref{eq:outord1soln}, \eqref{eq:ord2routersoln}, \eqref{eq:outordereps2soln} and \eqref{eq:matchingCAphi}.

\bibliographystyle{siam}	
\bibliography{bibliography}		

\begin{thebibliography}{10}

\bibitem{brubaker2011analysis}
{\sc N.~D. Brubaker and J.~A. Pelesko}, {\em Analysis of a one-dimensional
  prescribed mean curvature equation arising in the study of mems}, preprint,
  (2011).

\bibitem{brubaker2011nonlinear}
\leavevmode\vrule height 2pt depth -1.6pt width 23pt, {\em Non-linear effects
  on canonical mems models}, European J. Appl. Math., 22 (2011), pp.~455--470.

\bibitem{burns2011steady}
{\sc M.~Burns and M.~Grinfeld}, {\em Steady state solutions of a bi-stable
  quasi-linear equation with saturating flux}, European J. Appl. Math., 22
  (2011), pp.~317--331.

\bibitem{esposito2010memsbook}
{\sc P.~Esposito, N.~Ghoussoub, and Y.~Guo}, {\em Mathematical Analysis of
  Partial Differential Equations Modeling Electrostatic MEMS}, vol.~20 of
  Courant Lecture Notes in Mathematics, American Mathematical Society,
  Providence, RI, 2010.

\bibitem{finn1986equilibrium}
{\sc R.~Finn}, {\em Equilibrium capillary surfaces}, vol.~284 of Grundlehren
  der mathematischen Wissenschaften, Springer-Verlag, New York, 1986.

\bibitem{wardguopan}
{\sc Y.~Guo, Z.~Pan, and M.~J. Ward}, {\em Touchdown and pull-in voltage
  behavior of a mems device with varying dielectric properties}, SIAM J. Appl.
  Math., 66 (2005), pp.~309--338.

\bibitem{guo02008infinitely}
{\sc Z.~Guo and J.~Wei}, {\em Infinitely many turning points for an elliptic
  problem with a singular non-linearity}, J. Lond. Math. Soc. (2), 78 (2008),
  pp.~21--35.

\bibitem{habets2004positive}
{\sc P.~Habets and P.~Omari}, {\em Positive solutions of an indefinite
  prescribed mean curvature problem on a general domain}, Adv. Nonlinear Stud.,
  4 (2004), pp.~1--13.

\bibitem{hinch1991perturbation}
{\sc E.~J. Hinch}, {\em Perturbation methods}, Cambridge Texts in Applied
  Mathematics, Cambridge University Press, Cambridge, 1991.

\bibitem{le2008subsupersolution}
{\sc V.~K. Le}, {\em On a sub-supersolution method for the prescribed mean
  curvature problem}, Czechoslovak Math. J., 58 (2008), pp.~541--560.

\bibitem{lindsay2011asymptotics2}
{\sc A.~E. Lindsay and M.~J. Ward}, {\em Asymptotics of some nonlinear
  eigenvalue problems modelling a mems capacitor. part ii: multiple solutions
  and singular asymptotics}, European J. Appl. Math., 22 (2011), pp.~83--123.

\bibitem{mellet2010existence}
{\sc A.~Mellet and J.~Vovelle}, {\em Existence and regularity of extremal
  solutions for a mean-curvature equation}, J. Differential Equations, 249
  (2010), pp.~37--75.

\bibitem{moulton2008theory}
{\sc D.~E. Moulton and J.~A. Pelesko}, {\em Theory and experiment for soap-film
  bridge in an electric field}, J. Colloid Interface Sci., 322 (2008),
  pp.~252--262.

\bibitem{Obersnel2010}
{\sc F.~Obersnel and P.~Omari}, {\em Positive solutions of the dirichlet
  problem for the prescribed mean curvature equation}, J. Differential
  Equations, 249 (2010), pp.~1674--1725.

\bibitem{pan2009one}
{\sc H.~Pan}, {\em One-dimensional prescribed mean curvature equation with
  exponential nonlinearity}, Nonlinear Anal., 70 (2009), pp.~999--1010.

\bibitem{pan2011radial}
{\sc H.~Pan and R.~Xing}, {\em Radial solutions for a prescribed mean curvature
  equation with exponential nonlinearity}, Nonlinear Anal.,  (2011), pp.~--.
\newblock (DOI: 10.1016/j.na.2011.08.010).

\bibitem{pan2011time1}
\leavevmode\vrule height 2pt depth -1.6pt width 23pt, {\em Time maps and exact
  multiplicity results for one-dimensional prescribed mean curvature
  equations.}, Nonlinear Anal., 74 (2011), pp.~1234--1260.

\bibitem{pan2011time2}
\leavevmode\vrule height 2pt depth -1.6pt width 23pt, {\em Time maps and exact
  multiplicity results for one-dimensional prescribed mean curvature equations.
  ii}, Nonlinear Anal., 74 (2011), pp.~3751---3768.

\bibitem{pelesko2003modeling}
{\sc J.~A. Pelesko and D.~H. Bernstein}, {\em Modeling Mems and Nems}, Chapman
  \& Hall/CRC, Boca Raton, FL, 2003.

\end{thebibliography}

\end{document}